\documentclass[a4paper]{amsart}
\usepackage[margin=3.5cm]{geometry}

\usepackage{amsmath}
\usepackage{bbm}
\usepackage{latexsym}
\usepackage[english]{babel}
\usepackage{amsfonts}
\usepackage[T1]{fontenc}
\usepackage[utf8]{inputenc}
\usepackage{amsthm}
\usepackage{amssymb}
\usepackage{dsfont}
\usepackage{version}
\usepackage{caption}
\usepackage{extarrows}
\usepackage[backend=biber,style=alphabetic,maxbibnames=99]{biblatex}
\usepackage[psamsfonts]{eucal}
\usepackage{stmaryrd}
\usepackage{url}
\usepackage{hyperref}
\usepackage{color}
\usepackage{graphicx}
\usepackage{blindtext}
\usepackage{tikz}
\usetikzlibrary{arrows,matrix,positioning}
\tikzstyle{dmatrix}=[matrix of math nodes,row sep=2.5em, column sep=2.5em,
text height=1.5ex, text depth=0.25ex] 
\usepackage{mathtools}
\usepackage{microtype}
\usepackage{graphicx}
\usepackage[all]{xy}
\usepackage{lipsum}
\usepackage{float}
\usepackage{csquotes}
\usepackage{empheq}
\usepackage{enumitem}
\usepackage{todonotes}
\DeclareSymbolFont{epsilon}{OML}{ntxmi}{m}{it}
\DeclareMathSymbol{\epsilon}{\mathord}{epsilon}{"0F}

\theoremstyle{plain}
\newtheorem{theorem}{Theorem}[section]
\newtheorem{lemma}[theorem]{Lemma}
\newtheorem{prop}[theorem]{Proposition}
\newtheorem{cor}[theorem]{Corollary}

\newtheorem{prop/Def}[theorem]{Proposition/Definition}
\newtheorem{theorem/Def}[theorem]{Theorem/Definition}
\theoremstyle{definition}

\newtheorem{Def}[theorem]{Definition}
\newtheorem{rem}[theorem]{Remark}
\newtheorem{exa}[theorem]{Example}

\newtheorem{theointro}{Theorem}

\newtheorem*{the:toroidalmixed}{Theorem \ref{mixed_integrability}}
\newtheorem*{the:differencedegrees}{Corollary \ref{difference_degrees}}

\def \Q {{\mathbb Q}}
\def \R {{\mathbb R}}
\def \N {{\mathbb N}}
\def \Z {{\mathbb Z}}

\def \D {{\pmb{D}}}
\def \F {{\pmb{F}}}
\def \S {{\mathbb S}}
\def \O {{\mathcal O}}

\def \P {{\mathbb P}}

\def \T {{\mathbb T}}

\def \X {{ X_{\Sigma}}}
\def \div {{\operatorname{div}}}

\def \vol {{ \operatorname{vol}}}

\def \det {{ \operatorname{det}}}
\def \relint {{ \operatorname{relint}}}

\def \Hom {{ \operatorname{Hom}}}
\def \Star {{ \operatorname{Star}}}
\def \Ca {{ \operatorname{Ca}}}

\def \Div {{ \operatorname{Div}}}

\def \CabDiv {{ \operatorname{CbDiv}}}
\def \WbDiv {{ \operatorname{bDiv}}}

\def \Span {{ \operatorname{Span}}}
\def \Spec {{ \operatorname{Spec}}}

\def \trop {{ \operatorname{trop}}}

\def \codim {{ \operatorname{codim}}}

\def \snc {{ \operatorname{snc}}}

\def \timess {{ \operatorname{-times}}}

\def \deg {{ \operatorname{deg}}}

\def \Nef {{ \operatorname{Nef}}}

\def \supp {{ \operatorname{supp}}}

\def \ord {{ \operatorname{ord}}}

\def \Peff {{ \operatorname{Peff}}}
\def \Nef {{ \operatorname{Nef}}}

\def \PL {{ \operatorname{PL}}}
\def \Conic {{ \operatorname{Conic}}}

\def \sm {\textrm{sm}}
\def \sp {\textrm{sp}}
\def \seq{\textrm{seq}}

\renewbibmacro{in:}{}

\usepackage{tgpagella}
\usepackage{tgheros}
\usepackage{eulervm}
\usepackage{tikz}
\usepackage{etoolbox}
\usetikzlibrary{arrows,matrix,positioning}
\tikzstyle{dmatrix}=[matrix of math nodes,row sep=2.5em, column sep=2.5em,
text height=1.5ex, text depth=0.25ex] 

\numberwithin{equation}{section}

\bibliography{bibliography}

\title{Toroidal b-divisors and Monge--Amp\`ere measures}

\author[Botero]{Ana Mar\'ia Botero}
\address{Fachbereich Mathematik, Universit\"at
  Regensburg. Universit\"atsstr. 31, 93053 Regensburg, Deutschland}
\email{Ana.Botero@mathematik.uni-regensburg.de}
\urladdr{\url{https://homepages.uni-regensburg.de/~boa15169/}}

\author[Burgos Gil]{Jos\'e Ignacio Burgos Gil}
\address{Instituto de Ciencias Matem\'aticas (CSIC-UAM-UCM-UCM3).
  Calle Nicol\'as Ca\-bre\-ra~15, Campus UAB, Cantoblanco, 28049 Madrid,
  Spain} 
\email{burgos@icmat.es}
\urladdr{\url{http://www.icmat.es/miembros/burgos}}

\thanks{Botero was partially supported by the SFB Higher Invariants at
  the University of Regensburg. Burgos was partially supported by the
  MINECO research project MTM2016-79400-P and by the Severo Ochoa
  program for centers of excellence SEV-2015-0554 (ICMAT Severo
  Ochoa).}

\date{\today} \subjclass[2010]{Primary 14C17; Secondary 14T05, 26B25, 52B70.}
\keywords{b-divisors, convex analysis, polyhedral spaces, tropical
  geometry, Monge--Amp\`ere measures.}

\begin{document}

\begin{abstract}
We generalize the intersection theory of nef toric (Weil) b-divisors on
smooth and complete toric varieties to the
case of smooth and complete toroidal embeddings. As a key ingredient
we show the
existence of a limit measure, supported on the weakly embedded
rational conical
polyhedral space attached to the toroidal embedding, which arises as a limit of
discrete measures defined via tropical intersection theory on the
polyhedral space. We prove that the intersection theory of nef toroidal
Cartier b-divisors can be extended continuously to nef toroidal Weil
b-divisors
and that their degree can be computed as an integral with
respect to this limit measure. As an application, we show that a
Hilbert--Samuel type formula holds for big and nef toroidal Weil b-divisors.  
\end{abstract}

\maketitle
\tableofcontents

\section*{Introduction}
The theory of b-divisors (where b stands for ``birational'') was
introduced by Shokurov  \cite{Shokurov:Pf} in the context of Mori's
minimal model program. Since then b-divisors have appeared in many
contexts. For instance in the work by Fujino \cite{Fujino:basepoint}
on base point free theorems, in the work of K\"uronya and Maclean
\cite{KuronyaMaclean:Zariski}
on the Zariski decomposition of divisors, in the proof of the
differentiability 
of the volumes of divisors by Boucksom, Favre and Jonsson
\cite{BFJ:diffvol}, and in the work of Aluffi on Chow groups of
Riemann--Zariski spaces \cite{Aluffi} as Cartier b-divisors can be
understood as Cartier divisors on a Riemann--Zariski space. In
\cite{KKCar}, Kaveh and Khovanskii give an isomorphism between the
group of Cartier b-divisors and the Groethendieck group associated to
the semigroup of subspaces of rational functions preserving the top
intersection index, and although stated in another language, in
\cite{FS}, Fulton and Sturmfels give an isomorphism between the
Cartier b-Chow group of a toric variety and the polytope algebra which
preserves the intersection product.

Moreover, b-divisors have been associated to 
dynamical systems in \cite{BFJ:growth_meromorphic} and to
psh functions in \cite{BFJ:valuations}. In the last paper,
b-divisors whose support is a single point are studied. In
particular, a top intersection product, that can be $-\infty$, among
(relatively) nef b-divisors  
is defined and it is proved that such top intersection products can be
computed by means of a Monge--Amp\`ere type measure in a valuation
space. In \cite{BFF}, this top intersection product is generalized from the smooth case to case of isolated singularities. In the paper \cite{BKK}, a b-divisor is associated to the
invariant metric on the line bundle of Jacobi forms, and it
is shown that such a b-divisor is integrable, in the sense that its
top self intersection product is well defined and finite. Moreover, it is
proved that, considering this b-divisor, one recovers a Chern--Weil
formula, that says that the top self intersection product of the
b-divisor is computed as the 
integral on an open subset of a power of the first Chern form of the
metrized line
bundle of Jacobi forms, and a Hilbert--Samuel 
formula that states that the asymptotic growth of the
dimension of the space of Jacobi forms is governed by the top
self intersection product of the associated b-divisor. In addition it is
shown that in this case the associated b-divisor is toroidal and
that its top self intersection product can be computed using toric methods.

It is expected that the results of \cite{BKK} can be extended to
the invariant metrics on automorphic line bundles on  mixed Shimura
varieties, i.e. that the 
associated b-divisors are toroidal and their degrees computable
using toric techniques. In this spirit, in \cite{botero}, the first
author studied the theory of toric b-divisors on toric varieties and
showed that much of the theory of ordinary divisors on toric varieties
can be extended to the setting of b-divisors.

The aim of the present paper is to study toroidal b-divisors. In
particular, we generalize two results of \cite{BFJ:valuations} from the
local case to the global toroidal case. Namely, that there is a well
defined top intersection product between (global i.e. not only supported on a single point) nef
b-divisors and that this top intersection product is given by a
Monge--Amp\`ere type measure on a rational polyhedral complex. We
moreover prove a Hilbert--Samuel formula for nef and big toroidal
b-divisors and a Brunn--Minkowski type inequality.

We have chosen to restrict ourselves to toroidal b-divisors because
they appear naturally in the applications to mixed Shimura varieties
and they are technically simpler than arbitrary
b-divisors. 

After the present paper was completed the preprint \cite{DF:bdiv}
appeared. In this preprint a general theory of intersection of nef
b-divisors over countable fields is developed. The present article gives a concrete convex geometrical description of nef b-divisors and their top intersection numbers in the toroidal case. 

Toroidal b-divisors come in two flavors: Cartier and Weil. We explain this briefly.
Let $U \hookrightarrow X$ be a fixed smooth and complete toroidal
embedding without self-intersections of 
dimension $n$ with associated (balanced, weakly embedded, smooth) conical rational complex $\Pi$ (Proposition~\ref{prop:complex-toroidal}
and Definition~\ref{def:weak-toroidal}). This is the usual cone complex
associated to a simple normal crossings divisor $X\setminus U$ together with a weak
embedding and a balancing condition. Let $|\Pi|$ denote the support of $\Pi$ which carries a structure of a conical rational polyhedral space (Definition \ref{def:rat-space}). Let $R_{\sm}(\Pi)$ be the
directed system of all rational conical subdivisions of $\Pi$. This
is a directed set under the relation  
\[
\Pi'' \geq \Pi' \quad \text{ iff }\quad \Pi'' \text{ is a smooth
  subdivision of }\Pi'.
\]
An element $\Pi'$ in $R_{\sm}(\Pi)$ corresponds to a smooth and
complete toroidal embedding $U \hookrightarrow X_{\Pi'}$ together with a proper toroidal birational morphism $X_{\Pi'} \to X$ (Theorem
\ref{the:allowablemodification}). 

The \emph{toroidal Riemann--Zariski space} of the toroidal embedding
$(U,X)$ is defined formally as the inverse limit in the category of
locally ringed spaces
\[
\mathfrak{X}_{U} \coloneqq \varprojlim_{\Pi' \in
  R_{\sm}\left(\Pi\right)}X_{\Pi'}
\] 
with maps given by the proper toroidal birational morphisms $\pi \colon  X_{\Pi''}
\to X_{\Pi'}$ induced whenever $\Pi'' \geq \Pi'$. Then toroidal
b-divisors can be viewed as toroidal divisors on
$\mathfrak{X}_{U}$. Explicitly, for $\Pi' \in R_{\sm}(\Pi)$, we denote
by $\Div\left(X_{\Pi'},U\right)_{\R}$ the $\R$-vector space of toroidal $\R$-divisors on
$X_{\Pi'}$, i.e.~the group of divisors on $X_{\Pi'}$ supported on the
boundary
$X_{\Pi'} \setminus U$ with real coefficients. For $\Pi'' \geq \Pi'$
there are linear maps
\begin{displaymath}
  \xymatrix{
    \Div\left(X_{\Pi'},U\right)_{\R} \ar@/^/[r]^{\pi^{\ast}}& \Div\left(X_{\Pi''},U\right)_{\R}.
    \ar@/^/[l]^{\pi_{\ast}}
  }
\end{displaymath}
Then the spaces of Weil and Cartier toroidal b-divisors are defined as
the projective and injective limits
\begin{align*}
  \WbDiv(\mathfrak{X}_{U})_{\R} &\coloneqq \varprojlim_{\Pi' \in
  R_{\sm}\left(\Pi\right)}\Div\left(X_{\Pi'},U\right)_{\R},\\
  \CabDiv(\mathfrak{X}_{U})_{\R} &\coloneqq \varinjlim_{\Pi' \in
  R_{\sm}\left(\Pi\right)}\Div\left(X_{\Pi'},U\right)_{\R},
\end{align*}
respectively, with maps given by proper push-forward of divisors in the first case
and pull-back in the second
(Definition~\ref{def:toroidalbdivisor}). In other words, a Weil
toroidal b-divisor is given by a net
\[
\D = \left(D_{\Pi'}\right)_{\Pi' \in R_{\sm}\left(\Pi\right)},
\]
where for each $\Pi' \in R_{\sm}\left(\Pi\right)$, the element
$D_{\Pi'}$ is a
toroidal $\R$-divisor on $X_{\Pi'}$, and all these elements are
compatible 
under push-forward. For $\Pi' \in R(\Pi)$, we say that $D_{\Pi'}$
is the incarnation of $\D$ on $X_{\Pi'}$. On the other hand, a Cartier
toroidal b-divisor is determined by a single $\Pi' \in R(\Pi)$ and a
divisor $D_{\Pi '}\in \Div(X_{\Pi '},U)_{\R}$.
There is a natural inclusion
$\CabDiv(\mathfrak X)_{\R} \subseteq \WbDiv(\mathfrak X)_{\R}$.  Roughly
speaking, Cartier b-divisors are b-divisors that stabilize after a
birational map, while Weil b-divisors may keep changing for all blow
ups.  

To simplify notation, we will usually omit the coefficient ring $\R$ from the notation, real
coefficients being always implicit. 

As in toric geometry, for
any $\Pi' \in R_{\sm}(\Pi)$, we may view toroidal divisors on $X_{\Pi'}$
as piecewise linear functions on $|\Pi|$ which are linear on
each cone of $\Pi '$. Hence, it is easy
to see that a Cartier toroidal b-divisor corresponds to a real valued piecewise
linear
function whose locus of linearity is rational, while a
Weil toroidal b-divisor $\D$ corresponds to a (not necessarily piecewise linear) conical
function
\[
\phi_{\D}\colon \left|\Pi\right|(\Q) \longrightarrow \R,
\]
where $|\Pi |(\Q)$ denotes the set of rational points of $|\Pi|$.
Note that the only condition required from the function $\phi_{\D}$ is
that it is conical, which shows that the space of Weil b-divisors is
very wild. Nevertheless, it turns out that if we impose the nefness 
condition to $\D$
(Definition~\ref{def:nefbdiv}), then the function $\phi_{\D}$ extends
to a continuous (weakly concave) function   
\[
\phi_{\D}\colon \left|\Pi\right| \longrightarrow \R,
\]
whose restriction to each cone $\sigma \in \Pi$ is concave (see
Theorem \ref{the:nefextensionfunction}). 

Since Cartier toroidal b-divisors are determined on a concrete
birational model, the intersection theory of divisors gives
immediately an intersection theory of Cartier toroidal b-divisors. The
main result of this paper is that the intersection product of nef
Cartier toroidal b-divisors can be extended continuously to nef Weil
toroidal b-divisors and that this product can be computed as the
integral of a Monge--Amp\`ere type measure (Theorem~\ref{th:convergence-trop-meas} and Corollary \ref{cormixed} for the
combinatorial version and Theorem \ref{thm:intersectionbdivisors} for
the geometric version).  

\begin{theointro}\label{theointro} Let $(U,X)$ be a toroidal
  embedding with  $X$ a smooth and projective variety over a field of
  characteristic zero  and $X\setminus U$ the
  support of an effective snc ample divisor. Let $|\Pi|$ be the associated conical rational polyhedral space. Then the top
  intersection product of nef Cartier toroidal
  b-divisors can be extended continuously to a top intersection product
  of nef Weil toroidal b-divisors on $X$. Moreover, to a family
  $\D_{2},\dots,\D_{n}$ of nef Weil toroidal b-divisors on $X$ we associate a
  Monge--Amp\`ere type measure $\mu_{\D_2,\dotsc , \D_n},  $
  supported on a  compact subset $\mathbb{S}^{|\Pi|} \subseteq
  |\Pi|$, in such a way that  
  \[
    \D_{1}\dotsm \D_n =
    \int_{\mathbb{S}^{|\Pi|}}\phi_{\D_1}\mu_{\D_2,\dotsc , \D_n}. 
  \]
  for any nef Weil toroidal b-divisor $\D_{1}$ on $X$.
\end{theointro}

Note that, for varieties over a field of characteristic zero, the
conditions on $U$ and $X$ can always be achieved after shrinking $U$
and replacing $X$ by another birational variety. Note also that nowadays,
the existence of the product can also be deduced from \cite{DF:bdiv}. 

An important feature of the weakly embbeded conical space $|\Pi |$
is that it comes equipped with a balancing condition i.e. it is a
tropical cycle. This balancing condition plays an important role in
the definition of the Monge--Amp\`ere measure and in the strong
continuity poperties of concave functions on balanceable polyhedral
complexes \cite[Section~6]{BBS:caps}.   

As an application, following \cite[Section 5]{botero}, we define the
space of non zero global sections of a toroidal b-divisor $\D$ as the space of
rational functions $f$ such that $\operatorname{b-div}(f) + \D$ is effective. 
Then the volume of a Weil b-divisor is defined in
analogy to the volume of divisors by the asymptotic growth of the
space of global sections and a Weil b-divisor is called big if it has
positive volume. Moreover, to a Weil toroidal b-divisor $\D$ we
can associate an Okounkov body $\Delta_{\D}$ (Definition
\ref{def:okounkov}). Then we obtain the
following extension of the
Hilbert--Samuel theorem to the b-case (Theorem \ref{the:hilbertsamuel}).
\begin{theointro}\label{thm:HSintro}
Let $\D$ be a big and nef Weil toroidal b-divisor on $X$. Then
\[ 
\vol(\D) = n! \vol(\Delta_{\D}) = \D^n.
\]
\end{theointro}
In Theorem \ref{thm:HSintro}, the hypothesis $\D$
toroidal is necessary. In fact, in a forthcoming paper with R. de Jong and
D. Holmes we will show with an example that the volume function is not
continuous even for big and nef b-divisors defined over a countable
field and that it does not necessarily agree with the degree. 

As a corollary, we obtain the continuity of the volume function on the space of nef and big toroidal b-divisors (Corollary \ref{cor:cont}) and a Brunn--Minkowski type inequality (Corollary~\ref{cor:brunnminkowski}).
One of the key ingredients to prove the two stated theorems is the
combinatorial machinery developed in Sections \ref{measures-section}
and \ref{sec:monge-ampere-meas}. The existence of the limit measure
$\mu_{\D}$ associated to a nef toroidal b-divisor $\D$ follows
directly from Theorem \ref{th:convergence-trop-meas} and the existence
of the mixed limit measure $\mu_{\D_2,\dotsc ,\D_n}$ is a direct
consequence of Corollary \ref{cormixed}. These results are based on the
convex analysis on polyhedral spaces developed by M. Sombra and the
authors in \cite{BBS:caps}, that, in turn, use techniques from
\cite{BFJ} and \cite{MR3329027}.

Another key ingredient is a
result in \cite{GR} relating tropical and algebraic intersection
numbers on complete toroidal embeddings (Theorem
\ref{the:trop-alg-numbers}).

It is important to note that for the applications to algebraic geometry it is convenient to work
with conical complexes provided with an integral
structure. Nevertheless, when studying convex analysis on polyhedral complexes, the integral
structure plays no role, only the affine structure does.  Moreover, to
write down explicit estimates it is handy to choose a Euclidean
structure. Therefore, to study Monge--Amp\`ere measures associated to nef toroidal
b-divisors it is convenient to shift the focus from rational conical
complexes to Euclidean ones.

As has been noted in \cite{BKK} and \cite{botero}, a nef toroidal
b-divisor encodes the singularities of the invariant metric on an
automorphic line bundle over a mixed Shimura variety of non-compact
type along any toroidal compactification. This article together with the
above mentioned ones lays the ground of a geometric 
intersection theory with singular metrics,  satisfying Chern--Weil
theory and a Hilbert--Samuel formula, to be applied to
mixed Shimura varieties of non-compact type.

The article is organized as follows. In Section \ref{measures-section} we recall the tropical intersection theory on
Euclidean conical polyhedral spaces as in \cite{BBS:caps}. This
is a Euclidean version of the tropical intersection theory
on weakly embedded rational conical polyhedral complexes developed in
\cite{GR}. 

In Section \ref{sec:monge-ampere-meas} we show the combinatorial version of our main result stated in Theorem~\ref{theointro}. For this, we define the space of conical functions on Euclidean conical spaces and introduce a concavity notion for them. We show that the top intersection product of such concave functions exists, is finite and is given by the total mass of a week limit of discrete Monge--Amp\`ere measures. This is done by introducing the notion of the \emph{size} of a
tropical cycle. This allows us to prove a Chern--Levine--Nirenberg
type inequality (Lemma \ref{tropicalintersectioninequality}) from
which we conclude the weak convergence of the discrete measures
(Theorem \ref{th:convergence-trop-meas}). 

In Section \ref{sec:toroidal-embeddings} we give the definition of a quasi-embedded
\emph{rational} 
conical polyhedral space. In short a rational conical polyhedral space is
a conical polyhedral space endowed with a lattice structure. We recall the
definition of toroidal
embeddings and describe the rational conical polyhedral space
associated to one (see \cite{KKMD} or \cite{AMRT} for further
details). Following \cite{GR}, we also give a natural weak embedding
of this space. Moreover, we show that by adding boundary components
one can modify the toroidal structure of a toroidal embedding in such
a way that the rational conical polyhedral space becomes
quasi-embedded. Then we describe the proper toroidal birational
mosphisms of toroidal embeddings which, on the combinatorial
side, correspond to subdivisions of the corresponding rational conical
polyhedral complexes.

In Section \ref{sec:inters-theory-toro} we state and prove our main results. We show that nef
toroidal b-divisors have well defined top intersection products
(Definitions \ref{def:toroidalbdivisor}
and \ref{def:nefbdiv} and Theorem
\ref{thm:intersectionbdivisors}). For this, we first relate the geometric intersection product of toroidal
divisors with the \emph{rational tropical intersection product} on quasi-embedded
rational conical spaces (Theorem
\ref{the:trop-alg-numbers}) (see \cite{GR}). 
Then we use the
convergence results of Section \ref{sec:monge-ampere-meas} in order to
extend the top intersection product to nef toroidal b-divisors. However, note that the
Monge--Amp\`ere measures of Section \ref{sec:monge-ampere-meas} are
defined in a Euclidean setting (no integral structure). Therefore we
will use the comparison in Section \ref{sec:bridge} to relate the
rational tropical intersection product with the 
Euclidean one. 

Finally, in Section \ref{sec:applications}, as an application, we give a Hilbert--Samuel type formula for nef and big toroidal b-divisors. This relates the degree of a nef toroidal b-divisor both with the
volume of the b-divisor and with the volume of the associated convex Okounkov body (Definitions \ref{def:volume-b-div} and \ref{def:okounkov} and Theorem \ref{the:hilbertsamuel}). As a corollary, we obtain a Brunn--Minkowski type inequality (Corollary \ref{cor:brunnminkowski}).

As was mentioned already above, after the present paper was completed the preprint \cite{DF:bdiv}
appeared. In this preprint a general theory of intersection of nef
b-divisors over countable fields is developed. Since any toroidal
situation can be reduced to a situation defined over a countable
field, Theorem \ref{theointro} can be deduced from \cite[Theorem
6]{DF:bdiv}. We think that the present paper is still valuable as
the technique of proof is different and it gives a very concrete
interpretation of the intersection product in the toroidal case by
means of a Monge--Amp\`ere type measure.  It would be interesting to
know if, in general, the intersection product of nef b-divisors (over
a countable field) can also be interpreted in terms of a
Monge--Amp\`ere type measure.

\noindent\textbf{Acknowledgments} We would like to thank J\"urg
Kramer, Robin de
Jong, Walter
Gubler, David Holmes and Klaus Kuenneman for many
stimulating discussions. We want specially to thank Mart\'in Sombra for
helping us in proving the continuity properties of concave functions
on polyhedral spaces \cite{BBS:caps}. We also would like to thank the
anonymous referee for his or her careful reading of the manuscript as
well as many useful comments. 

Most of the work of this paper has been conducted while the
authors were visiting 
Humboldt University of Berlin, Regensburg University and the ICMAT. We
would like
to thank these institutions for their hospitality. 

\section{Euclidean tropical intersection theory}\label{measures-section}
In this section we recall the tropical intersection theory on
Euclidean conical polyhedral spaces as in \cite{BBS:caps}. This
is an adapted Euclidean version of the tropical intersection theory
on weakly embedded rational conical polyhedral complexes developed in
\cite{GR}.  


\subsection{Euclidean conical polyhedral spaces}

We give the definition of a quasi-embedded conical
polyhedral space endowed with a Euclidean structure. We also discuss 
morphisms and subdivisions of such spaces.

\begin{Def}\label{realconicalcomplex}
Let $X$ be a second countable topological space. A \emph{conical polyhedral structure} on $X$ is a pair
\[
\Pi = \left(\{\sigma^{\alpha}\}_{\alpha \in \Lambda}, \{M^{\alpha} \}_{\alpha \in \Lambda}\right)
\]
consisting of a finite covering by closed subsets $\sigma^{\alpha} \subseteq X$
and for each $\sigma^{\alpha}$, a finitely generated $\R$-vector space
$M^{\alpha}$ of continuous, $\R$-valued functions on $\sigma^{\alpha}$
satisfying the following conditions. Let $N^{\alpha} \coloneqq
\Hom(M^{\alpha}, \R)$ denote the dual vector space.  
\begin{enumerate}
\item  For each $\alpha \in \Lambda$, the evaluation map $\phi^{\alpha}\colon \sigma^{\alpha} \to N^{\alpha}$ given by the assignment
\[
v \longmapsto (u \mapsto u(v)) \quad (u \in M^{\alpha} ),
\]
maps $\sigma^{\alpha}$ homeomorphically to a strictly convex,
full-dimensional, polyhedral cone in $N^{\alpha}$. 
\item The preimage under $\phi^{\alpha}$ of each face of
  $\phi^{\alpha}\left(\sigma^{\alpha}\right)$ is a cone
  $\sigma^{\alpha'}$ for some index $\alpha' \in \Lambda$, and we have
  that $M^{\alpha'} = \left\{u|_{\sigma^{\alpha'}} \, \big{|} \, u \in
    M^{\alpha}\right\}$.
\item The intersection of two cones is a union of common faces.
\end{enumerate}
The $\R$ vector spaces $M^{\alpha}$ give $X$ a so called
\emph{linear structure}. 
\end{Def}
The following notations will be used.
\begin{enumerate}
\item  By abuse of notation we
will think of $\Pi $ as the set of cones 
$\{\sigma ^{\alpha }\}_{\alpha \in \Lambda }$. 
 For every integer
$k\ge 0$ we write $\Pi (k)$ for the set of cones of dimension $k$.
\item Given a cone $\sigma \in \Pi$, we will write $M^{\sigma}$,
$N^{\sigma}$ and $\phi^{\sigma}$ for the corresponding $\R$-vector space, dual
vector space and evaluation map, respectively.
  We denote by $\langle\ ,\ \rangle_{\sigma }$ the pairing induced by the dual
vector spaces $M^{\sigma }$ and $N^{\sigma }$. We will usually omit the index
\enquote{$\sigma $} from the pairing.
\item We will identify a cone $\sigma $ with its image in
  $N^{\sigma }$.  The linear structure of $N^{\sigma }$ induces a 
  linear structure in $\sigma $. Therefore
  we can talk of linear maps between cones. 
\item  If $\tau $ is a face of $\sigma $ we will write $\tau \prec
  \sigma $ or $\sigma \succ \tau $.
\item We will denote by $0_{\sigma }$ the zero for the linear
  structure of $N_{\sigma }$. Since $\sigma $ is strictly convex, the
  set $\{0_{\sigma }\}$ is a face of $\sigma $. By abuse of notation
  we will denote this face also as $0_{\sigma }$.  
\item  By the
  \emph{relative interior} of a cone $\sigma$, denoted
  $\relint (\sigma ) $ we mean the preimage
  under $\phi^{\sigma}$ of the interior of the
  cone $\phi^{\sigma}\left(\sigma\right) \subseteq
  N^{\sigma}$.

\item $\Pi$ is called \emph{simplicial} if every cone
$\phi^{\sigma}(\sigma)$ is generated by an $\R$-basis of
$N^{\sigma}$. 
\end{enumerate}

\begin{Def} Let $X$ be a second countable topological space and $\Pi, \Pi'$ two conical polyhedral structures on $X$. Then $\Pi'$ is a \emph{subdivision} of $\Pi$, denoted by $\Pi' \geq \Pi$, if for every $\sigma' \in \Pi'$ there exists a $\sigma \in \Pi$ with $\sigma' \subset \sigma$, the inclusion being a linear map. Two conical structures on $X$ are \emph{equivalent} if they admit a common subdivision.
\end{Def}
The following follows as in \cite[Proposition 2.4]{BBS:caps}.
\begin{prop}\label{prop:equi}
Let $X$ be a second countable topological space. Then 
\begin{enumerate}
\item the relation $\geq$ is a partial order on the set of conical polyhedral structures on $X$,
\item the subdivisions of a given conical polyhedral structure on $X$ form a directed set,
\item ``being equivalent" is an equivalence relation between conical polyhedral structures on $X$. 
\end{enumerate}
\end{prop}

\begin{Def}
A \emph{conical polyhedral space} $X$ is a second countable topological space equipped with an equivalence class of conical polyhedral structures. A \emph{conical polyhedral complex on} $X$ is the choice of a representative of the class of conical polyhedral structures on $X$. 
\end{Def}
\begin{Def}
If $X$ is a conical polyhedral space, we let $R(X)$ denote the set of conical polyhedral complexes on it. 
By Proposition \ref{prop:equi}, this is a directed set ordered by subdivision. 
\end{Def}
We will usually refer to a \emph{conical polyhedral space} just as a \emph{conical space} and to a \emph{conical polyhedral complex} just as a \emph{conical complex}.
\begin{Def}
The \emph{dimension} of a conical space $X$ is defined as as 
\[
\dim(X) = \sup_{\sigma \in \Pi}\dim(M_{\sigma})
\]
for any conical complex $\Pi$ on $X$. We say that $X$ has \emph{pure dimension $n$} if every cone of $\Pi$ that is maximal (with respect to the inclusion) has dimension $n$. These notions do not depend on the choice of $\Pi$.
\end{Def}

The following remark follows from \cite[Remark 2.6]{Payne}.
\begin{rem}\label{rem:connected}
Let $X$ be a conical space and let $\Pi$ be any conical complex on $X$. 
  The connected components of $X$ are in one to one
  correspondence with the zero dimensional cones of $\Pi $. The points
  belonging to the zero dimensional cones are called \emph{vertices}
  of $\Pi $. In
  particular, if $X$ is connected, then $\Pi$ has a unique vertex.
\end{rem}
\begin{Def} Let $X$ and $X'$ be conical spaces. Given conical complexes $\Pi$ on $X$ and $\Pi'$ on $X'$, a \emph{morphism of conical spaces between $\Pi$ and $\Pi'$} is a continuous map $f \colon X \to X'$ such that for every cone $\sigma \in \Pi$ there is $\sigma' \in \Pi'$ with $f(\sigma) \subseteq \sigma'$, and the restriction $f|_{\sigma} \colon \sigma \to \sigma'$ is a linear map. 
\end{Def} 

\begin{rem}\label{rem:nosinglelattice} Note that in contrast to the
  classical definition of a \emph{fan}, a conical space $X$ does not
  come equipped with the choice of 
  a particular embedding into a single vector space.  In order to
  define an intersection theory on $X$, in the case of a fan, it
  suffices to ask that is complete. In general, as we will see, we
  must be able to visualize a conical complex $\Pi$ on $X$ inside a
  single vector space in such a way that it is balanced. 
\end{rem} 

The following definition is the Euclidean version of \cite[Definition 2.1]{GR}.

\begin{Def}\label{weakembedding}
A \emph{weakly embedded conical space} is a triple $(X, N, \iota)$ where $X$ is a conical space, $N$ is a finite dimensional $\R$-vector space, and $\iota \colon X \to N$ is a map such that there is a conical complex $\Pi$ on $X$ for which the restriction of $\iota$ to every cone $\sigma$ of $\Pi$ is linear. The map $\iota$ is called the weak embedding of $X$ in $N$. A \emph{conical complex} $\Pi$ on the weakly embedded conical space $(X, N, \iota)$ is a conical complex $\Pi$ on $X$ satisfying the above condition, namely that $\iota$ is linear on each of its cones. 
\end{Def}
We will usually denote a
weakly embedded conical space by the underlying conical space
$X$ and, in this case, we denote the corresponding weak embedding, vector space and dual vector space by $\iota_X$ and $N^X$ and $M^X$, respectively. 
Given a conical complex $\Pi$ on the weakly embedded conical space $X$, for
every cone $\sigma \in \Pi$, we write $N^{X}_{\sigma}$ for $N^{X}
\cap \text{Span}\left(\iota_{X}(\sigma)\right)$ and
$M_{\sigma}^{X} = \Hom\left(N_{\sigma}^{X}, \R\right)$ for its
dual.

The following notion is stronger than that of a weakly embedded conical
space.

\begin{Def}
  A weakly embedded conical space $X$ is said to be \emph{quasi-embedded} if there is a conical complex $\Pi$ on $X$ such that the restriction $\iota_{X}|_{\sigma}$ is injective 
  in each cone $\sigma \in \Pi$ .  In this case, we
  identify each vector space $N^{\sigma }$ with its image $N^{X }_{\sigma}
    $ in $N^{X}$. As before, a \emph{conical complex} $\Pi$ on the quasi-embedded conical space $X$ is a conical complex $\Pi$ on $X$ satisfying the above condition, namely that $\iota$ is linear and injective on each of its cones. 
  \end{Def}

The following is a useful property of quasi-embedded conical
spaces that is not true in general for weakly embedded ones.

\begin{lemma}\label{cpproper}
  Let $X$ be a quasi-embedded
 conical space. Then the map $\iota _{X}\colon X\to
  N^{X}$ is proper. 
\end{lemma}
\begin{proof}
  Let $\Pi$ be a conical complex on $X$. Since $X$ is a finite union of closed cones $\sigma \in \Pi$, it is
  enough to show that $\iota _{X}|_{\sigma }$ is proper.  Since
  $X$ is quasi-embedded, we have that the map $\iota _{X}|_{\sigma } \circ {\phi ^{\sigma }}^{-1}\colon N^{\sigma }\to
  N^{X }$ is an injective linear map, hence proper. By
  definition, the map $\phi ^{\sigma }\colon \sigma \to N^{\sigma
  }$ is proper. Hence $\iota _{X}|_{\sigma }=\iota _{X}|_{\sigma } \circ {\phi ^{\sigma }}^{-1}\circ \phi^{\sigma }$ is proper. 
\end{proof}

\begin{Def}
  A \emph{Euclidean conical space} is a quasi-embedded conical
  space $X$ together with a Euclidean product on the real vector
  space $N^{X}$. Given any conical complex $\Pi$ on $X$, this Euclidean structure induces compatible
  Euclidean structures in each vector space $N^{\sigma}$ for $\sigma \in \Pi$.
\end{Def}

\begin{Def}\label{def:morphism-complex} 

  A \emph{morphism of weakly embedded conical spaces}
  consists of a morphism of conical spaces
  $f \colon X \to X'$ together with a morphism of finite-dimensional $\R$-vector spaces $f' \colon N^{X}
  \to N^{X'}$ forming a commutative square with the weak embeddings.
  
  \begin{center}
    \begin{tikzpicture}
      \matrix[dmatrix] (m)
      {
   X &  X'\\ N^{X} & N^{X'} \\
      };
      \draw[->] (m-1-1) to node[above]{ $f$ } (m-1-2);
      \draw[->] (m-1-1) to node[left]{$\iota_{X}$}(m-2-1);
     \draw[->] (m-2-1) to node[above]{ $f'$ } (m-2-2); 
     \draw[->] (m-1-2) to node[right]{$\iota_{X'}$}  (m-2-2);
     \end{tikzpicture}
   \end{center}
  
A \emph{morphism of quasi-embedded or of Euclidean conical spaces} is a morphism of weakly embedded conical spaces.
  
\end{Def}

\subsection{The Euclidean tropical intersection product}\label{tropiintersectionproduct}
Throughout this section $X$ will denote a Euclidean
conical space of pure dimension $n$ with quasi-embedding given by
$\iota_{X} \colon X \to N^{X}$. The goal of this section
is to define the Euclidean tropical intersection product between tropical cycles and piecewise linear functions on $X$. This is a
Euclidean version of the tropical intersection product given in
\cite{GR}. For the interested reader, the articles \cite{AR},
\cite{FS} and \cite{KA} constitute a more thorough reference for
tropical intersection theory on globally embedded conical polyhedral complexes
with an integral structure.

We start with some definitions. These are the Euclidean adaptations of \cite[Section~3.1]{GR}.

\begin{Def} \label{def:vectorst}
  Let $k\ge 1$ be an integer and let
  $\tau\in \Pi (k-1)$ be a cone. For every cone $\sigma \in \Pi(k)$
  with $\tau \prec \sigma $ we define the \emph{Euclidean normal vector
    $\hat v_{\sigma/\tau}$ of $\sigma$ relative to $\tau$}  to be the
  unique unitary vector of $N^{\sigma }$ that is orthogonal to
  $N^{\tau }$ and points in the direction of
  $\sigma $.  By abuse of notation $\hat v_{\sigma /\tau }$ will also
  denote its image in $N^{X }$. If $k = 1$, we write
  $\hat{v}_{\sigma}\coloneqq \hat{v}_{\sigma / \{0_{\sigma }\}}.$ 
\end{Def}

Recall that, given a conical complex $\Pi$ on $X$, we are denoting by $\Pi $ both the complex and its set of cones. 
\begin{Def}\label{MW} A weight on $\Pi $ is a map
  \begin{displaymath}
    c\colon \Pi \longrightarrow \R.
  \end{displaymath}
  It is called a $k$-dimensional weight if $c(\sigma )=0$ for all
  $\sigma \not \in \Pi (k)$. A $k$-dimensional weight is called a
  \emph{$k$-dimensional Euclidean Minkowski weight} on $\Pi$ if, for
  every cone $\tau \in \Pi(k-1)$, the Euclidean balancing condition 
\begin{equation}\label{eq:1}
  \sum_{\substack{\sigma \in \Pi (k)\\\tau \prec \sigma
    }}c\left(\sigma\right)\hat v_{\sigma/\tau} =0   
\end{equation}
holds true in $N^{X }$. 
\end{Def}
The set of weights on $\Pi$ is a real graded vector space denoted by
$W_{\ast}(\Pi
)$. The $k$-dimensional Euclidean Minkowski
weights form an abelian group, which is denoted by
$M_k(\Pi)$.


We can now define the pull-back of a Minkowski weight along a
subdivision.

\begin{Def}\label{def:1}
  Let $\Pi' $ be a subdivision of $\Pi $ with its induced structure of
  Euclidean conical complex and denote by $f\colon \Pi '\to \Pi $
  the corresponding morphism of Euclidean conical complexes. Let $c\in
  M_k(\Pi)$ be a Euclidean Minkowski weight. Then the pull-back of $c$ by
  $f$ is the Euclidean Minkowski weight
  \begin{displaymath}
    f^{\ast}(c)(\sigma ') = \begin{cases} c(\sigma) \quad &\text{if }
      \dim \sigma = \dim \sigma', \\ 0 \quad
      &\text{otherwise,} \end{cases}
  \end{displaymath}
  where $\sigma '\in \Pi '(k)$ and $\sigma $ is the minimal cone of
  $\Pi $ that contains $\sigma '$. This construction defines a
  group homomorphism
  \begin{displaymath}
    M_{k}(\Pi )\longrightarrow M_{k}(\Pi ').
  \end{displaymath}
  The fact that if $c$ is a Minkowski weight then $f^{\ast}(c)$ is
  also a Minkowski weight can be
  argued as in \cite[Example 2.11~(4)]{GKM09:tropic_moduli}.
\end{Def}

More generally, \emph{Euclidean tropical cycles} on the Euclidean conical space $X$ are defined as direct limits of Euclidean Minkowski weights over all
conical structures on $X$. Recall that $R(X)$ denotes the set of all conical complexes on $X$. It has the structure of a directed set ordered by inclusion. 

\begin{Def}\label{tropcycle} 
The group of \emph{Euclidean tropical $k$-cycles} on $X$ is defined as the direct limit
\[
Z_k(X) \coloneqq \varinjlim_{\Pi' \in R(X)}M_k\left(\Pi'\right),
\]
with maps given by the pull-back maps of Definition \ref{def:1}.
If $c$ is a $k$-dimensional Euclidean Minkowski weight on a Euclidean conical complex $\Pi$
in $R(X)$, we denote by $[c]$ its image in $Z_k(X)$. 
\end{Def}

\begin{Def}
  The \emph{degree} of a zero cycle $[c]\in Z_{0}(X )$ determined on $\Pi$, is defined as
  \begin{displaymath}
    \deg([c])=\deg(c)=\sum_{v\in \Pi (0)}c(v).
  \end{displaymath}
\end{Def}

\begin{Def}\label{def:positivecycle}
A Euclidean Minkowski weight is called \emph{positive} if it has
non-negative values.
A Euclidean tropical cycle is called \emph{positive} if it is
represented by a
positive Euclidean Minkowski weight. The sub-semigroups of positive
Euclidean Minkowski weights on $\Pi$ and of 
positive Euclidean tropical cycles of dimension $k$ on $X$ are denoted by
$M^+_k(\Pi)$ 
and by $Z^+_k(X)$, respectively.
\end{Def}
\begin{rem}
It would seem natural to call positive tropical cycles \emph{effective} in order to mimic the usual terminology for algebraic cycles. This would be however misleading since, as we will see later in Section \ref{sec:posit-prop-cycl}, it is not true that the tropicalization of effective cycles are positive.
\end{rem}


We now define the objects where we want to compute top intersection numbers, namely \emph{balanced Euclidean conical spaces}.
\begin{Def}\label{def:balanced-complex}
  The Euclidean conical space $X$ is said to
  be \emph{balanced} if it is provided with an $n$-dimensional Euclidean
  tropical cycle 
  $[X]\in Z_{n}(X )$ represented by a Minkowski weight $b \in
  M_n(\Pi)$, for some conical complex $\Pi$ in $R(X)$, satisfying 
  \[
    b(\sigma )>0, \ \forall \sigma \in \Pi (n).
  \]
  In this case, we say that the conical complex $\Pi$ is \emph{balanced}.
\end{Def}

\begin{rem}
  If $X$ is a balanced Euclidean conical space, then its is in particular a balanced Euclidean polyhedral space in the sense of \cite[Definition
  3.27]{BBS:caps}. Thus we have at our disposal the theory of concave
  functions on polyhedral spaces developed in that paper. See section
  \ref{sec:weakly-conc-funct} for a short recap. 
\end{rem}


We now define piecewise linear functions on the conical space $X$. 
 
 \begin{Def} A \emph{piecewise linear function} on $X$ is a function  $\phi \colon X \to \R$ for which there is a conical complex $\Pi$ on $X$ such that on each cone $\sigma \in \Pi$, the restriction $\phi|_{\sigma}$ is linear or equivalently, it is given by an element of $M_{\sigma}$. In this situation, we say that $\phi$ is \emph{defined} on $\Pi$.
 \end{Def}
 If $\phi $ is defined on $\Pi $, for each $\sigma \in \Pi$, we denote
 by $\phi_{\sigma}$ a linear function on $N^X$ satisfying  
 \[
 \phi|_{\sigma} = \phi_{\sigma} \circ \iota_{X,\sigma}.
 \]
  By abuse of notation
  $\phi _{\sigma }$ will also denote this linear function restricted to
  $N^{\sigma }$.
  \begin{Def}
  We denote by $\PL(X)$ the real vector space of piecewise linear
  functions on $X$, and for any conical complex $\Pi \in R(X)$ we
  denote by $\PL(\Pi)$ the subgroup of those piecewise linear
  functions that are defined on $\Pi$.
\end{Def}

\begin{rem} 
Let $\Pi' \geq \Pi$ in $R(X)$ and $f \colon \Pi' \to \Pi$ the
corresponding morphism of conical complexes. There is a natural
pull-back map $\PL(\Pi) \to \PL(\Pi')$ given by $\phi \mapsto
f^{\ast}\phi \coloneqq \phi
\circ f$. The vector space of piecewise linear functions $\PL(X)$ can be seen
as the direct limit  in the category of real vector spaces
\[
\PL(X) = \varinjlim_{\Pi \in R(X)} \PL(\Pi)
\]
with respect to these pull-back morphisms.
\end{rem}

A piecewise linear function on $X$ is always continuous because the
restrictions to the components of a finite closed covering are
continuous.  Moreover it is conical in the sense that, for all real
$\lambda >0$, $\phi (\lambda x)=\lambda \phi (x)$.

For any $\Pi \in R(X)$ and $k\ge 1$ we now construct a \emph{Euclidean tropical intersection product}
\begin{displaymath}
  \PL(\Pi )\times M_{k}(\Pi )\longrightarrow M_{k-1}(\Pi ). 
\end{displaymath}
\begin{Def} \label{def:4}
  Let $\Pi \in R(X)$ be a Euclidean conical complex on $X$, $\phi $ a piecewise linear function defined on $\Pi$
  and $c\in M_{k}(\Pi )$ a $k$-dimensional Euclidean Minkowski weight. Then the \emph{Euclidean tropical intersection product} is the $(k-1)$-dimensional Minkowski
  weight $\phi \cdot c \colon M_{k-1}(\Pi) \to \R$ given by 
  \begin{displaymath}
    \phi \cdot c(\tau )\coloneqq \sum_{\substack{\sigma \in
    \Pi(k)\\\sigma \succ
    \tau}}-\phi_{\sigma}\left(\hat v_{\sigma/\tau}\right)c(\sigma).
  \end{displaymath}
\end{Def}
The fact that $\phi \cdot c$ is indeed a Euclidean Minkowski weight
is proven in \cite[Proposition~3.19]{BBS:caps} adapting the standard
proof in tropical geometry.

We now see that this intersection product extends to an intersection
product between $\PL(X)$ and Euclidean tropical cycles. To this end we
see that the Euclidean tropical intersection product is compatible
with the restriction to
subdivisions.


\begin{lemma}\label{lemm:1}
  Let $\Pi ' \geq \Pi $ in $R(X)$ and $f \colon \Pi' \to
  \Pi$ the corresponding morphism of conical
  complexes. Let $\phi $ be a piecewise linear function defined on $\Pi $ and $c\in
  M_{k}(\Pi )$ a
  Minkowski weight. Then
  \begin{displaymath}
    f^{\ast}(\phi \cdot c)=f^{\ast }\phi \cdot f^{\ast}c.
  \end{displaymath}
\end{lemma}
\begin{proof}
  This is proved in \cite[Proposition 3.12]{BBS:caps}.
\end{proof}

We can now define the intersection product between $\PL(X)$
and Euclidean tropical cycles.  

\begin{Def}
Let $\phi$ be a piecewise linear function on $X$ and let $[c] \in
Z_k(X)$ be a Euclidean tropical cycle. Let $\Pi \in R(X)$ be any
conical complex on $X$ such that $\phi$ is determined on $\Pi$ and
such that $[c]$ is represented by $c \in M_k(\Pi)$. Then the bilinear
pairing 
\[
PA(X) \times Z_k(X) \longrightarrow Z_{k-1}(X).
\]
given by 
\[
\phi \cdot [c] \coloneqq [\phi \cdot c] 
\]
is well defined by Lemma \ref{lemm:1}. We call this pairing the
\emph{Euclidean tropical intersection product} as well.
\end{Def}

The Euclidean tropical intersection product satisfies the following
symmetry property.

\begin{prop} \label{prop:1}
  Let $\phi _{1}$ and $\phi _{2}$ be two piecewise linear functions and $c$
  a Euclidean Minkowski weight. Then
  \begin{displaymath}
    \phi _{1}\cdot (\phi _{2}\cdot c)=\phi _{2}\cdot (\phi _{1}\cdot c).
  \end{displaymath}
\end{prop}
\begin{proof}
  This is proved in \cite[Proposition 3.15]{BBS:caps}.
\end{proof}



We define Euclidean tropical top intersection numbers of
piecewise linear functions on $X$.  
\begin{Def}\label{def:tropintersectiondiv}
Assume that $X$ is balanced with balancing condition $[X]$. Let $\phi_1,
\dotsc, \phi_n \in \PL(X)$. The \emph{Euclidean tropical top
intersection number} $\langle \phi_1 \dotsm \phi_n\rangle $ is defined
by 
\[
  \langle \phi_1 \dotsm \phi_n\rangle
\coloneqq \deg \left(\phi_1\dotsm \phi_{n-1} \cdot \left(\phi_n \cdot
  [X]\right)\right). 
\] 
This defines a multilinear map
\[
\underbrace{\PL(X) \times \dotsb \times \PL(X)}_{n \timess} \longrightarrow \R.
\]
It is symmetric by Proposition \ref{prop:1}.
\end{Def}

 \begin{rem}\label{rem:cp}
   In \cite{GR}, the author works with weakly embedded conical
   complexes with an integral structure. As a consequence of working
   with a weakly embedded conical complex, only a tropical
   intersection product between tropical cycles 
   and so called \emph{combinatorially principal} piecewise linear functions (which are called \emph{Cartier divisors}) can
   be defined. In our setting, we 
   assume that the complex is quasi-embedded.  It follows that 
   every piecewise linear function is combinatorially principal, hence arbitrary
   products between piecewise linear functions and tropical cycles can be
   defined. As we will see later, the price to pay for this
   simplification in the algebro-geometric setting of Section~\ref{sec:from-weak-embeddings} is
   that we will have to add more components at the 
   boundary to be sure that the conical complex corresponding to a
   toroidal embedding is quasi-embedded. Moreover, in the study of Monge--Amp\`ere measures it is more natural to replace the integral
   structure by a Euclidean structure. 
 \end{rem}
 
 \subsection{Conical functions on Euclidean conical spaces}
 As before, $X$ denotes a Euclidean
conical space of pure dimension $n$ with quasi-embedding given by
$\iota_{X} \colon X \to N^{X}$.
%
%
%
%

\begin{Def} \label{def:3}
Let $\Pi$ be a conical complex on $X$. A \emph{conical function} $\psi$ on $\Pi $ is a map
  \begin{displaymath}
    \psi\colon \Pi (1) \longrightarrow \R.
  \end{displaymath}
  The space of conical functions on $\Pi $ is denoted by $\Conic(\Pi
  )$. 
  
  If $\Pi '\geq \Pi $ in $R(X)$, $f \colon \Pi' \to \Pi$ the
  corresponding
  morphism of conical complexes, and $\psi'$ a conical function on $\Pi
  '$, then the \emph{push-forward $f_{\ast}\psi'$ of $\psi'$ by $f$} is the conical function on $\Pi$ given by
  \begin{displaymath}
    f_{\ast}\psi'(\rho )=\psi'(\rho ),\ \forall \rho \in \Pi (1).
  \end{displaymath}
  In other words, if $j\colon \Pi (1) \hookrightarrow \Pi '(1)$ is the
  canonical inclusion of the sets of rays, then we have that $f_{\ast}\psi'=\psi'\circ
  j$, so $f_{\ast}$ is just \enquote{to forget} the rays of $\Pi '$ that
  are not in $\Pi $. 
\end{Def}

\begin{rem}
  One should not identify conical functions on $\Pi $ and one
  dimensional Minkowski 
  weights. Both are maps from $\Pi (1)$ to $\R$ but their role is very
  different.  
\end{rem}

\begin{Def}
If $\phi $ is a piecewise linear function on $\Pi$, then the associated conical function (which is also denoted by $\phi$) is defined as
  \begin{displaymath}
   \phi(\rho )=\phi (\hat{v}_{\rho }),
  \end{displaymath}
where $\hat{v}_{\rho } = \hat{v}_{\rho/\{0\}}$ is the generator of
$\rho $ of norm 1.
\end{Def}


The following result follows as in the classical case of fans.
\begin{lemma}\label{lemm:3}
 Let $\Pi \in R(X)$ be a conical complex on $X$. If $\Pi $ is simplicial then the map
  \begin{displaymath}
    \PL(\Pi )\longrightarrow \Conic(\Pi ) 
  \end{displaymath}
  is an isomorphism.
\end{lemma}
\begin{proof}
  This amounts to the well known fact that affine functions on a
  simplex are in bijective correspondence with tuples of values on
  the vertices of the simplex. 
\end{proof}

In view of Lemma \ref{lemm:3} we can define the push-forward map of
piecewise linear functions on simplicial subdivisions.

\begin{Def}\label{def:push-cartier}
  Let $\Pi ''\ge \Pi' \ge \Pi \in R(X)$ with $\Pi '$ simplicial. Let $f\colon \Pi ''\to \Pi' $ be
  the corresponding morphism of conical complexes. Then the
  push-forward of piecewise linear functions $ \PL(\Pi '')\to
  \PL(\Pi ')$ is defined as the composition
  \begin{displaymath}
    \PL(\Pi '')\longrightarrow \Conic(\Pi '')
    \xlongrightarrow{f_{\ast}} \Conic(\Pi ')
    \xlongrightarrow{\cong}\PL(\Pi ').
  \end{displaymath}
  More concretely, if $\phi $ is a piecewise linear function on $\Pi
  ''$ then $f_{\ast} \phi $ is the unique piecewise linear function on
  $\Pi '$ that agrees with $\phi $ in the 1-skeleton of $\Pi '$. 
\end{Def}

We now define the space of conical functions on $X$ as an inverse over all conical complex structures. 

\begin{Def}\label{def:conical-fun}
   The space of
   conical functions on $X$ is space of all functions $\phi \colon
   X\to \R$ such that $\phi (\lambda x)=\lambda \phi (x)$ for all
   $x\in X$ and $\lambda \in \R_{\ge 0}$ with the topology of
   pointwise convergence. 
\end{Def}

\begin{rem}
  The real vector spaces $\PL(\Pi )$ and $\Conic(\Pi)$ are finite
  dimensional. Hence they have a canonical topology. It is easy to
  verify that there are canonical identification
  \begin{equation}
  \label{eq:3}
  \Conic(X) = \varprojlim_{\Pi \in R(X
      )} \Conic(\Pi )= \varprojlim_{\Pi \in R_{\sp}(X
      )} \PL(\Pi ),
  \end{equation}
  where the limit is taken in the category of topological vector
  spaces with respect to the push-forward maps. The second
  identification follows from Lemma \ref{lemm:3} and the fact that
  simplicial subdivisions are cofinal.  
\end{rem}
Given an element $\psi \in \Conic(X)$, using \eqref{eq:3}, we can write $\psi = (\psi_{\Pi})_{\Pi \in R_{\sp}(X)}$ for $\psi_{\Pi} \in \PL(\Pi)$.

We can extend the Euclidean tropical
top intersection number of Definition \ref{def:tropintersectiondiv} to the case where there is at most one conical function involved.

\begin{lemma}\label{lemm:4}
Let $z\in Z_{k}(X)$ be a Euclidean tropical cycle of dimension $k$, $\phi
_{1},\dots ,\phi _{k-1} \in \PL(X)$ piecewise linear functions on $X$
and $\psi= (\psi_{\Pi })_{\Pi
  \in R_{\sp}(\Pi )} \in \Conic(X)$ a conical function. Choose a
simplicial subdivision $\Pi \in R_{\sp}(X)$ where $z$ can be
represented
by a Euclidean Minkowski weight $c$ and such that all of the
$\phi_i$'s are defined on $\Pi $. Then
the product 
\begin{displaymath}
  \psi _{\Pi} \cdot \phi _{1}\cdots  \phi _{k-1}\cdot c.
\end{displaymath}
is independent of the choice of $\Pi $. 
\end{lemma}
\begin{proof}
 In view of Lemma \ref{lemm:1} we are reduced to prove the following
 projection formula. Let $\Pi '\ge \Pi $ be simplicial conical complexes on $X$
and $f \colon \Pi' \to \Pi$ the corresponding morphism. Moreover, let $c_{1}\in M_{1}(\Pi )$ be a Euclidean Minkowski weight of dimension one on
 $\Pi $ and $\phi $ a piecewise linear function on $\Pi '$. Then
 \begin{equation}\label{eq:4}
   f^{\ast}(f_{\ast} \phi \cdot c_{1}) =  \phi \cdot f^{\ast}c_{1}. 
 \end{equation}

 Again by Lemma \ref{lemm:1}, we have that
 \begin{equation}
   \label{eq:5}
   f^{\ast}(f_{\ast} \phi \cdot c_{1})=f^{\ast}f_{\ast} \phi \cdot  f^{\ast}c_{1}.
 \end{equation}
 The piecewise linear function $\phi -f^{\ast}f_{\ast} \phi$ satisfies
   \begin{displaymath}
     (\phi -f^{\ast}f_{\ast} \phi)|_{\rho
   }=0,\ \forall \rho \in \Pi '(1),
 \end{displaymath}
 while the Euclidean Minkowski weight $f^{\ast}c_{1}$ satisfies
   \begin{displaymath}
     f^{\ast}c_{1}(\rho )=0,\ \forall \rho \in \Pi
   '(1)\setminus \Pi (1).
   \end{displaymath}
   From the explicit description of the product in Definition
   \ref{def:4} we deduce
   \begin{equation}
     \label{eq:6}
     (\phi -f^{\ast}f_{\ast} \phi)\cdot f^{\ast}c_{1} =0.
   \end{equation}
   Equations \eqref{eq:6} and \eqref{eq:5} imply \eqref{eq:4}, which
   proves the lemma.
\end{proof}


\begin{Def} \label{def:5}
Let $z$, $\phi
_{1},\dots ,\phi _{k-1}$ and $\psi= (\psi_{\Pi })_{\Pi
  \in R_{\sp}(X)}$, $\Pi $ and $c$ be as in Lemma~\ref{lemm:4}. Then the 
\emph{Euclidean top intersection number} of $z$, $\phi
_{1},\dots ,\phi _{k-1}$ and $\psi$ is defined by
\begin{displaymath}
  \langle \psi\cdot \phi _{1}\cdots \phi _{k-1}\cdot z\rangle \coloneqq
  \deg(\psi _{\Pi} \cdot \phi _{1}\cdots  \phi _{k-1}\cdot c).
\end{displaymath}
\end{Def}

One of the main motivations of this article is to extend Definition
\ref{def:5} to certain cases where all the functions involved are (not
necessarily piecewise linear) conical functions and not just one of them.

\section{Monge--Amp\`ere measures}
\label{sec:monge-ampere-meas}

Throughout this section $X$ will denote an $n$-dimensional balanced Euclidean conical space with quasi-embedding given by
$\iota_{X} \colon X \to N^{X}$ and balancing condition $[X]$. 

The goal of this section is to prove that given $\mathcal{C}$, an admissible family of concave functions on $X$ (Definition
\ref{tropinef}), for any $\mathcal{C}$-concave conical function
$\phi$ on $X$ (Definition \ref{def:cont-b-c-nef}), its top
intersection number exists, is finite, and is given by the integral of
$\psi $ with respect to 
a weak limit of discrete Monge--Amp\`ere measures associated to the
elements of the given admissible family (Corollary \ref{corpure}).
This is done by introducing the notion of the \emph{size} of a
tropical cycle. This allows us to prove a Chern--Levine--Nirenberg
type inequality (Lemma \ref{tropicalintersectioninequality}) from
which we conclude the weak convergence of the discrete measures
(Theorem \ref{th:convergence-trop-meas}).

\subsection{Concave functions on balanced conical spaces}
\label{sec:weakly-conc-funct}

For the convenience of the reader, we gather here some definitions and
results form \cite{BBS:caps} but translated to conical spaces
instead of polyhedral ones.

%

\begin{Def}\label{def:7} A
  piecewise linear function $\phi$ on $X$ is called \emph{strongly
    concave} if it is the restriction of a concave function on $N^X$. It 
  is called \emph{concave} if, for
  every positive Minkowski cycle $w$ on $X$, the product $\phi \cdot w$ is
  positive. It is called \emph{weakly concave} if the product with the
  balancing condition $\phi \cdot [X]$ is positive.
\end{Def}

By \cite[Proposition 4.9]{BBS:caps}, a strongly concave function is
concave and a concave function is
weakly concave. 

There are also notions of strong concavity, concavity and weak
concavity for arbitrary functions $f \colon X \to \R$ which are not
necessarily piecewise linear (\cite[Definition
5.5]{BBS:caps}).

The main results we will use from \cite{BBS:caps} are the following:

\begin{theorem} [{\cite[Theorem 6.2]{BBS:caps}}]\label{thm:2}
  Let $\phi $ be a
  weakly concave function on $X$. Then $\phi $ is continuous. 
\end{theorem}

\begin{theorem}[{\cite[Theorem~6.23]{BBS:caps}}]\label{thm:3}
  Let $(f_{i})_{i \in \N}$ a
  sequence of weakly concave functions on $X$ such that there exists
  a dense subset $C\subseteq X$ and for every $x\in C$ the sequence
  $(f_i(x))_{i \in \N}$ has a finite limit. Then the sequence $f_{i}$ converges
  pointwise everywhere to a function $f\colon X\to \R$. The function
  $f$ is weakly concave, hence continuous, and the convergence is
  uniform on compacts. 
\end{theorem}

\subsection{$\mathcal{C}$-concave functions}
\label{sec:nef-b-divisors}

Let $| \cdot |$ be the Euclidean norm on 
$N^{X}$ and
let
\[
\mathbb{S}^{X} \coloneqq \left\{ v \in |\Pi| \, \Big{|} \, |\iota_{\Pi}(v)| = 1 \right\}. 
\] 
The set $\mathbb{S}^{X}$ is compact since it is the inverse image of
a compact space under 
a proper map by Lemma \ref{cpproper}.   
Note that the Euclidean normal vectors $\hat{v}_{\sigma /\tau}$ from
Definition \ref{def:vectorst} are elements in $\mathbb{S}^{X}$.


We can view the space of conical function on $X$ as a space of
functions on $\mathbb{S}^X$.  In fact $S^{X}$ inherits from $X$ a
structure of compact polyhedral space. Then $\PL(X )$ can be identified
with the space of $\R$-valued piecewise linear function on
$\mathbb{S}^{X }$, while
$\Conic(X)$ can be identified with the space of all $\R$-valued
functions on 
$\mathbb{S}^{X}$. We will use freely these identifications.   


We will denote by $C^0\left(\S^{X}\right)$ the space of continuous
functions on $\S^{X}$ with the topology of uniform convergence. 

\begin{rem}\label{lem:dense}
By the Stone-Weierstrass Theorem, the subset $\PL\left(X\right)
\subseteq C^0\left(\S^{X}\right)$ is dense. 
\end{rem}

\begin{Def}\label{tropinef} 
Let $\mathcal{C}\subseteq \PL\left(X\right)$ be a collection of
piecewise linear functions on $X$. We 
say that $\mathcal{C}$ is an 
\emph{admissible family of concave functions} on $X$ if the
following properties are satisfied:  
\begin{enumerate}
\item\label{item:7}  If $\phi_1, \dotsc, \phi_r \in \mathcal{C}$, then
  the Euclidean 
  tropical cycle $\phi_1
  \dotsm \phi_r \cdot [X]$ is positive, i.e. if it belongs to $Z^+_{n-r }(X)$.
\item \label{item:8} $\mathcal{C}$ is a convex cone. 
\item \label{item:4} The vector space $\mathcal{C}-\mathcal{C}$ 
 is dense in $C^0\left(\S^{X}\right)$.
\end{enumerate}
If $\mathcal{C}$ is an admissible family of concave functions on $X$, an element
$\phi \in \mathcal{C}$ will be called $\mathcal{C}$-concave. 
\end{Def}

\begin{rem}\label{rem:weakly-concave} By the first condition, every
  element $\phi $ of 
  $\mathcal{C}$ is weakly concave in the sense of Definition
  \ref{def:7}.  
\end{rem}

\begin{rem}
  In Definition \ref{tropinef} only conditions \ref{item:7} and
  \ref{item:4} are essential. In fact, if $\mathcal{C}$ is  a set
  satisfying only \ref{item:7} and \ref{item:4} then the convex cone
  generated by $\mathcal{C}$ satisfies the three conditions.   
\end{rem}
\begin{exa}\label{exm:4}
The collection of piecewise linear concave functions on $X$ in the sense of 
\cite[Definition 4.6]{BBS:caps} (see Definition \ref{def:7}) is an 
example of an admissible family of piecewise linear functions on $X$
by \cite[Remark 4.36]{BBS:caps}.
\end{exa}

\begin{exa}\label{exm:3}
We will see in the next section that if
$X = |\Pi_Y|$ comes from the geometry of a smooth and complete toroidal
embedding $U \hookrightarrow Y$ satisfying certain mild hypothesis,
then there is a canonical 
admissible family $\mathcal{C}$ of concave functions on $X$
induced by the
collection of nef toroidal divisors on smooth toroidal
modifications of $Y$. 
\end{exa}  

\begin{rem}
  As we will see in Example \ref{exm:5}, the families in Examples 
  \ref{exm:4} and \ref{exm:3} may be different. In fact it is
  conceivable that there are two toroidal embeddings giving rise to
  isomorphic balanced conical spaces but such that the spaces of
  functions coming from nef divisors are different. This is the
  reason why it is useful to be able to choose an admissible family of concave functions $\mathcal{C}$ and to just identify
  the needed properties instead of choosing a particular family.
\end{rem}

From now on we fix an admissible family $\mathcal{C}$ of concave functions
on $X$. 

\begin{Def} \label{def:cont-b-c-nef}
The \emph{space of $\mathcal{C}$-concave conic functions on $X$}, denoted by
$\Conic\left(X\right)_{\mathcal{C}}$, is the 
closure of $\mathcal{C}$ in
$\Conic\left(X\right)$ with respect to pointwise convergence.
\end{Def}



The following is a key result.
Before stating it, recall that for a subset $A$ of a topological space $T$, the
  sequential closure $|A|_{\seq}$ of $A$ is the set of all points that
  are limits of sequences in $A$. Then $|A|_{\seq}\subseteq
  \overline A$. The space $T$ is called a \emph{Fr\'echet-Urysohn} space if,
  for all $A\subseteq T$, the condition $|A|_{\seq}=\overline A$
  holds. A Fr\'echet-Urysohn space is sequential, hence the topology
  of such
  spaces is determined by the convergent sequences. 

\begin{theorem}\label{thm:1} The space $\Conic\left(X\right)_{\mathcal{C}}$ of $\mathcal{C}$-concave functions on $X$ is
  contained in $C^{0}(\S^{X})$. Moreover the
  topologies induced in this space by the one of
  $\Conic\left(X\right)$ and the one of $C^{0}(\S^{X})$ agree.
  That is, in $\Conic\left(X\right)_{\mathcal{C}}$ the topology
  of pointwise convergence and that of uniform convergence are the same. In
  particular  
  $\Conic\left(X\right)_{\mathcal{C}}$ is metrizable.
\end{theorem}
\begin{proof}
 Let $(f_{\alpha })_{\alpha \in I}$ be a net of piecewise linear functions in $\mathcal{C}$ that converge to a conical function $f$. Choose a
 countable dense collection of points $x_{1},x_{2},\dots $ of $X$. Since the topology of the space of conical functions is that of
 pointwise convergence, for any $i> 0$ there is an $\alpha _{i}$ such
 that, for all $\alpha \ge \alpha _{i}$ and all $j\le i$, the condition
 \begin{displaymath}
   |f_{\alpha }(x_{j})-f(x_{j})|<\frac{1}{i}
 \end{displaymath}
 is satisfied. Hence the sequence $(f_{\alpha _{i}})_{i>0}$ converges to $f$
 in a dense subset of $X$. By Remark \ref{rem:weakly-concave} the
 functions $f_{\alpha _{i}}$ are weakly concave. Therefore, by Theorem
 \ref{thm:3}, the
 sequence  $(f_{\alpha _{i}})_{i>0}$ converges to a weakly concave
 (hence continuous) function $g$, that agrees with $f$ on the points
 $x_{i}$, $i>0$. Let now $y$ be another point of $X$. Repeating
 the argument with the sequence of points $y,x_{1},x_{2},\dots$, we
 obtain a new continuous function $g_{1}$, that agrees with $f$ in the
 point $y$ and agrees with $g$ in a dense subset. Hence
 $g(y)=f(y)$. Since $y$ is arbitrary, we deduce that $f=g$. Therefore
 $f$ is weakly concave and is a continuous conical function. Moreover,
 $f$ is the limit of the sequence $(f_{\alpha
    _{i}})_{i>0}$. We conclude that the space of
  $\mathcal{C}$-concave functions is Fr\'echet-Urysohn. Hence the
  topology is determined by the convergent sequences. Using again
  Theorem \ref{thm:3} a sequence in
  $\Conic\left(X\right)_{\mathcal{C}}$ converges if and only if
  it converges uniformly in $S^{X}$. This concludes the proof.
\end{proof}

\subsection{The size of a Minkowski cycle}
\label{sec:size-minkowski-cycle}

We have the following monotonicity lemma which we will use later on. 
\begin{lemma}\label{lemmainequality}
Let $\phi_1, \phi_2$ be piecewise linear functions on $X$ and $\Pi \in R(X)$ a conical complex where they are defined.  Assume that $\phi_1(x)
\geq \phi_2(x)$ for all $x \in X$. Then for all positive
Euclidean Minkowski weights $c \in M^+_1(\Pi)$ and every vertex $\nu$ of
$\Pi $, the inequality
\[
\left(\phi_1 \cdot c\right) (\nu ) \leq \left(\phi_2 \cdot c\right)
(\nu )
\]
is satisfied.
\end{lemma} 
\begin{proof} 
We have 
\begin{align*}
  \left(\phi_1 \cdot c\right) (\nu)   =
  \sum_{\substack{\sigma \in\Pi(1)\\\nu \prec \sigma }}
  -\phi_1( \hat{v}_{\sigma})c(\sigma) \leq
  \sum_{\substack{\sigma \in\Pi(1)\\\nu \prec \sigma }}
  -\phi_2(\hat{v}_{\sigma})c(\sigma) =
  \left(\phi_2 \cdot c\right) (\nu ), 
\end{align*}
as we wanted to show.
\end{proof}

To define the size of a positive Euclidean Minkowski cycle we
choose an auxiliary function.

\begin{Def}
  Let $\varphi_{0} \colon N^{X}\to \R$ be a concave piecewise
  linear function satisfying
  \begin{equation}
    \label{eq:8}
    \varphi_{0}(v)\le -\|v\|,
  \end{equation}
  write $\varphi=\varphi_{0}\circ \iota _{X}$
  and let $z\in Z_k^+(X)$ be a $k$-dimensional positive Euclidean
  tropical cycle. Then \emph{the size of $z$ (with respect to
    $\varphi$)} is defined as
  \begin{displaymath}
    |z|_{\varphi} \coloneqq \deg(\left(\varphi \, \cdot \right)^k z) \in
    \R.
  \end{displaymath}
\end{Def}
\begin{rem}\label{rem:1}\ 
  \begin{enumerate}
  \item It is clear that such a function $\varphi_{0}$ exists. For
    instance to construct one we can choose an orthonormal basis of
    $N^{X }$, denote $u_{1},\dots,u_{r}$ the corresponding
    coordinates, and write
    \begin{displaymath}
      \varphi_{0}(u_{1},\dots,u_{r})=
      2r\min(0,u_{1},\dots,u_{r})-(u_{1}+\dots+u_{r}).
    \end{displaymath}
\item \label{item:1} Since $\varphi_{0}$ is concave, the function
  $\varphi$ is 
  strongly concave (Definition \ref{def:7}). Therefore, by 
  \cite[Proposition 4.9]{BBS:caps} it is concave. Hence, for every $k$
  dimensional positive
  cycle $z$ and $j\le k$, the cycle $(\varphi\cdot)^{j}\cdot z$ is
  positive. In particular the size of $z$ is positive. 
   \end{enumerate}   
\end{rem}
\begin{lemma}\label{exasize} 
\begin{enumerate}
\item If $z \in Z_0^+(X) \simeq \R_{\geq 0}$ is $0$-dimensional, then
\begin{equation}\label{def:zero}
|z|_{\varphi } = \sum_{\nu \in \Pi (0)}c(\nu) \in \R_{\geq 0},
\end{equation}
where $\Pi \in R(X)$ is any conical complex where $z$ is represented by a Minkowsky weight $c$. 
\item Let $z \in Z_1^+(X)$ be a positive $1$-dimensional Euclidean
  tropical cycle and let $\Pi$ be as above.  Then
  \begin{equation}
    \label{eq:9}
    |z|_{\varphi}\ge \sum_{\tau \in\Pi(1)}c(\tau).
  \end{equation}
\end{enumerate}
\end{lemma}
\begin{proof}
  The first statement follows directly from the definition. 
    Let $\widetilde{\Pi} $ be a subdivision  of $\Pi$ such that
  $\varphi$ is piecewise linear on
  $\widetilde{\Pi}$. By the definition of $\varphi$, 
  for every $\tau \in
  \widetilde{\Pi} (1)$, the inequality 
  $\varphi(\hat{v}_{\tau })\le -1$ is satisfied. Therefore
\begin{displaymath}
  |z|_{\varphi} = \deg(\varphi \cdot z)
      = \sum_{\tau  \in
        \widetilde{\Pi}(1)}-\varphi(\hat{v}_{\tau})c(\tau)
      \ge \sum_{\tau \in\widetilde{\Pi}(1)}c(\tau)
      = \sum_{\tau \in \Pi(1)}c(\tau). 
\end{displaymath}
\end{proof}


The following is a Chern--Levine--Nirenberg type inequality for the
Euclidean tropical intersection product and is a key step to prove
the main result of this section.

\begin{lemma}\label{tropicalintersectioninequality}
Let $z \in Z_k^+(X)$ and $\phi \in \PL(X)$ be a $k$-dimensional positive Euclidean
tropical cycle 
and a piecewise linear function, respectively. Assume that the
Euclidean tropical intersection product
$\phi \cdot z$  is a positive Euclidean tropical cycle. Then the
inequality
\[
\left|\phi \cdot z\right|_{\varphi} \leq \left(\sup_{\hat{v}\in
    \mathbb{S}^{X}}\left|\phi(\hat{v})\right|\right)\cdot
\left|z\right|_{\varphi}
\]
is satisfied.
\end{lemma}
\begin{proof}
Let $\widetilde{\Pi}
\geq \Pi$ be a subdivision of $\Pi$ such that $\varphi$ is piecewise
linear on $\widetilde{\Pi}$ and such that the cycle $z$ and 
the function $\phi $ are defined in $\widetilde \Pi $. We define the
positive real constant $B$ by  
\[
B \coloneqq \sup_{\tau \in
  \widetilde{\Pi}(1)}\left|\phi\left(\hat{v}_{\tau}\right)\right|\le
\sup_{\hat{v} \in
  \S^{X}}\left|\phi\left(\hat{v}\right)\right|.
\]
Then for every $\tau \in \widetilde{\Pi}(1)$ we have that
\[
\phi(\hat{v}_{\tau}) \geq -B \geq B\varphi(\hat{v}_{\tau}) .
\]
Hence, since both $\phi$ and $B\varphi$ are piecewise linear on
$\widetilde{\Pi}$, we conclude that  
\[
 \phi \geq B\varphi.
 \]
Therefore, using Lemma \ref{lemmainequality}, the positivity of
$\left(\varphi\, \cdot \right)^{k-1}
z$ (Remark \ref{rem:1}~\eqref{item:1}), and the commutativity of the
Euclidean tropical intersection product, we get  
 \[
   \left|\phi \cdot z\right|_{\varphi} =
   \deg\left(\left(\varphi \, \cdot \right)^{k-1}\phi \cdot z \right)=
   \deg\left( \phi \cdot \left(\varphi \, \cdot \right)^{k-1} z \right)\leq
   B \deg\left( \left(\varphi \, \cdot \right)^{k}z \right)= B|z|_{\varphi}, 
 \] 
 as we wanted to show.

\end{proof}

\subsection{Weak convergence of Monge--Amp\`ere measures}

We recall the definition of the total variation norm (see e.g.~\cite[Definition 4.2.5 and Proposition 4.2.5]{measures}). 
\begin{Def} Let $Y$ be a locally compact topological space and let $\mathcal{M}(Y)$ be the space of finite Radon measures on $Y$, i.e.~ the space of continuous linear forms on the space $C^0(Y)$ of continuous real-valued functions on $Y$ with respect to its weak topology.
  The \emph{total variation norm} $\| \cdot \|$ on $\mathcal{M}(Y)$ is given by
\[
\left\| \mu \right\| \coloneqq \sup\left\{\sum_{i = 1}^{\infty} \left|\mu(A_i)\right| \, \big{|} \, \{A_i\}_{i \geq 1} \subseteq Y \text{ measurable, } A_i \cap A_j = \emptyset, i \neq j, \, \bigcup_{i\geq 1}A_i = Y\right\}
\]
for any $\mu \in \mathcal{M}(Y)$. In case that $Y$ is compact, the
total variation is just the norm of the measure as a continuous linear
form on $C^{0}(Y)$.
\end{Def}

In order to prove the main result of this section (Theorem
\ref{th:convergence-trop-meas}), we use the following standard
version of Prokhorov's theorem which follows from
\cite[Proposition~8.6.2]{bogachev:measures}. 
\begin{theorem}\label{the:prokhorov} 
  Let $Y$ be a compact metrized space.
  Then the following is satisfied. 
  \begin{enumerate}
  \item $\mathcal{M}(Y)$ is a metrizable space. In particular, if $K \subseteq \mathcal{M}(Y)$ is compact, then $K$ is sequentially compact. 
  \item If $U \subseteq \mathcal{M}(Y)$ has bounded total variation, then $U$ is relatively compact.
  \end{enumerate}
\end{theorem}

\begin{Def}\label{def:discretemeasure}\ 
  \begin{enumerate}
  \item Let $z \in Z_1(X)$ be a $1$-dimensional Minkowski cycle
    and let $\Pi \in R(X)$ such that $z$ is represented by a Minkowski weight $c$ in $M_1(\Pi)$. 
     We define the discrete measure $\mu_z$ on
    $\S^{X}$ by  
\[
\mu _{z}\coloneqq \sum_{\tau \in \Pi(1)} c(\tau) \cdot \delta_{\hat{v}_{\tau}},
\] 
where $\delta_{\hat{v}_{\tau}}$ denotes the Dirac delta measure
supported on $\hat{v}_{\tau} \in \mathbb{S}^{X}$. (This does not depend on the choice of $\Pi$.) 
\item   Let $\phi \in \PL(X)$. The discrete
  Monge--Amp\`ere measure $\mu_{\phi}$ on
    $\S^{X}$ is defined by 
\[
  \mu_{\phi} \coloneqq \mu _{\phi^{n-1}\cdot [X]} 
\]
  \end{enumerate}
\end{Def}

The total variation of a discrete measure with finite support is given
by the sum of the absolute value of the measures of the points in the
support. Therefore, for $z$ and $\Pi $ as in Definition
\ref{def:discretemeasure} we have  
\begin{equation}\label{eq:10}
  \|\mu _{z}\| = \sum_{\tau \in \Pi(1)} |c(\tau)|.
\end{equation}

\begin{rem} Although defined in a different setting, we note the
  similarity between the discrete measure $\mu_{\phi}$ and the
  Monge--Amp\`ere measure $\mathcal{M}(g)$ given in \cite[Section~4.2]{BFJ:valuations}, defined with respect to a piecewise-affine
  plurisubharmonic function $g$ (see \cite[Proposition~4.9]{BFJ:valuations}).
\end{rem}  

The following proposition is a consequence of Lemma
\ref{tropicalintersectioninequality}.

\begin{prop}\label{prop:boundedmeasure}
Let $\psi$ be a $\mathcal{C}$-concave function on $X$, and let $(\phi
_{j })_{j\in \N}$ be a sequence of $\mathcal{C}$-concave piecewise linear functions converging to $\psi$. Moreover, fix a collection $\gamma_1,
\dotsc, \gamma_{n-1-k} \in \mathcal{C}-\mathcal{C}$ for $k \in
\{0,\dotsc, n-1\}$. Then the set  
\[
\left\{\mu _{\gamma_1 \dotsm \gamma_{n-1-k} \cdot \phi_{\alpha
    }^k\cdot [X]}\right\}
\]
of measures on $\S^{X}$ has bounded total variation. 
\end{prop}
\begin{proof}
Since, by Theorem \ref{thm:1}, the convergence 
\[
\lim_{j\in \N } \phi_{j }|_{\S^{X}} 
= \phi_{\D}|_{\S^{X}}
\]
is uniform and $\phi_{\D}|_{\S^{X}}$ is a continuous function on a
compact set, there exists a
positive real number $B$ such that  
\[
\sup_{j\in \N} \, \sup_{x \in \S^{X}} \left|{\phi_{j}}(x)\right| \leq B.
\]

By assumption, for each $\ell = 1, 
\dotsc, n-1$, there exist elements $\beta ^{0}_{\ell}$ and
$\beta^{1}_{\ell}$ in $\mathcal{C}$ such
that $\gamma_{\ell} = \beta ^{0}_{\ell} - \beta^{1}_{\ell}$.

Since
there are finitely many, we may choose a positive real number $C$ such
that
\[
\sup_{\ell,i} \, \sup_{x \in \S^{X}} \left|\beta ^{i}_{\ell}(x)\right| \leq C.
\]
Since the involved piecewise linear functions are $\mathcal{C}$-concave, we have that for every
$j\in \N$ and for every tuple $(i_{1},\dots,i_{n-1-k})\in \{0,1\}^{n-1-k}$,
the $1$-dimensional Euclidean tropical cycle
\[
\beta^{i_{1}}_1 \dotsm \beta^{i_{n-1-k}}_{n-1-k} \cdot \left(\phi_{j
  }\right)^k \cdot [X]
\]
is positive. 

Fix $j\in \N$ and let $\Pi \in R(X)$ be a conical complex on $X$ where the
$\phi_{j }$ and the functions $\gamma _{i}$, $i=1,\dots,n-1-k$, are
defined.
Then, using equation \eqref{eq:10}, Lemma
\ref{tropicalintersectioninequality} and the 
estimate \eqref{eq:9}, we get
\begin{align*}
\|\mu _{\gamma_1 \dotsm \gamma_{n-1-k} \cdot \phi_{\alpha }^k \cdot
  [X]}\|
  &=\sum_{\tau \in \Pi(1)} \left| \gamma_1 \dotsm
    \gamma_{n-1-k}\cdot \phi_{\alpha }^k\cdot[X](\tau) \right|\\
  &= \sum_{\tau \in \Pi(1)} \left|(\beta ^{0}_1-\beta^{1}_1) \dotsm
   (\beta ^{0}_{n-1-k} - \beta^{1}_{n-1-k})\cdot \phi_{\alpha }^k
   \cdot[X](\tau)\right|  \\ 
  &\leq \sum_{\{0,1\}^{n-1-k}}\sum_{\tau \in \Pi(1)}
    \beta^{i_{1}}_1 \dotsm
    \beta^{i_{n-1-k}}_{n-1-k}\cdot \phi_{\alpha }^k
    \cdot[X](\tau)\\
  &\le \sum_{\{0,1\}^{n-1-k}}\left|\beta^{i_{1}}_1 \dotsm
    \beta^{i_{n-1-k}}_{n-1-k}\cdot \phi_{\alpha }^k\cdot
    [X]\right|_{\varphi}\\
  &\leq 2^{n-1-k}\cdot C^{n-1-k} \cdot B^k \cdot |[X]|_{\varphi},
\end{align*}
proving the proposition.
\end{proof}

The following is the main result of this section. The proof is
inspired in the classical proof of the existence of Monge--Amp\`ere
measures of \cite[Proposition 3.1]{TR}. 
\begin{theorem}\label{th:convergence-trop-meas}
Let $\psi$ be a $\mathcal{C}$-concave function on $X$, $k \in
\{0,\dotsc, n-1\}$ and 
$\gamma_1, \dotsc, \gamma_{n-1-k} \in \mathcal{C}-\mathcal{C}$. We view $\psi$ as a function on $\S^X$. Then the following holds true.
\begin{enumerate}
\item \label{item:2} Let $(\phi_{i})_{i\in \N}$ and $(\phi' 
_{j})_{j\in \N}$ be sequences of $\mathcal{C}$-concave piecewise linear functions both converging to $\psi$. Assume that
\begin{displaymath}
  \lim_{i\in \N}\mu _{\gamma_1 \dotsm \gamma_{n-1-k} \cdot
  \phi_{i}^k \cdot [X]}=\mu ,\qquad
   \lim_{j\in \N}\mu _{\gamma_1 \dotsm \gamma_{n-1-k} \cdot
  \phi_{j}'^k \cdot [X]}=\nu, 
\end{displaymath}
for some Radon measures $\mu $ and $\nu $ (with respect to the
weak-$\ast$ topology). Then
\begin{displaymath}
  \mu =\nu.
\end{displaymath} 
\item \label{item:3} The map from
  $\PL\left(X\right)_{\mathcal{C}}$ to Radon measures on
  $\S^{X}$ given by
\begin{equation}\label{eq:7}
  \phi \longmapsto \mu _{\gamma_1 \dotsm \gamma_{n-1-k} \cdot
  \phi^k \cdot [X]}
\end{equation}
extends to a continuous operator from
$\Conic\left(X\right)_{\mathcal{C}}$ to
Radon measures on $\S^{X}$. This operator is also
denoted as in (\ref{eq:7}).
\end{enumerate}
\end{theorem}
\begin{proof}
The fact that statement \eqref{item:2} implies statement
\eqref{item:3} is a
standard consequence of Theorem~\ref{thm:1}, Proposition~\ref{prop:boundedmeasure} 
and Theorem~\ref{the:prokhorov}.

We prove the theorem by induction on $k$. If $k=0$ there is nothing to
prove. So we can assume that both statements of the Theorem are true
for $k-1$. By part~\eqref{item:4} of Definition \ref{tropinef}, in
order to prove that $\mu =\nu $, it is enough to
prove that $\mu (\eta)=\nu (\eta)$ for $\eta\in
\mathcal{C}-\mathcal{C}$.

By Proposition \ref{prop:1} we have that
\begin{align*}
  \mu (\eta)& =\lim_{i \in \N}\deg(\eta \cdot \gamma_1 \dotsm
  \gamma_{n-1-k} \cdot
  \phi_{i}^k \cdot [X])\\
   &=
 \lim_{i\in \N}\deg(\phi_{i }\cdot \eta \cdot \gamma_1
  \dotsm \gamma_{n-1-k} \cdot 
  \phi_{i}^{k-1} \cdot [X]) \\ 
  &= \lim_{i\in \N} \mu _{\eta
    \cdot \gamma_1 
  \dotsm \gamma_{n-1-k} \cdot \phi_{i }^{k-1} \cdot
  [\Pi]}(\phi_{i}).
\end{align*}
By induction hypothesis, the sequence of measures $\mu _{\eta \cdot
  \gamma_1 
  \dotsm \gamma_{n-1-k} \cdot \phi_{i }^{k-1} \cdot [X]}$, $i\in \N$,
converges to the measure $\mu _{\eta \cdot \gamma_1
  \dotsm \gamma_{n-1-k} \cdot \psi^{k-1} \cdot [X]}$. Moreover,
by Theorem \ref{thm:1} the sequence of functions  
$\phi _{j}, j \in \N$, converge uniformly to the continuous function $\psi$. Therefore, the double limit
\begin{displaymath}
  \lim_{(i,j)\in \N\times \N} \mu _{\eta
    \cdot \gamma_1 
  \dotsm \gamma_{n-1-k} \cdot \phi_{i }^{k-1} \cdot
  [X]}(\phi_{j})
\end{displaymath}
exists and agrees with the diagonal limit $i=j$. Therefore
\begin{displaymath}
  \mu (\eta)=\mu _{\eta \cdot \gamma_1 \dotsm \gamma_{n-1-k} \cdot
    \psi^{k-1} \cdot [X]}(\psi).
\end{displaymath}
Similarly,
\begin{displaymath}
  \nu (\eta)=\mu _{\eta \cdot \gamma_1 \dotsm \gamma_{n-1-k} \cdot
    \psi^{k-1} \cdot [X]}(\psi).  
\end{displaymath}
Hence, we get that $\mu (\eta)=\nu(\eta)$.
This concludes the proof of the theorem. 
\end{proof}

\begin{Def}
  Let $\psi$ be a $\mathcal{C}$-concave function on $X$. The associated
  \emph{Monge--Amp\`ere measure} is defined by
  \begin{displaymath}
    \mu _{\psi}\coloneqq \mu _{\psi^{n-1}\cdot [X]}.
  \end{displaymath}
\end{Def}
We obtain the following corollary which is the main result of this section.
\begin{cor}\label{corpure}
Let $\psi$ be a $\mathcal{C}$-concave function on
$X$ and let $(\phi _{i})_{i\in \N}$ be a sequence of
$\mathcal{C}$-concave piecewise linear functions on $X$ converging to $\psi$. Then the limit 
\[
\deg(\psi) \coloneqq \lim_{i \in \N}\deg\left(\phi_{i}^{n} \cdot [X]  \right)
\]
exists, is finite, and is given by
\[
\deg(\psi) = \int_{\mathbb{S}^{X}} \psi(u)\,d \mu_{\psi}.
\]
It is called the \emph{degree} of the $\mathcal{C}$-concave function $\psi$. 
\end{cor}
\begin{prop/Def}\label{propdef:mixedmeasure}
There is a symmetric map from the space of $(n-1)$-tuples of
$\mathcal{C}$-concave functions on $X$ to the space of finite
measures on $\mathbb{S}^{X}$, called the \emph{mixed Monge--Amp\`ere
  measure}, and denoted by
\begin{displaymath}
  (\psi_{i_1}, \dotsc, \psi_{i_{n-1}})\longmapsto \mu_{\psi_{i_1}, \dotsc, \psi_{i_{n-1}}},
\end{displaymath}
such that for every natural number $\ell$ and for every choice of
non-negative real numbers $\lambda_1, \dotsc, \lambda_{\ell}$, the
equality
\[
\mu_{\lambda_1\psi_1 + \dotsc +\lambda_{\ell}\psi_{\ell}} = \sum_{i_1, \dotsc, i_{n-1} = 1}^{\ell} \lambda_{i_1}\dotsc \lambda_{i_{\ell}} \mu_{\psi_{i_1}, \dotsc, \psi_{i_{n-1}}}
\]
is satisfied for every collection $\psi_1, \dotsc, \psi_{\ell}$ of $\mathcal{C}$-concave functions on $X$. 
\end{prop/Def} 
\begin{proof}
The argument is the same as the one given in the proof of \cite[Theorem~5.17]{BM}.
\end{proof}
The following corollary follows from the definition of the mixed Monge--Amp\`ere measure and Corollary \ref{corpure}. 
\begin{cor}\label{cormixed}
Let $\psi_1, \dotsc, \psi_n$ be a collection of $\mathcal{C}$-concave functions on $X$, and let $(\phi _{i,j })_{j\in \N}$,
$i=1,\dots,n$  be
sequences of
$\mathcal{C}$-concave piecewise linear functions converging respectively to
$\D_{i}$. 
Then the limit
\[
\deg\left( \psi_1 \dotsm \psi_n \right) \coloneqq \lim_{j} \deg\left(\phi_{1,j } \dotsm
\phi_{n,j} \cdot [X]  \right)
\]
exists, is finite and is given by 
\[
\deg\left( \psi_1 \dotsm \psi_n \right) = \int_{\mathbb{S}^{X}} \psi_1(u)
\,d\mu_{\psi_2,\dotsc, \psi_n}.
\]
Moreover, for any $1 \leq i \leq n$, we have integral formulae 
\[
\int_{\mathbb{S}^{X}} \psi_1(u) \,d\mu_{\psi_2,\dotsc, \psi_n} =
\int_{\mathbb{S}^{X}} \psi_i(u) \,d\mu_{\psi_1,\dotsc,
  \widehat{\psi}_i, \dotsc, \psi_n}.
\]
It is called the
\emph{mixed degree} of the $\mathcal{C}$-concave functions $\psi_1,
\dotsc, \psi_n$. 
\end{cor}


\begin{rem}\label{rem:difference}
  By multilinearity, we can extend the definition of Monge--Amp\`ere
  measures and degrees to functions of the space
  $\Conic(X)_{\mathcal{C}}-\Conic(X)_{\mathcal{C}}$. Then the
  corollaries \ref{corpure} and \ref{cormixed} extend to this
  setting. 
\end{rem}

\section{Toroidal embeddings and rational conical polyhedral spaces}\label{sec:toroidal-embeddings}

In this section, we define quasi-embedded
\emph{rational} 
conical polyhedral spaces. In short, these are conical polyhedral spaces endowed with a lattice structure. We recall the
definition of a toroidal
embedding and describe a natural rational conical polyhedral space
associated to it (see \cite{KKMD} or \cite{AMRT} for further
details). Following \cite{GR}, we also give a natural weak embedding
of this space. Moreover, we show that by adding boundary components
one can modify the toroidal structure of a toroidal embedding in such
a way that the rational conical polyhedral space becomes
quasi-embedded. Then we describe the proper toroidal birational
modifications of a toroidal embedding which, on the combinatorial
side, correspond to subdivisions of rational conical
complexes on this rational conical space.


\subsection{Quasi-embedded rational conical polyhedral spaces}

\begin{Def}\label{conicalcomplex}

Let $X$ be a second countable topological space. A \emph{rational conical polyhedral structure} on $X$ is a pair
\[
\Pi = \left(\{\sigma^{\alpha}\}_{\alpha \in \Lambda}, \{M^{\alpha} \}_{\alpha \in \Lambda}\right)
\]
consisting of a finite covering by closed subsets $\sigma^{\alpha} \subseteq X$
and for each $\sigma^{\alpha}$, a finitely generated $\Z$-module
$M^{\alpha}$ of continuous, $\R$-valued functions on $\sigma^{\alpha}$
satisfying the following conditions. Let $N^{\alpha} \coloneqq
\Hom(M^{\alpha}, \R)$ denote the dual lattice.  
\begin{enumerate}
\item  For each $\alpha \in \Lambda$, the evaluation map $\phi^{\alpha}\colon \sigma^{\alpha} \to N^{\alpha}$ given by the assignment
\[
v \longmapsto (u \mapsto u(v)) \quad (u \in M^{\alpha} ),
\]
maps $\sigma^{\alpha}$ homeomorphically to a strictly convex,
full-dimensional, rational polyhedral cone in $N_{\R}^{\alpha}$. We call the sets $\sigma^{\alpha}$ \emph{cones}. 
\item The preimage under $\phi^{\alpha}$ of each face of
  $\phi^{\alpha}\left(\sigma^{\alpha}\right)$ is a cone
  $\sigma^{\alpha'}$ for some index $\alpha' \in \Lambda$, and we have
  that $M^{\alpha'} = \left\{u|_{\sigma^{\alpha'}} \, \big{|} \, u \in
    M^{\alpha}\right\}$.
\item The intersection of two cones is a union of common faces.
\end{enumerate}
The $\Z$ modules $M^{\alpha}$ give $X$ a so called
\emph{integral structure}. 

\end{Def}
A subdivision of a rational conical polyhedral structure is defined as in Section \ref{measures-section} but with the condition that it has to be rational as well. And we say that two rational conical polyhedral structures are \emph{equivalent} if they admit a common subdivision. 
\begin{Def}\label{def:rat-space} A \emph{rational conical polyhedral space} $X$ is a second countable topological space equipped with an equivalence class of rational conical polyhedral structures. A \emph{rational conical polyhedral complex on} $X$ is the choice of a representative of the class of rational conical polyhedral structures on $X$.
\end{Def}
Most of the notations and
terminology of Section \ref{measures-section} carry over to the case 
of rational conical polyhedral complexes, by taking into account the
integral structure.
\begin{enumerate}
\item Rational conical polyhedral complexes and rational conical polyhedral spaces will be referred as \emph{rational conical complexes} and \emph{rational conical spaces}, respectively.
\item Given a rational conical space $X$, the set $R(X)$ consists of all rational conical complexes on $X$ ordered by inclusion. This has the structure of a directed set.
\item As in Remark \ref{rem:connected}, if $X$ is a rational conical space, then the set of cones of dimension
  zero in any rational conical complex on $X$ is in bijection with the set of connected components of $X$.
\item The terminology concerning cones, faces, interior, support and dimension is the
  same as in the non-rational case keeping in mind the compatibility
  between the integral structures.
  \item The notion of a simplicial rational conical complex is the
    same. However, in the rational case we also have a notion of
    smoothness. A rational conical complex is called \emph{smooth} if
    every cone $\sigma \in \Pi$ is unimodular, i.e.~if
    $\phi^{\sigma}(\sigma)$ is generated by a $\Z$-basis of 
    $N^{\sigma}$. Clearly, a smooth rational conical complex is
    automatically simplicial. We denote the set of simplicial and smooth complexes on a rational conical space $X$ with their directed set structures by $R_{\sp}(X)$ and $R_{\sm}(X)$, respectively.
\item The notion of a morphism between rational conical spaces
  is the same except that we require the restriction to each cone
  to be integral.
\item The notions of weakly-embedded and quasi-embedded rational conical
  spaces are the same except that the co-domain of the weak-
  (respectively quasi-) embedding is an $\R$-vector space
  $N^{X}_{\R}$ with an integral structure $N^{X}$ and the
  restriction of the weak (respectively quasi-) embedding to each cone
  is required to be integral. 
\end{enumerate}

\subsection{A bridge between Euclidean and integral
  structures}\label{sec:bridge}

 Following \cite{GR}, there is a \emph{rational tropical intersection
   product} on quasi-embedded rational conical spaces. We compare
 the rational tropical intersection with the Euclidean one from Section
 \ref{measures-section} by means of
 the \emph{normalization} of cycles.

 We fix a quasi-embedded rational conical space $X$, with quasi-embedding given by $\iota_X \colon X \to N_{\R}^X$, and start with some
 definitions. These are adapted from \cite[Section 3.1]{GR} and are
 small modifications of standard concepts in tropical geometry. See
 for instance the articles \cite{AR},
\cite{FS} and \cite{KA}.   

\begin{Def}\label{def:lattice-vectorst} Let $\Pi$ be a rational conical complex on $X$. Let $k\ge 0$ be an integer and let
  $\tau\in \Pi (k-1)$ be a cone. For every cone $\sigma \in \Pi(k)$
  with $\tau \prec \sigma $. We define the \emph{lattice normal vector
    $v_{\sigma/\tau}$ of $\sigma$ relative to $\tau$}  to be the image
  in the quotient $N^{X}_{\R}/N^{\tau }_{\R}$ of the unique
  generator of $N^{\sigma }/N^{\tau }$ that points in the direction of
  $\sigma $.  For every pair of cones $\sigma $ and $\tau$ as before
  we will chose  a lifting $\tilde v_{\sigma/\tau}\in N^{\sigma }_{\R}$ of
  $v_{\sigma/\tau}$. If $k = 1$, we write $v_{\sigma}\coloneqq
  v_{\sigma / \{0_{\sigma 
    }\}} = \tilde v_{\sigma / \{0_{\sigma }\}}.$ 
\end{Def}

\begin{Def}\label{MWrat} Let $\Pi$ be a rational conical complex on $X$. A $k$-dimensional weight on $\Pi$ is called a 
  \emph{$k$-dimensional (lattice) Minkowski weight}
  on $\Pi$ if, for every cone $\tau \in \Pi(k-1)$, the relation
\begin{equation}\label{eq:1.1}
  \sum_{\substack{\sigma \in \Pi (k)\\\tau \prec \sigma
    }}c\left(\sigma\right)v_{\sigma/\tau} =0   
\end{equation}
holds true in $N^{X}_{\R}/N^{\tau }_{\R}$.
Equivalently, $c$ satisfies the relation
\begin{equation}\label{eq:2}
  \sum_{\substack{\sigma \in \Pi (k)\\\tau \prec \sigma
    }}c\left(\sigma\right)\tilde v_{\sigma/\tau} \in N^{\tau }_{\R}.      
\end{equation}
Usually, lattice Minkowski weights will be called Minkowski weights. 
The $k$-dimensional Minkowski weights on $\Pi$ form a real
vector subspace, which is also denoted by $M_k(\Pi)$. Note the symbol
$M_k(\Pi)$ denotes lattice Minkowski weights when $\Pi $ is rational
and Euclidean Minkowski weights when $\Pi $ is Euclidean.   
\end{Def}
The condition \eqref{eq:1.1} is called the \emph{(lattice) balancing condition}
around $\tau $, while the condition \eqref{eq:1} is called the
\emph{Euclidean balancing condition}. The balancing conditions
\eqref{eq:1} and \eqref{eq:1.1} depend on the choice
of the quasi-embedding.

The following notions carry over from the Euclidean to the
lattice case directly.

\begin{enumerate}
\item The definition of balanced rational conical space is the same
  as in the Euclidean case.
\item The definition of the pull-back along a subdivision is the same.
\item The definition of the group of \emph{(lattice) tropical cycles}
  is analogous. This group is also denoted by $Z_k(X)$. 
\item The definition of the space $\PL(X)$ of piecewise linear
  functions on $X$ is the same.  We must have in mind that now we
  only allow rational subdivisions although we are
  working with real coefficients. 
  
\item Since $X$ has a rational structure, we can define $X(\Q)$ as the
  union of the subsets of rational points on each rational cone.   
  
\item The space of conical functions $\Conic(X)$ is defined as the
  space of functions $f$ on  $X(\Q) = X \cap \iota _{X }^{-1}(N^{X}_{\Q})$ with real values, satisfying
  \begin{displaymath}
    f(\lambda x)=\lambda f(x),\qquad \lambda \in \Q_{\ge 0}, 
  \end{displaymath}
  with the topology of pointwise convergence. Then 
     \[
     \Conic(X) = \varprojlim_{\Pi \in R_{\sm}(X)} \PL(\Pi)=
     \varprojlim_{\Pi \in R_{\sp}(X)} \PL(\Pi). 
 \]
\end{enumerate}
A difference with the Euclidean case is that now the limit is taken
over a countable set, so every convergent net of conic functions has a
converging subsequence.

 \begin{rem}\label{rem:2}
   Consider the quasi-embedded rational conical space $X$  and let
   $\widehat {X} $ be the Euclidean one obtained by choosing a
   metric on $N_{\R}^{X }$ and forgetting the integral structure. Since
  we allow only rational conical complexes $\Pi$ on $X$, the set
  $R(\widehat{X})$ is much bigger than $R(X)$ and hence the spaces of
  functions are different. Nevertheless, 
   there is a commutative diagram
   \begin{equation} \label{eq:12}
     \xymatrix{ \PL(X)\ar[r] \ar[d] &  \Conic(X)\\
       \PL(\widehat X) \ar[r] & \Conic(\widehat X)\ar[u]
     }
   \end{equation}
   The space $\PL(\widehat X)$ is the
   space of all piecewise linear functions on $X$, while
   $\PL(X)$ is the space of piecewise linear functions
   whose linearity locus is defined over $\Q$. The space  
   $\Conic(\widehat X)$ is the space of conical
   functions on $X$, while the space $\Conic(X)$ is the space of real
   valued conical functions on 
   $X (\Q)$. The arrows in 
   diagram \eqref{eq:12} are the obvious ones. In particular the
   upward arrow on the right of the diagram sends a conical function on
   $X$ to its restriction to $X(\Q)$.
 \end{rem}

The definition of the intersection product in the lattice case is
different from the Euclidean case, because of the change in the
definition of Minkowski weights and normal vectors. To avoid confusion
we will use a different symbol. Recall that $X$ denotes a
quasi-embedded rational conical space.

\begin{Def} \label{def:4rat} Let $\Pi \in R(X)$ a rational conical complex on $X$. Let $\phi \in \PL(\Pi) $ be a piecewise linear function and 
  $c \in M_k(\Pi)$ a Minkowski weight. Then the
  \emph{(lattice) tropical intersection product} $\phi \odot c\in
  M_{k-1}(\Pi )$
is the Minkowski weight given, for $\tau \in \Pi(k-1)$, by
\[
\left(\phi \odot c\right)(\tau) \coloneqq \sum_{\substack{\sigma \in
    \Pi(k)\\\tau \prec
    \sigma}}-\phi_{\sigma}\left(\tilde v_{\sigma/\tau}\right)c(\sigma) +
\phi_{\tau}\left(\sum_{\substack{\sigma \in \Pi(k)\\ \tau \prec
    \sigma}} c(\sigma)\tilde v_{\sigma/\tau}\right).
\]
Note that this is well defined since $c \in M_k(\Pi)$ is a
$k$-dimensional Minkowski weight and hence  
\[
\sum_{\substack{\sigma \in \Pi(k)\\ \tau \prec
    \sigma}} c(\sigma)\tilde
v_{\sigma/\tau} \in N^{\tau}_{\R}.
\] 
Moreover, if $\tilde v'_{\sigma /\tau }$ is another choice of
liftings, then $w_{\sigma /\tau }\coloneqq \tilde v_{\sigma /\tau
}-\tilde v'_{\sigma /\tau }\in N^{\tau }_{\R }$ and therefore
\begin{displaymath}
  \sum_{\substack{\sigma \in
    \Pi(k)\\\tau \prec
    \sigma}}-\phi_{\sigma}\left(w_{\sigma/\tau}\right)c(\sigma) +
\phi_{\tau}\left(\sum_{\substack{\sigma \in \Pi(k)\\ \tau \prec
    \sigma}} c(\sigma)w_{\sigma/\tau}\right)=0,
\end{displaymath}
so the intersection product is independent of the choice of liftings.
\end{Def}
As with Minkowski weights, lattice tropical cycles will be called just
tropical cycles and lattice tropical intersection will be
called tropical intersection.

As in the Euclidean case, the tropical intersection product
extends to a bilinear pairing between piecewise linear functions and
tropical cycles.

\begin{Def}
Let $z \in
Z_k(X)$ be a $k$-dimensional tropical cycle and let $\phi \in \PL(X)$. Let $\Pi \in R(X)$ be such that $z$ is represented by a $k$-dimensional
Minkowski weight $c\in
M_{k}\left(\Pi\right)$ and such that $\phi$ is defined on $\Pi$. Then the tropical intersection product $\phi
\odot z \in Z_{k-1}(\Pi)$ given by 
\[
\phi \odot z \coloneqq [\phi \odot c]
\]
is well defined.
\end{Def}

\begin{rem}
 If $X$ is balanced, then the definition of the \emph{tropical top
   intersection numbers} is the analogue of the Euclidean case
 (Definition \ref{def:tropintersectiondiv}) but using the (lattice)
 tropical product.

\end{rem}

We are now ready to relate Euclidean and lattice structures. 
As before, we denote by $\widehat{X}$ the Euclidean conical space
induced by $X$ by forgetting the rational structure and by choosing a
Euclidean metric $\langle \ ,\ \rangle $ on $N_{\R}^{X}$. Given $\Pi
\in R(X)$ we denote by $\widehat \Pi$ the induced Euclidean conical
complex on $\widehat X$. Note that if $\Pi$ is smooth, then $\widehat
\Pi$ is simplicial.  


We introduce the following notation. For $\Pi \in R(X)$ and for a cone $\sigma \in \Pi$ we let 
\[
  \vol(\sigma) \coloneqq \vol_{\langle\ ,\ \rangle}(N^{\sigma }_{\R}/N^{\sigma })=
  \sqrt{\det\left(\langle v_{i}, v_{j}\rangle \right)_{i,j}}
\]
where $\{v_{1},\dots,v_{k}\}$ is an integral basis of $N^{\sigma
}$. Note that $\vol(\sigma) $ depends on both, the
rational structure and the Euclidean one.

Recall that $W_{\ast}(\Pi )$ denotes the space of weights of $\Pi $. We define a map $\widehat{\phantom {A}}\colon W_{\ast}(\Pi )\to
W_{\ast}(\Pi )$ given by
\begin{displaymath}
  \widehat{c}(\sigma) \coloneqq \vol(\sigma) c(\sigma),
\end{displaymath}

\begin{lemma}\label{lem:eucl}
If $c \in M_k(\Pi)$ is a $k$-dimensional Minkowski weight on $\Pi$ then $\widehat{c}$
is a Euclidean Minkowski weight $\widehat{c} \in M_k(\widehat{\Pi})$.
\end{lemma}
\begin{proof}
Let $\tau \in \Pi(k-1)$. We have to show that $\widehat{c}$ is a
Euclidean Minkowski weight. 
For any $\sigma \in \Pi(k)$ containing $\tau$ let $\tilde
v_{\sigma/\tau}\in N^{\sigma }_{\R}$ be a lifting of the lattice
normal vector as in Definition~\ref{def:lattice-vectorst} and let
\[
 \tilde v_{\sigma/\tau} = v_{\sigma, \tau} + v_{\sigma, \tau^{\perp}}
  \]
  be an orthogonal decomposition of $\tilde v_{\sigma/\tau}$ with 
  $v_{\sigma, \tau} \in N^{\tau}_{\R}$ and $ v_{\sigma, \tau^{\perp}}$ 
  orthogonal to $N^{\tau}_{\R}$. The Euclidean normal vector
  $\hat v_{\sigma /\tau}$ of Definition~\ref{def:vectorst} is just the
  normalization of $v_{\sigma, \tau^{\perp}}$, i.e. we have $\hat
  v_{\sigma /\tau} = v_{\sigma, \tau^{\perp}}/\|v_{\sigma,
    \tau^{\perp}}\|$. If $\{v_{1},\dots,v_{k-1}\}$ is an integral
  basis of $N^{\tau }$, then $\{v_{1},\dots,v_{k-1},\tilde v_{\sigma
    /\tau }\}$ is a basis on $N^{\sigma }$. Therefore,
  \begin{equation}\label{eq:11}
   \|v_{\sigma, \tau^{\perp}}\| = \frac{\vol(\sigma)}{\vol(\tau)}. 
  \end{equation}
  
  We compute
  \begin{align*}
       \sum_{\substack{\sigma \in \Pi (k)\\\tau \prec \sigma
    }}\widehat c\left(\sigma\right)\hat v_{\sigma/\tau} &=
  \sum_{\substack{\sigma \in \Pi (k)\\\tau \prec \sigma 
    }}c \left(\sigma\right)\vol(\sigma)\hat v_{\sigma/\tau} \\
    & =
  \vol(\tau)\sum_{\substack{\sigma \in \Pi (k)\\\tau \prec \sigma
    }}c\left(\sigma\right) \tilde v_{\sigma,\tau^{\perp}} \\
    & =
  \vol(\tau)\sum_{\substack{\sigma \in \Pi (k)\\\tau \prec \sigma
    }}c\left(\sigma\right) \left( \tilde v_{\sigma/\tau}-v_{\sigma,\tau}\right) = 0,
\end{align*}
In the last equation we have used that, since $c$ is a Minkowski
weight then $\sum c(\sigma )\tilde v_{\sigma/\tau}$ belongs to
$N^{\tau }$, hence agrees with its orthogonal projection to $N^{\tau
}_{\R}$ which is $\sum c(\sigma )v_{\sigma,\tau}$. We deduce that
$\widehat{c}$ is a Euclidean Minkowski weight.
\end{proof}

\begin{Def} 
Let $\Pi \in R(X)$ and let $c \in M_k(\Pi)_{\R}$ be a $k$-dimensional Minkowski
weight. Then the Euclidean Minkowski weight $\widehat{c}\in 
M_{k}(\widehat {\Pi })$ is called the \emph{normalization} of $c$.
 The normalization $\widehat{z}$ of a tropical cycle $z \in Z_k(X)$
 is defined to be the class $[\widehat c] \in Z_k(\widehat X)$ of the
 normalization of any representative Minkowski weight $c \in
 M_k(\Pi)$ of $z$.  
\end{Def}

\begin{rem}\label{normalizedproduct}
If $c \in M_0(\Pi)$ is $0$-dimensional, then $\widehat{c}=c$. 
\end{rem}

The following proposition shows the compatibility between the tropical
intersection product and the Euclidean one, allowing us to replace
the integral structure by the Euclidean one in computations.

\begin{prop}\label{propnorm} Let $\Pi \in R(X)$ be a rational conical complex on $X$. Let $\phi \in \PL(X)$ be a piecewise linear function defined on $\Pi$ and let $c\in M_{k}(\Pi )$. Then 
\[
\widehat{\phi \odot c} = \phi \cdot \widehat{c}.
\]
Hence, also for a $k$-dimensional Minkowski cycle $z \in
Z_{k}(X )$ we have 
\[
\widehat{\phi \odot z} = \phi \cdot \widehat{z}.
\]
\end{prop}
\begin{proof}
Let $\tau \in \Pi(k-1)$. We use the same notation as in the proof of Lemma \ref{lem:eucl}. Since $c$ is a Minkowski weight, we have that
\begin{align}\label{eqnproj1}
\sum_{\substack{\sigma \in \Pi(k)\\\tau \prec \sigma}}c(\sigma)\tilde v_{\sigma / \tau} = \sum_{\substack{\sigma \in \Pi(k)\\\tau \prec \sigma}}c(\sigma) v_{\sigma, \tau}.
\end{align}
We compute, using equation \eqref{eq:11}, 
\begin{align*}
\widehat{\phi \odot c}(\tau)  & = \left( \sum_{\substack{\sigma \in
                                 \Pi(k)\\ \tau \prec \sigma}} -
  \phi_{\sigma}\left(\tilde v_{\sigma / \tau}\right)c(\sigma) +
  \phi_{\tau}\left(\sum_{\substack{\sigma \in \Pi(k)\\ \tau \prec \sigma}}c(\sigma)\tilde v_{\sigma / \tau} \right)\right) \cdot \vol(\tau) \\
  & =  
 \left( \sum_{\substack{\sigma \in
                                \Pi(k)\\ \tau \prec \sigma}} -
 \phi_{\sigma}\left(\tilde v_{\sigma / \tau}\right)c(\sigma) +
 \phi_{\tau}\left(\sum_{\substack{\sigma \in \Pi(k)\\ \tau \prec \sigma}}c(\sigma)v_{\sigma, \tau} \right) \right) \cdot \vol(\tau)  \\
 & =  
 \left( \sum_{\substack{\sigma \in
                                 \Pi(k)\\ \tau \prec \sigma}} -c(\sigma) \phi_{\sigma}\left(v_{\sigma,\tau^{\perp}}\right)  \right)\cdot  \vol(\tau)\\
                                &= \sum_{\substack{\sigma \in
                                \Pi(k)\\ \tau \prec \sigma}}
  -c(\sigma) \vol(\sigma)
  \phi_{\sigma}\left(v_{\sigma,\tau^{\perp}}\right) \cdot
  \|v_{\sigma,\tau^{\perp}}\|^{-1}\\
  &=  \sum_{\substack{\sigma \in
    \Pi(k)\\ \tau \prec \sigma}} -\widehat{c}(\sigma)
  \phi_{\sigma}\left(\hat v_{\sigma /\tau}\right) \\ 
  &= \phi \cdot \widehat{c}(\tau ),
\end{align*}
hence the first statement of the proposition follows. The second
statement clearly follows from the first.
\end{proof}

\subsection{The rational conical space attached to a toroidal embedding}

Throughout this section $k$ will denote an algebraically closed field
of characteristic~$0$.
All of the varieties appearing in this section
will be defined over $k$ even if not stated explicitly.     
We recall the definition of a toroidal embedding and describe its
associated rational conical space. The following definition is taken
from \cite[Definition 1, pg.~54]{KKMD}. 

\begin{Def}\label{deftoroidal}
  Let $X$ be an $n$-dimensional normal, algebraic variety
  over $k$ and let $U$ be a smooth Zariski open subset of $X$. An open
  immersion $U \hookrightarrow X$ is a \emph{toroidal embedding} if for
  every closed point $x \in X$ there exists an $n$-dimensional torus
  $\T$, an affine toric variety $X_{\sigma} \supseteq \T$, a point $x'
  \in X_{\sigma}$ and an isomorphism of $k$-local algebras 
  \begin{align}\label{localiso}
    \widehat{\O}_{X,x} \overset{\simeq}{\longrightarrow}
    \widehat{\O}_{X_{\sigma},x'}
  \end{align}
  such that the ideal in $\widehat{\O}_{X,x}$ generated by the ideal of
  $X \setminus U$ corresponds under this isomorphism to the ideal in
  $\widehat{\O}_{X_{\sigma},x'}$ generated by the ideal of $X_{\sigma}
  \setminus \T$. Here, the hat \enquote{~$\widehat{}$~} denotes the
  completion of the local ring at a point. Such an isomorphism is
  called a \emph{chart} at $x$ and the pair $(X_{\sigma}, x')$ is
  called a \emph{local model} at $x$.
  
  If all the irreducible components of the boundary divisor
  $X \setminus U$ of a toroidal embedding are normal, then it is called
  a \emph{toroidal embedding without self intersection}.
\end{Def}

\begin{Def}\label{strata}
Let $U \hookrightarrow X$ be a toroidal embedding (defined over $k$) without self
intersection and let $\{B_i \,|  
\, i \in I\}$ be the irreducible components of the boundary divisor $B
= X \setminus U$. For every subset $J \subseteq I$, write $B_J
\coloneqq \bigcap_{i \in J}B_i \neq \emptyset$. The \emph{strata} of
the toroidal embedding are the irreducible components of the sets of
the form $B_J
\setminus \bigcup_{i \notin J}B_i$. The strata will be denoted 
$\{S_{\alpha }\}_{\alpha \in \Lambda }$ where $\Lambda $ is a finite
set. The maximal strata correspond to the irreducible components of
the open set $U$.\end{Def}

The following lemma is \cite[Proposition-Definition~2, pg.~57]{KKMD}.
\begin{lemma}
Let notations be as above and consider a subset $J \subseteq I$ such
that $B_{J}\not = \emptyset$ and let $S_{\alpha _{0}}$ be an irreducible
component of  $B_J\setminus \bigcup_{i \notin J}B_i$. Then the
following holds true:
\begin{enumerate}
\item $B_{J}$ is normal.
\item $S_{\alpha_{0} }$ is non-singular.
\end{enumerate} 
Moreover, the sets $S_{\alpha }$, $\alpha \in \Lambda $ define a
stratification of $X$,
i.e.~every point of $X$ is in exactly one stratum and the closure of a
stratum is a union of strata. Furthermore, if $x \in X$ and
$\left(X_{\sigma},x'\right)$ is a local model at $x$, then the
closures $\overline{S_{\alpha }}$ of the strata $S_{\alpha }$ such
that $x \in \overline{S_{\alpha }}$ correspond formally to the
closure of the torus orbits 
in $X_{\sigma}$ containing $x'$. In particular, if $x \in S_{\alpha }$, then
$S_{\alpha }$ corresponds formally to the torus orbit $O(x')$ itself.  
\end{lemma}
The following Proposition/Definition is adapted from \cite[Definition
3, pg.~59]{KKMD}. See also Corollary 1 in page 61 of \cite{KKMD}.

\begin{prop/Def}\label{propdef:strata}
Let notations be as in Definition \ref{strata}. For any non-empty
stratum $S_{\alpha }$ of the toroidal embedding $U \hookrightarrow
X$, the combinatorial open set $\Star(S_{\alpha }) \subseteq X$ is defined by
\[
\Star(S_{\alpha })\coloneqq \bigcup_{\beta \colon S_{\alpha }
  \subseteq \overline{S}_{\beta }}S_{\beta } = X \setminus
\bigcup_{\gamma \colon \overline{S}_{\gamma } \cap S_{\alpha } =
  \emptyset} S_{\gamma }.
\]
Moreover, let 
\begin{align*}
&M^{S_\alpha } \coloneqq \left\{B \in \Ca\text{-}\Div \left(\Star
                 \left(S_\alpha \right)\right) \,\big{|}\, \supp(B)
                 \subseteq \Star \left(S_\alpha \right) \setminus U
                 \right\},\\
&M_+^{S_{\alpha }} \coloneqq \left\{ B \in M^{S_\alpha } \,\big{|}\, B \text{
                               effective} \right\}. 
\end{align*}
Then $M^{S_\alpha }$ is a free abelian group (a lattice) while
$M_+^{S_\alpha }$ has the structure of a sub-semigroup. For each stratum
$S_\alpha $ we denote by $N^{S_\alpha } = \left(M^{S_\alpha }\right)^{\vee}$ the dual
lattice of $M^{S_\alpha }$ and by $\langle \phantom{x},\phantom{x}
\rangle_{S^\alpha }$ the induced pairing. 
Finally, let
\[
\sigma^{S_\alpha } \coloneqq \left\{ v \in N_{\R}^{S_\alpha }
  \,\big{|}\, \langle m, v \rangle_{S^\alpha } \geq 0, \,  \forall m
  \in M_+^{S_\alpha } \right\} \subseteq N_{\R}^{S_\alpha }.
\]
Then $\sigma^{S_\alpha } \subseteq N_{\R}^{S_\alpha }$ is a strongly
convex rational polyhedral cone of maximal dimension.  
\end{prop/Def}
The idea behind Proposition/Definition \ref{propdef:strata} is that
given a stratum $S$, we have produced a maximal dimensional cone
$\sigma^S$ in the
finite-dimensional real vector space $N_{\R}^ S$ which comes equipped
with a canonical lattice $N^S$. 

We now see that
these cones can be glued together into a rational conical complex.
For a toroidal embedding $U \hookrightarrow X$ without self
intersection, let $|\Pi_{(X,U)}|$ be the
quotient topological space defined by
\[
  \left|\Pi_{(X,U)}\right| \coloneqq \bigsqcup_{S_\alpha  \operatorname{stratum}}
  \sigma^{S_\alpha } / \sim  
\]
where $\; \sim \;$ is the equivalence relation generated by isomorphisms 
\[
\beta^{\alpha ,\alpha '} \colon \sigma^{S_\alpha } \overset{\simeq}{\longrightarrow} \text{ face of } \sigma^{S_{\alpha '}}
\]
whenever $S_\alpha  \subseteq \Star\left(S_{\alpha '}\right)$. Here,
the map $\beta^{\alpha ,\alpha '}$ is the restriction of the map
$N_{\R}^{S_\alpha } \to N_{\R}^{S_{\alpha '}}$ defined as the dual of
the map $M_{\R}^{S_{\alpha '}} \to M_{\R}^{S_\alpha }$, which in turn
is induced by the map $M^{S_{\alpha '}} \to M^{S_\alpha }$ given by
restricting divisors from $\Star\left(S_{\alpha '}\right)$ to
$\Star\left(S_{\alpha }\right)$ (see \cite[Chapter II, Section
1]{KKMD}). We have the following proposition.
\begin{prop}\label{prop:complex-toroidal} If $U \hookrightarrow X$ is
  a toroidal embedding without self
intersection,  then the pair 
\[
\Pi_{(X,U)} = \left(\left\{ \sigma^{S_\alpha }
  \right\}_{S_\alpha  \operatorname{stratum}}, \left\{ M^{S_\alpha }
  \right\}_{S_\alpha  \operatorname{stratum}} \right)
\]
defines a rational conical structure on $\left|\Pi_{(X,U)}\right|$ in the sense of Definition
\ref{conicalcomplex}. Hence, $\left|\Pi_{(X,U)}\right|$ is a rational conical space in the sense of Definition \ref{def:rat-space} with rational conical structure given by $\Pi_{(X,U)}$. 
\end{prop}
\begin{proof}
The proof can be found in \cite[Chapter II, pg.~71]{KKMD}.
\end{proof}
The collection of lattices $\left\{M^{S_\alpha }\right\}$ in the above
proposition gives the \emph{integral structure} of the toroidal
embedding.

The following lemma follows from \cite[Chapter II, Corollary 1]{KKMD}. 
\begin{lemma}\label{stratacones}
Let $U \hookrightarrow X$ be a toroidal embedding without self
intersection and let $x \in X$
belonging to a stratum $S$. If $(X_{\sigma}, x') $ is a local model at
$x$ then
\[
M^S \simeq M(\T) / \left(M(\T) \cap \sigma^{\perp}\right) \quad \text{ and }\quad \, \sigma^S \simeq \sigma \; , 
\]
where $M(\T)$ refers to the lattice of characters of the torus $\T \subseteq X_{\sigma}$ and $\sigma^{\perp}$ is the set defined by 
\[
\sigma^{\perp} \coloneqq \{m \in M(\T)\, | \, \langle m,v\rangle = 0, \,  \forall v \in \sigma \}.
\]
In particular, the local model $(X_{\sigma}, x') $ is determined up to
isomorphism by the stratum $S$. 
\end{lemma}

Given a cone $\sigma$ in $\Pi_{(X,U)}$, we will denote by $S^{\sigma}$ the
stratum corresponding to $\sigma$ and by $\overline{S^{\sigma}}$ its
closure in $X$. 
\begin{exa}\label{toricexample}
Let $\Sigma$ be a fan in $N_{\R}$ for some lattice $N$ and let $M
\coloneqq N^{\vee}$ be its dual lattice. Furthermore, let $\X$ be its
associated normal toric variety over $k$ with dense torus $\T =
\Spec(k[M])$. Clearly, the inclusion $\T \hookrightarrow \X$ defines a
toroidal embedding. The components of the boundary divisor $B=\X
\setminus \T$ are the $\T$-invariant prime divisors $B_{\tau}$
corresponding to the rays $\tau \in \Sigma(1)$, and the strata of $X$
are the $\T$-orbits $O(\sigma)$ corresponding to the cones $\sigma \in
\Sigma$. The combinatorial open sets of $\X$ are precisely its
$\T$-invariant affine open subsets.

The isomorphism 
\[
M / \left(M \cap \sigma^{\perp}\right) \simeq M^{O(\sigma)}
\]
given by the assignment 
\[
[m] \longmapsto \div\left(\chi^m\right),
\]
 where $\chi^m$ denotes the character of the torus associated to $m \in M$, induces an identification of lattices 
 \[
 N^{O(\sigma)}\simeq N_{\sigma} = N \cap \Span (\sigma) 
 \] 
 and of cones 
 \[
 \sigma^{O(\sigma)} \simeq \sigma.
 \] 
\end{exa}
\begin{rem}\label{rem:rays-boundary}
As in the toric case, the set of rays $\Pi _{(X,U)}(1)$ of the rational
conical complex associated to a toroidal embedding $U
\hookrightarrow X$ is in bijection with the set of irreducible
components of the boundary divisor $B= X \setminus U$. Indeed for
every irreducible component $B_i$, the corresponding ray in
$\Pi_{(X,U)}(1)$, which we will denote by $\tau_{B_i}$ is the linear
function $\tau_{B_i} \colon M^{S_{\{i\}}} \to \Z$ given by $n B_i
\mapsto n$. Conversely, one can show that any ray $\tau \in \Pi_{(X,U)}(1)$
arises in this way (see \cite[pg.~63]{KKMD}). For any such ray $\tau$,
we will denote by $B_{\tau}$ the corresponding irreducible boundary
component.
\end{rem}
Before giving a more general class of examples of toroidal embeddings,
we recall some definitions.
\begin{Def}\label{def:sncdivisor}
Let $B \subseteq X$ be a divisor on a smooth variety $X$. We say that
$B$ is a \emph{normal crossing divisor} (abbreviated nc) if
the following condition hold:
\begin{enumerate}
\item For all $x \in X$ we can choose local coordinates $x_1, \dotsc ,
  x_n$ and natural numbers $\ell_1, \dotsc , \ell_n$ such that $B =
  \left\{\prod_ix_i^{\ell_i} = 0\right\}$ in a neighborhood of $x$. 
\end{enumerate} 
We say that $B$ is a \emph{simple normal crossing divisor}
(abbreviated snc) if furthermore
\begin{enumerate}[resume]
\item\label{sncversusnc} Every irreducible component of $B$ is smooth. 
\end{enumerate}
\end{Def}

We can now give a large class of examples of toroidal varieties. 
\begin{exa}\label{exa:toroidalnc}
Let $(X, B)$ be a pair consisting of a smooth projective variety $X$
of dimension $n$ together with a snc divisor $B \subseteq X$.
We denote by $\{B_i\}_{i
  \in I}$ the irreducible components of $B$.
 Set $U \coloneqq X
\setminus B$. Then $U \hookrightarrow X$ is a toroidal embedding. The
rational conical complex associated to the toroidal embedding $U
\hookrightarrow X$  is smooth and is constructed by adding a
$k$-dimensional cone for each subset $J\subseteq I$ with $\#J=k$ and
each irreducible component of $\bigcap_{j\in J}B_{j}$. In particular,
the zero dimensional cones correspond to the irreducible components
of $X=\bigcap_{j\in \emptyset} B_{j}$.
\end{exa}
\begin{rem}
It follows from the definition of a toroidal embedding $U
\hookrightarrow X$ that the boundary $X \setminus U$ is a divisor,
however, it may not be snc. Nevertheless, by Hironaka's resolution of
singularities \cite{HIR},  we can always find an allowable
modification of $X' \to X$ (Definition~\ref{def:allowablemodification}) such that the boundary divisor $X'
\setminus U$ is $\snc$.
\end{rem}

\subsection{From weak embeddings to quasi embeddings}
\label{sec:from-weak-embeddings}

Following \cite{GR}, when the ambient variety is proper, there is a
natural weak embedding of the rational
conical complex associated to a toroidal embedding without self 
intersection. We show in this section that, in the projective case, by
adding boundary
components, one can
modify the toroidal structure of a toroidal embedding in such a way
that the rational conical complex becomes quasi-embedded. 

\begin{Def}\label{def:weak-toroidal} Let $U \hookrightarrow X$ be a
  toroidal embedding with $X$ proper and let $|\Pi| =|\Pi _{(X,U)}|$ be the
  corresponding rational conical space. The group $M^{|\Pi|}$ is defined
  to be the set of classes of invertible regular functions on the open
  set $U$, modulo locally constant functions, i.e.
\[
M^{|\Pi|} \coloneqq \Gamma\left(U,
  \O_X^{\times}\right)/\Gamma\left(U,k^{\times}\right).
\]
Since $X$ is proper, this is a torsion free finitely generated abelian
group. That is, it is a lattice. 
Let $N^{|\Pi|} \coloneqq \left(M^{|\Pi|}\right)^{\vee}$ be its dual
lattice. For every stratum $S$ of $X$ we have a morphism of lattices
$M^{|\Pi|} \to M^S$ given by 
\[
f \longmapsto \div(f)|_{\Star(S)}. 
\]
Dualizing, we get a linear map $\sigma^S \to N^{|\Pi|}_{\R}$. These
maps glue to give a continuous function  
\[
\iota_{|\Pi|} \colon \left|\Pi\right| \longrightarrow N^{|\Pi|}_{\R},
\] 
which is integral linear on the cones of $\Pi$, i.e.~$|\Pi|$ has structure of a
weakly embedded rational conical space.
\end{Def}

Two of the following examples are taken from \cite[Example 2.2]{GR}.
\begin{exa}\
  \begin{enumerate}
  \item Consider the toric setting $X = \X$ from Example
    \ref{toricexample}, and write $\Pi =\Pi _{(\X,\T)}$. Here, we have
    the lattice $M^{|\Pi|} =
    \Gamma\left(\T, \O_{\X}^{\times}\right)/k^{\times}$, which we can
    identify with $M$ via the isomorphism $M \simeq M^{|\Pi|}$
    given by the assignment
\[
m \longmapsto \chi^m.
\]
We see that the image of $\sigma^{O(\sigma)}$ in $N_{\R}$ under the
weak embedding $\iota_{\Pi}$ is precisely $\sigma$. Hence,
$\Pi$ is a weakly embedded rational conical complex, naturally
isomorphic to $\Sigma$. Note that in this case, the weak embedding is
globally injective. 
\item  For a non-toric example, consider $X = \P^2$ with homogeneous 
  coordinates $(x_0: x_1: x_2)$ but with open part $U$ given by
\[
U = X \setminus \left(H_1 \cup H_2\right),
\] 
where $H_i$ is the hyperplane given by  $\{x_i = 0 \}$. This is a
toroidal embedding with snc boundary divisor and we see that the
rational conical complex $\Pi=\Pi_{(X,U)}$ is naturally identified with the
non-negative orthant $\R^2_{\geq 0}$, whose rays $\R_{\geq 0}(1,0)$
and  $\R_{\geq 0}(0,1)$ correspond to the divisors $H_1$ and $H_2$,
respectively. The lattice $M^{|\Pi|}$ is generated by $x_1/x_2$, and
using that generator to identify $M^{|\Pi|}$ with $\Z$, we see that
the weak embedding $\iota_{\Pi}$ sends $(1,0)$ to $1$ and $(0,1)$ to
$-1$. Note that in this case $\iota_{\Pi}$ is not a quasi-embedding
since for example the cone $\R^2_{\geq 0}$ is two-dimensional while
$N^{|\Pi|}_{\R}$ has dimension one.
\item Consider again $\P^{2}$ with the same homogeneous coordinates,
  $D_{1}$ the line $x_{0}=0$ and $D_{2}$ the conic
  $x_{0}^{2}+x_{1}^{2}+x_{2}^{2}=0$. then $D_{1}\cup D_{2}$ is a snc
  divisor and the corresponding conical complex consist of two copies
  of the non-negative orthant glued together by the axes. The
  description of the quasi-embedding is similar to the previous one.  
  \end{enumerate}
\end{exa}

The following key proposition says that given a toroidal embedding
with a snc boundary divisor, we can always modify the toroidal structure
in such a way that the associated weakly embedded rational conical
complex becomes quasi-embedded. 
\begin{prop}\label{prop:changedivisor}
Let $U \hookrightarrow X$ be a toroidal embedding with $X$ smooth and
projective, 
such that the 
boundary divisor $B = X\setminus U$ is snc and let $|\Pi| = |\Pi_{(X,U)}|$ be
its associated weakly embedded rational conical space. Then there
exists a snc divisor $B'$ with $|B| \subseteq |B'|$ such that, writing
$U'=X\setminus B'$, 
the
weakly embedded rational conical space $|\Pi'| = |\Pi_{(X,U')}|$ is
quasi-embedded, i.e.~the restriction 
of the weak
embedding
\[
\iota_{|\Pi'|}|_{\sigma'} \colon |\sigma'| \longrightarrow N_{\R}^{|\Pi'|}
\] 
to any cone $\sigma' \in \Pi'$ is injective.
\end{prop}
\begin{proof}
  Recall that $n$ denotes the dimension of $X$.
  If $n=1$ then $B=\{p_{1},\dots,p_{k}\}$ is a finite set of
  points. Choose rational functions $f_{i}$ such that
  $\ord_{p_{i}}(f_{i})\not =0$ and write
  \begin{displaymath}
    B'=\bigcup |\div(f_{i})|=\{q_{1},\dots,q_{r}\}
  \end{displaymath}
  The corresponding polyhedral complex is given by the finite set of
  rays $\{\tau _{q_{j}}\}$, such that $\tau _{q_{j}}$ is joined with
  $\tau _{q_{j'}}$ at zero if and only if $q_{j}$ and $q_{j'}$ are in
  the same irreducible component.  Let $v_{j}$ denote the
  primitive vector of $\tau _{q_{j}}$ and let $x_{i}$ denote the point
  of $M^{|\Pi'|}$ corresponding to $f_{i}$. By construction, for each $j$
  there is an $i$ such that $\langle
  \iota(v_{j}),x_{i}\rangle=\ord_{q_{j}}(f_{i})\not = 0$. Therefore
  $\iota _{\Pi '}(v_{j})\not = 0$ and $\Pi '$ is quasi-embedded.

  Assume now that $n\ge 2$. Write $B=B_{1}\cup \dots \cup B_{r}$ for
  the decomposition of $B$ into irreducible components. There is a
  hypersurface $C$ such that $B_{i}+C$
  is very ample for $i=1,\dots , r$. Moreover we can find
  hypersurfaces $A_{i,j}$, 
  \begin{displaymath}
    B_{i}\sim A_{i,j}-C,\quad i=1,\dots r,\ j=1,\dots,n,
  \end{displaymath}
  and a second hypersurface $C_{1}\not = C$ such that
  $C_{1}\sim C$. Here the symbol $\,\sim\,$ means linear
  equivalence. Finally by Bertini's theorem we can assume that all 
  the hypersurfaces $C$, $C_{1}$ and $A_{i,j}$ are different, smooth
  and irreducible and
  \begin{displaymath}
    B'\coloneqq B\cup C\cup C_{1}\cup \bigcup_{i,j}A_{i,j}
  \end{displaymath}
  is a snc. Then there are rational functions $f_{i,j}$ and $g$ such
  that
  \begin{align*}
    \div(f_{i,j})&=B_{i}+C-A_{i,j}\\
    \div(g)&=C-C_{1}.
  \end{align*}
  As in the statement of the theorem, write $U'=X\setminus B'$ and
  $|\Pi '|=|\Pi _{(X,U')}|$. 
  Let $x_{i,j}$ be the point of $M^{|\Pi '|}$ corresponding to $f_{i,j}$
  and $y$ the point corresponding to $g$. For an irreducible component
  $E$ of $B'$, write $v_{E}$ for the primitive generator of the ray
  corresponding to the the divisor $E$. By  construction we have
  \begin{align*}
    \langle v_{B_{i}}, x_{k,j}\rangle & = \delta _{i,k}\\
    \langle v_{A_{i,j}}, x_{k,\ell}\rangle & = -\delta _{i,k}\delta
                                             _{j,\ell}\\
    \langle v_{B_{i}}, y\rangle & = \langle v_{A_{i,j}}, y\rangle=0\\
    \langle v_{C}, x_{k,j}\rangle & = 1\\
    \langle v_{C_{1}}, x_{k,j}\rangle & = 0\\
    \langle v_{C}, y\rangle & = 1\\
    \langle v_{C_{1}}, x_{k,j}\rangle & = -1
  \end{align*}
From the above identities, it follows that any subset of $n$ vectors
contained in $\{v_{B_{i}},v_{A_{k,j}},v_{C},v_{C_{1}}\}$ is linearly
independent. This implies that the weak embedding $\iota _{\Pi '}$ is
a quasi-embedding.
\end{proof}

\begin{rem}\label{rem:ample-effective}
  Let $U \hookrightarrow X$ be a toroidal embedding with
    snc boundary divisor $B = X \setminus U$ and $X$ 
    projective. Then, by modifying the toroidal structure in a similar
    way as we did in the proof of Proposition \ref{prop:changedivisor},
    we may assume that there exists an ample and effective divisor
    with support contained in $B$.
    Moreover, by adding a small multiple of
    all components of $B$, we may assume that there is an ample
    $\R$-divisor $A$
    with positive
    multiplicity along all components of $B$. This will be useful for
    the monotone approximation lemma in
    Section~\ref{sec:hilb-samu-form}.
\end{rem}

\subsection{Toroidal modifications and subdivisions of rational
  conical complexes}
Recall from the classical theory of toric varieties that given a toric
variety $\X$ corresponding to a fan $\Sigma$, there is a bijective
correspondence between proper birational toric morphisms to $\X$ and
subdivisions of the fan $\Sigma$. Following \cite[Chapter 2,
Section 2]{KKMD}, a similar
phenomenon occurs in the toroidal case. In this section we describe
the proper toroidal birational modifications of a toroidal embedding
which, on the combinatorial side, correspond to subdivisions of the
associated rational conical complex.

 \begin{Def} 
 Let $U_{X_1} \hookrightarrow X_1$ and $U_{X_2} \hookrightarrow X_2$ be two toroidal embeddings and let $f \colon X_1 \to X_2$ be a birational morphism mapping $U_{X_1}$ to $U_{X_2}$. Then $f$ is called \emph{toroidal} if for every closed point $x_1 \in X_1$ there exist local models $(X_{\sigma_1},x'_1)$ at $x_1 \in X_1$ and $(X_{\sigma_2},x'_2)$ at $f(x_1) \in X_2$, and a toric morphism $g \colon X_{\sigma_1} \to X_{\sigma_2}$, with $f(x'_1) = x'_2$, such that the following diagram commutes.
 \begin{center}
    \begin{tikzpicture}
      \matrix[dmatrix] (m)
      {
       \widehat{\O}_{X_1,x_1} &   \widehat{\O}_{X_{\sigma_1},x'_1}\\
       \widehat{\O}_{X_2,x_2} &\widehat{\O}_{X_{\sigma_2},x'_2}\\
      };
      \draw[->] (m-1-1) to node[above] { $\simeq$ } (m-1-2);
      \draw[->] (m-2-1) to node[left]{${\hat{f}}^{\#}$}(m-1-1);
     \draw[->] (m-2-1) to node[above] { $\simeq$ } (m-2-2); 
     \draw[->] (m-2-2) to node[right]{$\hat{g}^{\#}$}  (m-1-2);
     \end{tikzpicture}
   \end{center}
Here, $\hat{f}^{\#}$ and $\hat{g}^{\#}$ are the ring homomorphisms induced by $f$ and $g$, respectively. 
\end{Def}
\begin{rem} 
The following two properties are satisfied.
\begin{enumerate}
\item The composition of two birational toroidal morphisms is again a
  birational toroidal morphism.
\item A toroidal morphism
  $f \colon (U_1 \hookrightarrow X_1) \to (U_2 \hookrightarrow
  X_2)$ induces a morphism
  \begin{displaymath}
f_{|\Pi|} \colon |\Pi_{(X_1,U_{1})}| \to
  |\Pi_{(X_2,U_{2})}|
\end{displaymath}
of 
  rational conical spaces. The restrictions of $f_{|\Pi|}$ to the cones
  of $\Pi_{(X_1,U_{1})}$ are dual to pulling back Cartier divisors. From this,
  we see that $f_{|\Pi|}$ can also be considered as a morphism between
  weakly embedded rational conical spaces by adding to it the data
  of the linear map $N^{|\Pi_{(X_1,U_{1})}|} \to N^{|\Pi_{(X_2,U_{2})}|}$ dual to the
  pullback 
  $\Gamma\left(U_{2}, \O_{X_2}^{\times}\right) \to
  \Gamma\left(U_{1}, \O_{X_1}^{\times}\right)$.
\end{enumerate}
\end{rem}
The following definition is taken from \cite[Definition 1, pg.~73]{KKMD}.
\begin{Def}
A toroidal birational morphism $f \colon (U \hookrightarrow Y) \to (U
\hookrightarrow X)$ between two toroidal embeddings of the same open
subset $U$ is called \emph{canonical over $X$} if the following conditions
hold true:
\begin{enumerate}
\item The diagram 
\begin{center}
\begin{tikzpicture}
\matrix[dmatrix] (m)
{
Y & &X \\
&U& \\
};
\draw[->] (m-1-1) to node[above] {$f$} (m-1-3);
\draw[left hook->] (m-2-2) to (m-1-1);
\draw[right hook->]  (m-2-2) to (m-1-3);
\end{tikzpicture}
   \end{center}
   is commutative.
\item For all $x_1, x_2 \in X$ in the same stratum $S$ and for all morphisms
\begin{align}\label{eqn:can}
\xi \colon \hat{\O}_{X,x_1} \longrightarrow \hat{\O}_{X,x_2}
\end{align}
which preserve the strata (i.e.~if $S \subseteq \overline{S'}$ for some stratum $S'$ then $\xi$ takes the ideal of $\overline{S'}$ at $x_1$ to the ideal of $\overline{S'}$ at $x_2$), we have that $\text{Spec}(\xi)$ can be lifted to give an isomorphism $Y \times_X \text{Spec}(\hat{\O}_{X,x_2}) \simeq   Y \times_X \text{Spec}(\hat{\O}_{X, x_1})$ preserving the strata, i.e.~such that the following diagram commutes.
 \begin{center}
    \begin{tikzpicture}
      \matrix[dmatrix] (m)
      {
        Y \times_X \text{Spec}(\hat{\O}_{X,x_2}) &   Y \times_X \text{Spec}(\hat{\O}_{X, x_1})\\
       \text{Spec}(\hat{\O}_{X,x_2}) &\text{Spec}(\hat{\O}_{X, x_1})\\
      };
      \draw[->] (m-1-1) to node[above] { $\simeq$ } (m-1-2);
      \draw[->] (m-1-1) to (m-2-1);
     \draw[->] (m-2-1) to node[above] { $\text{Spec}(\xi)$ } (m-2-2); 
     \draw[->] (m-1-2) to  (m-2-2);
     \end{tikzpicture}
   \end{center}
   \end{enumerate}
\end{Def}
We can now define the class of toroidal birational morphisms which correspond to subdivisions of rational conical complexes. 
The following is \cite[Definition 3, pg.~87]{KKMD}. 
\begin{Def}\label{def:allowablemodification} Consider a toroidal birational morphism $ f \colon (U \hookrightarrow Y) \to (U \hookrightarrow X)$ forming a commutative diagram 
\begin{center}
\begin{tikzpicture}
\matrix[dmatrix] (m)
{
Y & &X \\
&U& \\
};
\draw[->] (m-1-1) to node[above] {$f$} (m-1-3);
\draw[left hook->] (m-2-2) to (m-1-1);
\draw[right hook->]  (m-2-2) to (m-1-3);
\end{tikzpicture}
   \end{center}
 and satisfying the following two conditions: 
\begin{enumerate}
\item $Y$ has an open covering $\{V_i\}$ such that $U \subseteq V_i$, $f(V_i) \subseteq \text{Star}(S_i)$ for some stratum $S_i$ of $X$ and $V_i$ is affine and canonical over $\text{Star}(S_i)$.
\item $Y$ is normal.
\end{enumerate}
Toroidal embeddings $U \hookrightarrow Y$ as above are called
\emph{allowable modifications} of the toroidal embedding $U
\hookrightarrow X$.  
\end{Def}
The following important theorem follows from \cite[Theorem 6*, 8*]{KKMD}.
\begin{theorem}\label{the:allowablemodification} Given a toroidal
  embedding without self intersection $U \hookrightarrow X$, there is
  a bijective correspondence between subdivisions of the rational
  conical complex $\Pi_{X,U}$ and isomorphism classes of proper
  allowable modifications of $X$.
\end{theorem}

\section{Intersection theory of toroidal b-divisors}\label{sec:inters-theory-toro}

In this section $k$ still denotes an algebraically closed field of
characteristic zero and $U \hookrightarrow X$ will denote a toroidal
embedding with $X$ a smooth and projective $n$-dimensional
$k$-variety with snc boundary divisor
$B=X\setminus U$. We further assume that the corresponding
weakly embedded rational conical space $|\Pi| \coloneqq |\Pi _{(X,U)}|$
is quasi-embedded and that there is an effective ample $\R$-divisor
$A$ with support
$B$. By Chow's lemma, resolution of singularities, Proposition
\ref{prop:changedivisor} and Remark \ref{rem:ample-effective}, if we
start with a proper variety $X'$, an open dense subset $U'\subseteq X'$
and $B' = X'\setminus U'$, we can always find $X$, $U$ and $B$ as above
with a birational map $\pi \colon X\to X'$ such that $\pi
^{-1}(B')\subseteq B$. Therefore for many questions, the above
assumptions are harmless.

The goal of this section is to show that nef
toroidal b-divisors have well defined top intersection products
(Definitions \ref{def:toroidalbdivisor}
and \ref{def:nefbdiv} and Theorem
\ref{thm:intersectionbdivisors}) and that these intersection products
can be computed as the integral of a function associated to one of the
b-divisors with respect to a Monge--Amp\`ere like measure associated
to the remaining b-divisors. The existence of the product is also
proved in \cite{DF:bdiv} is a more general setting.

The idea of the construction is,
first, following
\cite{GR}, to relate the geometric intersection product of toroidal
divisors with the \emph{rational tropical intersection product} on quasi-embedded
rational conical complexes (Theorem~\ref{the:trop-alg-numbers}). Second to use the
convergence results of Section \ref{sec:monge-ampere-meas} in order to
extend the top intersection product to nef b-divisors. However, note that the
Monge--Amp\`ere measures of Section \ref{sec:monge-ampere-meas} are
defined in a Euclidean setting (no integral structure). Therefore we
will use the comparison in section \ref{sec:bridge} to relate the
rational tropical intersection product with the 
Euclidean one.

\subsection{Intersection products of toroidal
  divisors}\label{sec:int-prod-tor}

We will give the definition of $\R$-toroidal divisors and give
a bijection between the set of $\R$-toroidal divisors on $(X,U)$ and the set of
piecewise linear functions on 
$\Pi$. Moreover, following \cite{GR}, we recall the tropicalization
of an algebraic cycle and relate algebraic and tropical intersection
numbers. We end this section by showing that one can
compute tropically the top intersection numbers of divisors.  
 
\begin{Def} \label{def:6}
  Let $\Div(X)_{\R}$ be the vector space of $\R$-Cartier divisors on
  $X$. We define the subspace
  $\Div(X,U)_{\R} \subseteq \text{Div}(X)_{\R}$ consisting of
  $\R$-Cartier divisors which are supported on the boundary
  $B=X \setminus U$.  It is a finite dimensional $\R$-vector space and
  it is endowed with a canonical topology.  Elements in
  $\Div(X,U)_{\R}$ are called \emph{$\R$-toroidal Cartier divisors
    (of $(X,U)$)}. From now on we will only consider
  $\R$-divisors. Thus to
  simplify notation, we will omit the coefficient ring $\R$
  from the notation and call $\R$-toroidal Cartier divisors simply
  toroidal Cartier divisors.
\end{Def}

 Recall from Remark \ref{rem:rays-boundary}, that we have a bijective
 correspondence between the set of rays of $\Pi$ and the set of
 irreducible components of the boundary divisor $B = X \setminus U$.
 For a ray $\tau \in \Pi(1)$ we denote by $B_{\tau }$ the
 corresponding component and by $v_{\tau} =
v_{\tau/0_{\tau }}$ the primitive lattice normal vector spanning the
ray $\tau$. 
\begin{Def}\label{def:functiontoroidalsheaf}
Let $D \in \Div(X,U)$ be a toroidal Cartier divisor on $X$. The
corresponding piecewise linear function  
\[
\phi_D \colon |\Pi| \longrightarrow \R
\]
defined on $\Pi$ is given by
\begin{displaymath}
  \phi_D|_{\sigma }=-D|_{\Star(S^{\sigma })}\in M^{\sigma }=(N^{\sigma
  })^{\vee},
\end{displaymath}
where, for a cone $\sigma \in \Pi$, we denote by $S^{\sigma }$
the corresponding stratum of $X$. Since we are assuming that the
toroidal embedding is smooth, we can give an alternative description
going through conical functions. The function $\phi_{D}$ is linear on each
cone and, for $\tau \in \Pi (1)$,
\[
\phi _{D}(v_{\tau})= -\ord _{B_{\tau }}(D),
\]
where $\ord _{B_{\tau }}(D)$ denotes the coefficient of $B_{\tau}$ in $D$.
\end{Def}

By Remark \ref{rem:rays-boundary}, any piecewise linear function defined on $\phi$ on
$\Pi$ induces a toroidal Cartier divisor $D_{\phi}$ by setting
$D|_{B_{\tau}} = -\phi(v_{\tau})$ for any $\tau \in \Pi(1)$. These
constructions are clearly inverses of each other. 
We summarize the above in the following proposition, which
can be seen as a special case of \cite[Theorem 9*]{KKMD}.

\begin{prop}\label{prop:functiontoroidaldivisor}
The map 
\begin{align}\label{toroidaldivisorscomplex}
\Div(X,U) \longrightarrow \PL(\Pi)
\end{align}
given by the assignment
\[
D \longmapsto \phi_D
\]
is an isomorphism of finite dimensional real vector spaces.
\end{prop}

We recall the definition of the tropicalization of an algebraic cycle
class on $X$ as is explained in \cite[Section 4.2]{GR}.

For $0 \leq k \leq n$ we denote by $Z_k(X)= Z_k(X)_{\R}$ the group of
algebraic $k$-cycles on $X$ with real coefficients. For any $C \in
Z_k(X)$, the assignment
\[
\trop(C ) \colon \Pi(k) \longrightarrow \R,
\] 
given by
\[
\sigma \longmapsto \operatorname{deg}\left(C  \cdot
  \left[\overline{S^{\sigma}}\right] \right), 
\] 
is a $k$-dimensional Minkowski weight on $\Pi$. Moreover, this Minkowski
weight is compatible with taking refinements and thus the
following definition makes sense.
\begin{Def} Let notations be as above. The map 
\[
\trop \colon Z_k(X) \longrightarrow Z_k(|\Pi|) 
\]
given by 
\[
C \longmapsto \left[\trop(C )\right]
\]
is called the \emph{tropicalization map}. In particular, if $[X]$ is
the fundamental cycle of $X$, then $\trop([X])$ is the
tropical cycle that is represented by the Minkowski weight in
$M_n(\Pi)$ given by assigning
weight one to all $n$-dimensional cones of $\Pi$. We set
\begin{displaymath}
  [|\Pi|]\coloneqq \trop([X]).
\end{displaymath}

\end{Def}

\begin{rem}\
  \begin{enumerate}
  \item The cycle $[|\Pi|]$ is a balancing condition on the quasi-embedded rational
    conical space $|\Pi|$.
  \item The tropicalization map factors through the group of numerical
    classes of $k$-cycles on $X$ (with real coefficients), which is
    denoted by $N_k(X)_{\R} = N_k(X)$, and hence we get a well defined
    tropicalization map  
    \[
      N_k(X) \longrightarrow Z_k(|\Pi|)
    \]
    which we also denote by $\trop$. 
  \end{enumerate}
\end{rem}

The following theorem relates algebraic intersection numbers and
tropical intersection numbers. Recall that the rational conical space $|\Pi|$ is quasi-embedded.   
\begin{theorem}\label{the:trop-alg-numbers}
 Let $D \in \Div(X,U)$ be a toroidal divisor. Then for every
 $k$-dimensional cycle class $[C]$ in $N_k(X)$ the following tropical
 cycle classes agree 
\[
\trop\left(D \cdot [C]\right) = [\phi_D] \odot \trop([ C]) \in Z_{k-1}(|\Pi|),
\]
where on the right hand side, the class $[\phi_D]$ is seen as an
element in $\PL(|\Pi|)$.      
\end{theorem}
\begin{proof}
Using Remark \ref{rem:cp}, this follows from
\cite[Proposition 4.17]{GR}.
\end{proof}
Hence we can compute top intersection numbers of
 arbitrary divisors using the tropical intersection product.
 
\begin{cor}\label{cor:changecomplex} Let $X_{0}$ be a proper variety over
  $k$ and $D_1, \dotsc, D_n \in
  \Div(X_{0})$ a collection of divisors on $X_{0}$. Write $B_{0}=|D_{1}|\cup
  \dots \cup |D_{n}|$. Let $\pi\colon X_{1}\to X_{0}$ be a proper birational
  morphism and $B_{1}$ a snc divisor of $X_{1}$ satisfying the following two properties:
  \begin{enumerate}
   \item $\pi
  ^{-1}(B_{0})\subseteq B_{1}$. 
  \item $X_{1}\setminus B_{1}\hookrightarrow
  X_{1}$ is a toroidal embedding such that $|\Pi _{1}|\coloneqq |\Pi
  _{(X_{1},X_{1}\setminus B_{1})}|$ is  
  a quasi-embedded rational conical space. 
  \end{enumerate}
  Then the algebraic top intersection
  number $\deg(D_1 \dotsm D_n)$ can be computed tropically on the
  quasi-embedded, balanced rational conical complex $\Pi_{1}$ as  
\[
\deg\left(D_1 \dotsm D_n\right) = \deg\left(\phi_{\pi^*D_1} \odot\dotsm
\odot\phi_{\pi^*D_n} \odot [|\Pi_{1}|]\right).
\]
If we choose a Euclidean norm in $N_{\R}^{|\Pi _{1}|}$ and denote by 
\enquote{ $\widehat{\phantom {a}} $ }the map between lattice Minkowski cycles and
Euclidean Minkowski cycles, then the algebraic top intersection
number can also be computed as
\begin{displaymath}
  \deg\left(D_1 \dotsm D_n \right)= \deg\left(\phi_{\pi^* D_1} \cdot \dotsm \cdot
  \phi_{\pi^* D_n} \cdot \widehat{\left[|\Pi_{1}|\right]}\right).
\end{displaymath}
\end{cor}
\begin{proof}
  The pair $(X',B')$ exists thanks to resolution of singularities
  \cite{HIR}, Proposition
  \ref{prop:changedivisor} and Theorem \ref{the:allowablemodification}. 
By Theorem \ref{the:trop-alg-numbers} and the functoriality of the
intersection product, we get 
\[
\deg\left(D_1 \dotsm D_n\right) = \deg\left(\pi^* D_1 \dotsm \pi^*D_n\right) =
\deg\left(\phi_{\pi^* D_1}
\odot \dotsm \odot 
\phi_{\pi^* D_n}\odot [|\Pi_{1}|]\right),   
\]
proving the first statement.
By Proposition \ref{propnorm} and Remark
\ref{normalizedproduct}, we get
\[
\phi_{\pi^* D_1} \odot \dotsm \odot \phi_{\pi^* D_n}\odot [|\Pi_{1}|] = \phi_{\pi^*
  D_1} \cdot\dotsm \cdot \phi_{\pi^* D_n} \cdot \widehat{\left[|\Pi_{1}|\right]},
\]
which proves the second one.
\end{proof}

\subsection{Positivity properties of cycles}
\label{sec:posit-prop-cycl}

From now on we assume we have chosen a Euclidean norm on $N_{\R}^{|\Pi| }$
and denote by $\widehat{|\Pi|} $ the induced Euclidean conical complex. We
also denote by $\widehat{[|\Pi| ]}$ the balancing condition obtained by
normalizing $[|\Pi |]=\trop([X])$. 

Since the map \enquote{$\trop$} between algebraic cycles on $X$ and Minkowski
cycles on $|\Pi| $ is neither surjective nor injective, the
positivity notions in the algebraic and tropical worlds do not
correspond exactly. Recall that $N^{k}(X)$ denotes the space of
 $\R$-cycles of codimension~$k$ up to numerical equivalence. The cone
$\Peff_{n-k}(X)=\Peff^{k}(X)$ of pseudo-effective cycles is the
closure of the cone generated by classes of effective cycles. The dual
cone is the cone of numerically effective cycles
\begin{displaymath}
  \Nef^{k}(X)=\{\beta \in N^{k}(X) \mid \beta\cdot \alpha \ge 0, \
  \forall \alpha \in \Peff_{k}(X)\}
\end{displaymath}
The space of nef toroidal Cartier divisors is denoted by
$\Div^{+}(X,U)$. 

The first relation between positivity in the algebraic and the
tropical worlds is the following.

\begin{lemma}\label{lemm:5}\
  \begin{enumerate}
  \item\label{item:5} Let $\alpha \in \Nef^{n-k}(X)$, then $\trop(\alpha )\in
    Z_{k}(|\Pi| )$ is a positive cycle. Therefore $\widehat{\trop(\alpha )}\in
    Z_{k}(\widehat {|\Pi|} )$ is also positive.
  \item\label{item:6} If $D\in \Div^{+}(X,U)$ is a nef toroidal Cartier
    divisor then the corresponding piecewise linear function $\phi _{D}$ defined on $\widehat{\Pi}$ 
    is weakly concave in the sense of Definition \ref{def:7}
    (\cite[Definition 4.6]{BBS:caps}). 
  \end{enumerate}
\end{lemma}
\begin{proof}
  For any cone $\sigma \in \Pi (k)$,
  \begin{displaymath}
    \trop(\alpha )(\sigma )=\alpha \cdot \overline{S^{\sigma }}\ge 0,
  \end{displaymath}
  because $\alpha $ is nef and $\overline{S^{\sigma }}$ is an effective
  cycle. The second statement follows from the first and the equality
  \begin{displaymath}
    \phi _{D}\odot [\Pi ] = \trop(D).
  \end{displaymath}
\end{proof}

We next see several examples that show that the above lemma is almost
all we can expect.

\begin{exa}\label{exm:2}
  Let $X$ be the blow up of $\P^{2}$ at a point. Let $B$ be a snc divisor such
  that the exceptional divisor $E$ is contained in the support of $B$
  and such that the associated complex is quasi-embedded. Then $[E]$
  is an effective cycle but $\trop([E])$ is not positive. So the
  statement \eqref{item:5} of the above lemma can not be extended to
  pseudo-effective cycles.
\end{exa}

\begin{exa}\label{exm:1}
  Let again $X$ be the blow up of $\P^{2}$ at a point $p$. Let
  $r_{1}$, $r_{2}$ be the strict
  transforms of two different lines passing through $p$ and 
  $\ell_{i}$, $i=1,2,3$ be the strict transforms of three general lines on
  $\P^{2}$. Put $B=\ell_{1}\cup
  \ell_{2}\cup \ell_{3}\cup r_{1}\cup r_{2}$. There are rational functions
  $f_{1}$, $f_{2}$ and $g$ with
  \begin{displaymath}
    \div(f_{1})=\ell_{1}-\ell_{2},\quad
    \div(f_{2})=\ell_{2}-\ell_{3},\quad
    \div(g)=r_{1}-r_{2}.
  \end{displaymath}
  This easily implies that the complex associated to the toroidal
  embedding $X\setminus B\hookrightarrow X$ is quasi-embedded. Let $E$
  denote again the exceptional divisor. Then 
  $\trop([E])\ge 0$ because $E$ is not contained in
  $B$. Nevertheless $E$ is not nef. Therefore the converse of
  statement \eqref{item:5} is not true. Put $D= \ell_{1}-r_{1}$. Since
  $D\sim E$, we see that $\phi _{D}$ is weakly concave but $D$ is not
  nef. Therefore the
  converse of statement \eqref{item:6} does not hold.
\end{exa}

\begin{exa}
  We put ourselves in the situation of Example \ref{exm:1} and
  let $B'=\ell_{1}\cup
  \ell_{2}\cup \ell_{3}$. Then the obtained conical complex is still
  quasi-embedded, but the map $\trop$ satisfies
  $\trop(\pm[E])=0$. Therefore $\trop(\alpha )\ge 0$ does not even
  imply that $\alpha $ is effective.    
\end{exa}

\begin{exa}\label{exm:5}
  Let $X$ be an elliptic curve, $O$ the marked point i.e. the neutral
  element for the group law, $P$ a non
  torsion point and $Q=-P$. Put $B=\{O,P,Q\}$. There is a rational
  function $f$ on $X$ such that
  \begin{displaymath}
    \div(g)=2O-P-Q.
  \end{displaymath}
  Therefore the conical complex $\Pi$ associated to $X\setminus
  B\hookrightarrow X$ consists of three rays $\tau _{O}$, $\tau _{P}$
  and $\tau _{Q}$ with lattice generators $v _{O}$, $v _{P}$
  and $v _{Q}$. The lattice $N^{|\Pi|}$ is one dimensional and can
  be identified with $\Z$. Then the quasi-embedding is given by
  \begin{displaymath}
    \iota(v_{O})=2,\qquad \iota(v_{P})=\iota (v_{Q})=-1.
  \end{displaymath}
  For simplicity a Minkowski weight $c\in M_{1}(\Pi )$ will be
  denoted as a triple of real numbers $(c_{O},c_{P},c_{Q})$. The
  balancing condition $[|\Pi|]$ is the Minkowski weight
  $(1,1,1)$. The Minkowski weight $(1,2,0)$ is also positive but it
  does not come from the geometry of $X$. Consider the divisor
  $D=-P+2Q$. This is a nef divisor because it has positive
  degree. Nevertheless
  \begin{displaymath}
    \phi _{D}\odot (1,2,0)=-2<0.
  \end{displaymath}
  Hence it is not true that nef divisors always give rise to concave
  functions in the sense of Definition \ref{def:7}.
\end{exa}

We end this section showing that the set of nef toroidal Cartier divisors in
allowable modifications of $X$ provides an admissible family of concave functions on $\widehat{|\Pi|}$ in the sense of Definition~\ref{tropinef}.
\begin{lemma}\label{lem:neftropicollection} Let $U\hookrightarrow X$
  be as at the beginning of Section \ref{sec:inters-theory-toro}. Then
  the set $\mathcal{C}$ of piecewise linear functions on $|\Pi| $ given by  
\[
\mathcal{C}=\left\{\phi_{D} \, \big{|} \, D\in \Div^{+}(X',U),\
  \pi\colon X'\to X \text{ allowable modification}\right\}
\]
forms an admissible family of concave functions on
$\widehat{|\Pi|}$ in the sense of Definition \ref{tropinef}.
\end{lemma}
\begin{proof}
We have to show that $\mathcal{C}$ satisfies the three properties
given in Definition \ref{tropinef}.
To show property \eqref{item:7}, let $D_{1},\dots,D_{r}$ be
nef Cartier toroidal divisors on the allowable modifications
$X_{\Pi_1}, \dots , 
X_{\Pi_r}$, respectively. Let $\Pi'$ be a smooth common refinement of
$\Pi_1, \dots, \Pi_r$ and denote by $\pi_i \colon X_{\Pi'} \to
X_{\Pi_i}$ the corresponding allowable modification. Then, by Theorem
\ref{the:trop-alg-numbers} and Proposition \ref{propnorm}, we get 
\begin{displaymath}
  \deg(\phi_{D_1} \dotsm \phi_{D_r} \cdot\widehat{\left[|\Pi |\right]})
  =\deg\left(\pi_1^*[D_1]\dotsm \pi_r^*[D_r] \cdot
    [X_{\Pi'}]\right) \geq 0,
\end{displaymath}
where the last inequality uses the fact that the pullback of a nef
divisor under a proper map is
nef and Kleiman's criterion for nefness.

Property \eqref{item:8} is clear. 

To prove \eqref{item:4} we first recall that the set of
piecewise linear functions on $|\Pi|$ with rational slopes is
dense in the set of continuous conical functions on $|\Pi|$ with the
topology of uniform convergence on compacts. Thus it is enough to show
that any piecewise linear function on $|\Pi|$ with rational
slopes belongs to $\mathcal{C}-\mathcal{C}$. 

A piecewise linear
function $\phi $ with rational slopes on $|\Pi|$ defines a Cartier toroidal
divisor $D_{\phi }$ on an allowable modification $X'$ of $X$. Since we are
assuming that there is an ample divisor on $X$ whose support is
contained in the boundary divisor $B$, there is an allowable modification $\pi \colon X''\to
X'$ and an 
ample toroidal divisor $A$ on $X''$. We can choose an integer $r>0$ such
that $C=\pi ^{\ast}D_{\phi }+ rA$ is also ample. Therefore
\begin{displaymath}
  \phi =\phi _{C}-\phi _{rA},\quad \phi _{C}, \phi _{rA}\in \mathcal{C}, 
\end{displaymath}
completing the proof.
\end{proof}

\subsection{Toroidal b-divisors}

We give the definition of the toroidal Riemann--Zariski space
$\mathfrak{X}_{U}$ of $(X,U)$ and define Cartier and Weil toroidal
b-divisors on $\mathfrak{X}_{U}$. Extending the bijection between
toroidal divisors on $X$ and piecewise linear functions defined on $\Pi$
(Proposition \ref{prop:functiontoroidaldivisor}), we give a bijection
between Cartier (respectively Weil) toroidal b-divisors
and piecewise linear (respectively conical) functions on the conical rational polyhedral space
$|\Pi|$. Moreover, extending the results in
Section~\ref{sec:int-prod-tor} we show that the top intersection
product of Cartier toroidal b-divisors can be computed
tropically on the rational conical space.

Consider the directed set $R_{\sm}(\Pi)$ of smooth rational conical subdivisions of
$\Pi$. For any $\Pi' \in
R_{\sm}(\Pi)$, we denote by $X_{\Pi'}$ the corresponding smooth
proper allowable modification from Theorem
\ref{the:allowablemodification}. The \emph{toroidal Riemann--Zariski
  space} of $(X,U)$ is defined formally as the inverse limit in the category of locally ringed spaces
\[
\mathfrak{X}_{U} \coloneqq \varprojlim_{\Pi' \in R_{\sm}\left(\Pi\right)}X_{\Pi'}
\] 
with maps given by the proper toroidal birational morphisms $X_{\Pi''}
\to X_{\Pi'}$ defined whenever $\Pi'' \geq \Pi'$.

Toroidal b-divisors can be viewed as divisors on $\mathfrak{X}_{U}$:
\begin{Def}\label{def:toroidalbdivisor}
The group of \emph{$\R$-Cartier toroidal b-divisors} on $X$ is defined
to be the injective limit
\[
\CabDiv(\mathfrak X_{U})_{\R} \coloneqq \varinjlim_{\Pi' \in
  R_{\sm}(\Pi)} \Div\left(X_{\Pi'},U\right)_{\R}, 
\]
with maps given by the pull-back map of $\R$-toroidal divisors. It is
endowed with its inductive limit topology, called the \emph{strong
  topology}.
 
The group of \emph{$\R$-Weil toroidal b-divisors} on $X$ is defined to
be the projective limit 
\[
\WbDiv(\mathfrak X_{U})_{\R} \coloneqq \varprojlim_{\Pi' \in
  R_{\sm}(\Pi)} \Div\left(X_{\Pi'},U\right)_{\R},
\]
with maps given by the push-forward map of $\R$-toroidal divisors
(note that the $X_{\Pi'}$'s are smooth, hence we can identify Cartier
and Weil divisors). It is endowed with its projective limit topology,
called the \emph{weak topology}.

We will write $\R$-Cartier and $\R$-Weil toroidal b-divisors in bold
notation $\D$ to distinguish them from classical $\R$-toroidal
divisors $D$.  
\end{Def}
We make the following remarks.
\begin{rem} 
\begin{enumerate}
\item As before, to simplify notation, we will usually omit the
  coefficient ring $\R$ from the notation (real coefficients being
  always implicit) and refer to $\R$-toroidal b-divisors simply as toroidal b-divisors. 
\item Since the set of smooth subdivisions is directed, a toroidal Cartier
  b-divisor can be represented by a pair $(X_{\Pi '},D)$, where
  $X_{\Pi' }$ is the allowable modification of $X$ given by the
 rational conical complex $\Pi ' \in R_{\sm}(\Pi)$. Two pairs represent the same Cartier b-divisor
  if there is a common refinement and the pull back of both divisors
  to the corresponding modification agree. 
\item We can view a Weil toroidal b-divisor as a family 
\[
\D = \left(D_{\Pi'}\right)_{\Pi' \in R_{\sm}(\Pi)},
\]
where for each $\Pi' \in R_{\sm}(\Pi)$, we have that $D_{\Pi'} \in
\Div(X_{\Pi'},U)$, and these elements are compatible under
push-forward. 

\item We can view a Cartier toroidal b-divisor as a Weil toroidal b-divisor 
\[
\pmb{E} = \left(E_{\Pi'}\right)_{\Pi' \in R_{\sm}(\Pi)},
\]
for which there is a model $X_{\widetilde{\Pi}}$ for some
$\widetilde{\Pi} \in R_{\sm}(\Pi)$ such that for every other 
model $X_{\Pi''}$ with $\Pi'' \geq \widetilde{\Pi}$ in
$R_{\sm}(\Pi)$, the incarnation $E_{\Pi''}$ is the pull-back of
$E_{\widetilde{\Pi}}$ on $X_{\widetilde{\Pi}}$. 

Hence, we have the inclusion  
\[
 \CabDiv(\mathfrak X_{U}) \subseteq \WbDiv(\mathfrak X_{U}),
 \]
 and we may refer to a Weil toroidal b-divisor just as a toroidal
 b-divisor. 
 
 \item A net $\left(\pmb{Z}_i\right)_{i \in I}$ converges to a
   b-divisor $\pmb{Z}$ in $\WbDiv(\mathfrak X_{U})$ if and only if for
   each $\Pi' \in R_{\sm}(\Pi)$ we have that
   $\left(Z_{i,\Pi'}\right)_{i \in I}$ converges to $Z_{\Pi'}$
   coefficient-wise.  
 \item By the following Proposition \ref{characterization} we have that the spaces of Weil and Cartier b-divisors on $X$ agree with the spaces of piecewise linear and conical functions on $|\Pi|$, respectively. However, they are different from the ones on $\widehat{|\Pi|}$, because the allowed
   subdivisions are different (see Remark \ref{rem:2}).
\end{enumerate}
\end{rem}
We have the following combinatorial characterization of toroidal
b-divisors. Recall that $|\Pi| (\Q)$ denotes the set of points of
$|\Pi|$ with rational coordinates in any cone of $\Pi $.
\begin{prop}\label{characterization} 
The map $\D \mapsto \phi_{\D}$ from Proposition
\ref{prop:functiontoroidaldivisor} induces homeomorphisms 
\[
\CabDiv(\mathfrak X_{U}) \simeq \PL\left(|\Pi|\right)
\]
and 
\[
\WbDiv(\mathfrak X_{U}) \simeq \Conic\left(|\Pi|\right)
\]
between toroidal b-divisors on $X$ and functions on $|\Pi|$.
Thus, the space $\WbDiv(\mathfrak X_{U})$ is homeomorphic to the
  space of conical functions $|\Pi| (\Q)\to \R$ with the topology of
  pointwise convergence. 
\end{prop}
\begin{proof}
This follows from the definition of conical functions on a rational
conical complex (Remark \ref{rem:2}) and Proposition \ref{prop:functiontoroidaldivisor}. 
\end{proof}
As in the case of functions, one can define the top intersection
product of a collection of toroidal b-divisors when there is at most
one Weil b-divisor involved (all the other must be Cartier).  
\begin{Def}\label{def:intersection-cartier}
Let $\D_1, \dotsc, \D_n$ be toroidal b-divisors on $X$ and assume
that at most one of the $\D_i$'s is not Cartier. Without lost of
generality, assume that $\D_1$ is not Cartier. Let $\Pi' \in
R_{\sm}(\Pi)$ such that all of the $\D_i$'s for
$i = 2 ,\dotsc, n$ are determined on $X_{\Pi'}$. The top intersection
product $\langle \D_1 \dotsm \D_n\rangle $ is defined by  
\[
\langle \D_1 \dotsm \D_n\rangle  \coloneqq \deg(D_{1, \Pi'} \cdot D_{2,\Pi'} \dotsm D_{n,\Pi'}).
\]
\end{Def}
By the projection formula in algebraic geometry, this is independent
of the choice of the common refinement $\Pi'$.  

%
\begin{rem}
It follows from Corollary \ref{cor:changecomplex} that the top
intersection product of a collection of toroidal b-divisors where at most one of them is Weil can be computed tropically.  
\end{rem}

\subsection{Top intersection product of nef toroidal Weil b-divisors} 

Let $X$, $U$, $B$ and $\Pi $ be as at the beginning of section
\ref{sec:inters-theory-toro}. Chose a Euclidean metric on $N_{\R}^{|\Pi| }$
and denote by $\widehat{|\Pi|} $ the induced Euclidean conical space.

We start by defining \emph{nef}
Cartier and \emph{nef} Weil toroidal b-divisors. Using Lemma
\ref{lem:neftropicollection} we deduce that the set of piecewise linear functions on $\widehat{|\Pi|}$ which are induced from nef Cartier toroidal
b-divisors forms an admissible family $\mathcal{C}$ of concave functions in the sense of Definition~\ref{tropinef}. Using the
results of Section 
\ref{sec:monge-ampere-meas}, we deduce that the top intersection product of
Cartier b-divisors can be extended continuously to nef Weil toroidal
b-divisors and, by linearity, to differences of nef Weil toroidal
b-divisors.  

\begin{Def}\label{def:nefbdiv}
A Cartier toroidal b-divisor $\pmb{E} = \left(E_{\Pi'}\right)_{\Pi'
  \in R_{\sm}(\Pi)} \in \CabDiv(\mathfrak{X}_{U})$ is said to be \emph{nef} if $E_{\Pi'} \in
\Div(X_{\Pi'},U)$ is nef for some (hence any) determination $E_{\Pi'}$
of $\pmb{E}$. The set of nef toroidal Cartier b-divisors forms a
cone in $\CabDiv(\mathfrak{X}_{U})$, denoted by $\CabDiv^{+}(\mathfrak{X}_{U})$. The cone of nef toroidal Weil
b-divisors is the closure  
in $\WbDiv(\mathfrak X_{U})$ of  $\CabDiv^{+}(\mathfrak{X}_{U})$. It is
denoted by $\WbDiv^{+}(\mathfrak{X}_{U})$.
\end{Def} 

\begin{rem} \label{rem:5} By Proposition \ref{characterization} and
  Remark \ref{rem:2} the inclusion between Cartier and Weil b-divisors 
  $\CabDiv(\mathfrak{X}_U)\hookrightarrow \WbDiv(\mathfrak{X}_U)$ can be factored as
  \small
  \[
    \CabDiv\left(\mathfrak{X}_{U}\right)=\PL(|\Pi|)\to
    \PL\left(\widehat{|\Pi|}\right)\to
    \Conic\left(\widehat{|\Pi|}\right)\to
    \Conic(|\Pi| )=\WbDiv\left(\mathfrak{X}_{U}\right).
  \]
  \normalsize
By Lemma \ref{lem:neftropicollection} the image of
$\CabDiv^{+}\left(\mathfrak{X}_{U}\right)$ in $\PL\left(\widehat{|\Pi|}\right)$ forms an
admissible family of concave functions. By Theorem \ref{thm:1} the
elements in the closure of $\CabDiv^{+}\left(\mathfrak{X}_{U}\right)$ in
$\Conic\left(\widehat{|\Pi|}\right)$ are continuous functions. Since a
continuous function is determined by its values on a dense subset and
the topologies on $\Conic\left(\widehat{|\Pi|}\right)$ and $\Conic(|\Pi| )$
are both that of pointwise convergence, we deduce that the closure of
$\CabDiv^{+}\left(\mathfrak{X}_{U}\right)$ in $\Conic\left(\widehat{|\Pi|}\right)$ is
naturally homeomorphic to its closure in $\Conic(|\Pi|)$. The last one
can be identified with the cone $\WbDiv^{+}(\mathfrak{X}_U)$. As a
consequence, in order to work with nef toroidal Weil b-divisors,
there is no difference between working on $|\Pi|$ or in
$\widehat{|\Pi|}$.    
\end{rem}

\begin{rem}\label{rem:4}
  A consequence of the preceding remark and Theorem \ref{thm:1} is
  that the closure of $\CabDiv^{+}(\mathfrak{X}_{U})$ in
  $\WbDiv(\mathfrak{X}_{U})$ agrees with its sequential
  closure. It follows that
  if $\pmb{D}\in \WbDiv^{+}(\mathfrak{X}_{U})$, then there is a sequence
  $(\D_{i})_{i\in \N}$ of nef toroidal Cartier b-divisors converging to
  $\pmb{D}$. Moreover, when we view $\pmb{D}$ and the $\D_{i}$ as
  functions on $|\Pi|$, the convergence is uniform on compacts.  
\end{rem}
From now on we fix
$\mathcal{C}=$ image of $\CabDiv^{+}(\mathfrak{X}_{U})$ in $\PL\left(\widehat{|\Pi|}\right)$ as the admissible family of
concave functions. 

The following theorem describes nef toroidal b-divisors
combinatorially. In view of Remark \ref{rem:5}, it is a direct
consequence of Theorem \ref{thm:1}. 

\begin{theorem}\label{the:nefextensionfunction}
Let $\D$ be a nef toroidal b-divisor on $X$. Then the corresponding
function $\phi_{\D}$ on $|\Pi|(\Q)$ of Proposition
\ref{characterization} extends to a continuous function on $|\Pi|$,
which we denote also by $\phi_{\D}$. The function $\phi
_{\D}$ defines a $\mathcal{C}$-concave function on $\S^{\widehat{|\Pi|
  }}$ in the sense of Definition~\ref{def:cont-b-c-nef}..
\end{theorem}

Since we can view toroidal Cartier b-divisors as toroidal Weil b-divisor, there is a
potential ambiguity when we say that a toroidal Cartier b-divisor is nef. Lemma
\ref{lem:ambiguity} below shows that this 
potential ambiguity is not a real ambiguity. Before stating it we make
the following remark.

\begin{rem}\label{rem:negativity}
  Let $\Pi'' \in R_{\sm}\left(\Pi\right)$ and let $D$ be a nef toroidal
  divisor on $X_{\Pi''}$. Then for any $\Pi' \leq \Pi''$ in
  $R_{\sm}\left(\Pi\right)$ we have that $D \leq \pi^*\pi_* D$,
  where $\pi \colon X_{\Pi''} \to X_{\Pi'}$ denotes the corresponding
  proper birational morphism.  Indeed, the divisor $D - \pi^*\pi_* D$
  is $\pi$-nef (i.e. has non-negative intersection with every curve
  contracted by $\pi$) and is $\pi$-exceptional. Hence, from the
  well-known Negativity Lemma (see e.g. \cite[Lemma 3.39]{KM}), it follows that $D - \pi^*\pi_* D \leq 0$.
\end{rem}
\begin{lemma}\label{lem:ambiguity}
Let $\D$ be a nef Weil toroidal b-divisor and assume that $\D$ is
Cartier. Then $\D$ is a nef toroidal Cartier b-divisor in the sense of
Definition \ref{def:nefbdiv}. 
\end{lemma}
\begin{proof}
Let $\Pi' \in R_{\sm}\left(\Pi\right)$ be such that $\D$ is determined on $X_{\Pi'}$. We have to show that $D_{\Pi'}$ is nef on $X_{\Pi'}$. For this,
let $C \subseteq X_{\Pi'}$ be an irreducible curve. It suffices to show that the
intersection product $D_{\Pi'} \cdot C$ is non-negative.

Let $\left\{B'_i \; \big{|} \; i \in I'\right\}$ be the irreducible
components of the boundary divisor $B'=X_{\Pi'} \setminus U$ and for
any subset $J' \subseteq I'$, denote by $B_{J'}$ the boundary intersection
$\bigcap_{j \in J'}B_j$ (in particular, $B_{\emptyset} = X_{\Pi'}$).
Let $K' \subseteq I'$ such that $B_{K'}$ is the minimal boundary intersection
containing $C$.

If $\codim\left(B_{K'}\right) \geq 2$, we can find a
subdivision $\Pi'' \geq
\Pi'$ in $R_{\sm}\left(\Pi\right)$ and a curve $\tilde{C} \subseteq X_{\Pi''}$
such that the following two conditions are satisfied: 
\begin{enumerate}
\item $\pi_*\tilde{C} = aC$ for some natural number $a > 0$.
\item Denoting by $\left\{B''_i \; \big{|} \; i \in I''\right\}$ the
  irreducible components of the boundary divisor $B'' = X_{\Pi''}
  \setminus U$, the minimal boundary intersection $B_{K''}$ containing $\tilde{C}$
  (for some subset $K'' \subseteqq I''$) satisfies that
  $\codim\left(B_{K''}\right) = 1$.
\end{enumerate}
If $\pi^*\tilde{C}D_{\Pi'} \cdot \tilde{C} \geq 0$, then using the
projection formula, we get that $D_{\Pi'} \cdot \pi_*\tilde{C} =
D_{\Pi'} \cdot C \geq 0$. Hence, replacing $\Pi '$ by $\Pi ''$, we may
assume that $B_{K'}$ has
codimension $\leq 1$. 

Let $\left\{\D_{i}\right\}_{i\in \N}$ be a sequence of
nef toroidal Cartier b-divisors converging to $\D$. We view them as toroidal Weil
b-divisors. In particular, on
$X_{\Pi'}$, we have that  
\[
D_{i, \Pi'} \xrightarrow[i\in \N]{} D_{\Pi'}
\]
component-wise,
and by continuity of the intersection product,
\begin{align}\label{eq:convergence}
D_{i, \Pi'}\cdot C \xrightarrow[i\in \N]{}
  D_{\Pi'} \cdot C.
\end{align}
Now, for each $i\in \N$, let $\Pi_{i} \in
R_{\sm}\left(\Pi\right)$ be a determination of $\D_{i}$. We
may assume that $\Pi_{i} \geq \Pi'$. Also, we let $\pi_{i}
\colon X_{\Pi_{i}} \to X_{\Pi'}$ denote the corresponding proper
birational morphism. Let $C_{i}$ be
the strict transform of the curve $C$ under $\pi_{i}$. Note that
this is well defined by the assumption that the minimal boundary
intersection that
contains $C$ has codimension less or equal than one.

Using the projection formula, we compute
\begin{align*}
D_{i ,\Pi'} \cdot C &= D_{i,\Pi'} \cdot {\pi_{i}}_*C_{i} \\
&= \pi_{i}^*D_{i,\Pi'} \cdot C_{i} \\
&= \pi_{i}^*{\pi_{i}}_*D_{i, \Pi_{i}} \cdot C_{i} \\
&= \left( \pi_{i}^*{\pi_{i}}_*D_{i, \Pi_{i}} - D_{i,\Pi_{i}}\right)\cdot C_{i} + D_{i, \Pi_{i}} \cdot C_{i} \\
&\geq 0.
\end{align*}
Indeed, the first summand is non-negative since it follows from Remark~\ref{rem:negativity} that both the terms
$\pi_{i}^*{\pi_{i}}_*D_{i, \Pi_{i}} -
D_{i,\Pi_{i}}$ and $C_{i}$ are effective and
intersect properly. The second summand is non-negative since
$D_{i,\Pi_{i}}$ is nef and $C_{i}$ is effective. 

By \eqref{eq:convergence}, $D_{\Pi'}\cdot C$ is a limit of
non-negative real numbers. Hence it is itself non-negative. This concludes
the proof. 
\end{proof}

The next is the main result of this paper.

\begin{theorem}\label{thm:intersectionbdivisors}
  The restriction of the top intersection product of
  toroidal Cartier b-divisors (Definition \ref{def:intersection-cartier}) to nef toroidal
  Cartier b-divisors
  \begin{displaymath}
    (\CabDiv^{+}(\mathfrak X_{U}))^{n}\longrightarrow \R
  \end{displaymath}
  can be extended continuously to a symmetric multilinear intersection
  product of nef toroidal
  Weil b-divisors
  \begin{displaymath}
    (\WbDiv^{+}(\mathfrak X_{U}))^{n}\longrightarrow \R.
  \end{displaymath}
  If $\D_1, \dotsc, \D_n$ is a collection of nef toroidal Weil
  b-divisors on $X$, then their top intersection product is given by
\[
\langle \D_1 \dotsm \D_n \rangle =
\int_{\mathbb{S}^{\widehat{|\Pi|}}}\phi_{\D_1}(u) \,d\mu_{\D_2,
  \dotsc, \D_n}, 
\]
where $\mu_{\D_2, \dotsc, \D_n}$ denotes the mixed Monge--Amp\`ere measure
induced by the collection of $\mathcal{C}$-concave functions $\phi_{\D_2},
\dotsc, \phi_{\D_n}$ on $\S^{\widehat{|\Pi|}}$ from Definition
\ref{propdef:mixedmeasure}. 
\end{theorem}
\begin{proof}
  The proof is just putting together everything we have done up to
  now. Every nef toroidal Cartier b-divisor $\D$ on $X$ defines a
  piecewise linear function $\phi _{\D}$ on $\widehat{|\Pi|} $ (Proposition~\ref{characterization}). The family of piecewise linear functions on
  $\widehat{|\Pi|} $ obtained in this way forms an admissible family
  $\mathcal{C}$-concave functions on $\widehat{|\Pi|} $ (Lemma
  \ref{lem:neftropicollection}). The space of nef toroidal Weil b-divisors on
  $\mathfrak X_{U}$ is homeomorphic to the space of $\mathcal{C}$-concave functions on $\widehat{|\Pi|} $ (Proposition~\ref{characterization} and Remark \ref{rem:5}). Thus the result is
  a direct consequence of Corollary \ref{cormixed}.
\end{proof}

\begin{rem}
  Let $\D_1, \dotsc, \D_n$ be a collection of nef toroidal Weil
  b-divisors on $X$. 
  By Remark~\ref{rem:4} there are sequences $(\D_{j,k})_{k\in \N}$,
  $j=1,\dots,n$ of nef toroidal Cartier b-divisors on $\mathfrak{X}_{U}$
  converging to them. Then, for each $j$, the sequence
  of functions $(\phi _{\D_{j,k}})_{k\in \N}$ converges uniformly on
  compacts to $\phi _{\D_{j}}$. Moreover
  \begin{equation}\label{eq:13}
    \langle  \D_1 \dotsm \D_n \rangle = \lim_{k\to \infty}
    \langle  \D_{1,k} \dotsm \D_{n,k} \rangle.
  \end{equation}
  One has to be careful that the continuity condition \eqref{eq:13} is
  only true when the sequences  $(\D_{j,k})_{k\in \N}$ consist of nef toroidal Cartier b-divisors. Namely, one can construct sequences of toroidal Cartier b-divisors
  $(\D'_{j,k})_{k\in \N}$ such that, for each $j$, the sequence
  of functions $(\phi _{\D_{j,k}})_{k\in \N}$ converges uniformly on
  compacts to $\phi _{\D_{j}}$ and nevertheless the continuity
  condition \eqref{eq:13} does not hold.  
\end{rem}

\begin{rem}
  Since the intersection product is multilinear (for the semigroup law
  of $\WbDiv^{+}(\mathfrak{X}_{U})$), it can be extended by
  multilinearity to the space
  \begin{displaymath}
   \WbDiv^{+}(\mathfrak{X}_{U})-\WbDiv^{+}(\mathfrak{X}_{U}) 
 \end{displaymath}
 of toroidal Weil b-divisors that are differences of nef ones.
\end{rem}

\section{Applications}\label{sec:applications}

Let $U \hookrightarrow X$ be a toroidal embedding as at the beginning
of Section \ref{sec:inters-theory-toro}. That is, we assume that $X$
is smooth and projective, $B=X\setminus U$ is a snc divisor, that there
is an effective ample divisor $A$  with support $B$ and that the 
corresponding rational conical space $|\Pi|$ is
quasi-embedded. Recall that, after Definition \ref{def:6}, by divisor
we mean divisor with $\R$ coefficients. We also assume that we have
chosen an auxiliary
Euclidean metric 
on $N_{\R}^{|\Pi| }$.

A toroidal b-divisor on $X$ is \emph{big} if it has enough
global sections (Definition \ref{def:big}). In this section, as an
application of our results, we show a
Hilbert--Samuel type formula for nef and big toroidal b-divisors on
$X$ relating the degree of a nef toroidal b-divisor both with its
volume and with the volume of the associated convex
Okounkov body (Definitions \ref{def:volume-b-div} and
\ref{def:okounkov} and Theorem \ref{the:hilbertsamuel}). As a corollary, we obtain the continuity of the volume function on the space of nef and big b-divisors (Corollary \ref{cor:cont}) and a Brunn--Minkowski type inequality (Corollary~\ref{cor:brunnminkowski}).

\subsection{Volumes and convex Okounkov bodies of toroidal b-divisors}

We start with the definition of the space of global sections of a
toroidal b-divisor.
\begin{Def} 
Let $F = k(X)$ be the field of rational functions of $X$. Then for any
toroidal b-divisor $\D = \left(D_{\Pi'}\right)_{\Pi' \in
  R_{\sm}(\Pi)}$ on $X$ one defines the space of global sections of
$\D$ by 
 \[
H^0(\mathfrak{X}_{U}, \D) = \{f \in F^{\times} \,|\, b\text{-div}(f) +
\D \geq 0\} \cup \{0\}\subseteq F, 
\]
where $b\text{-div}(f)$ is the (Cartier) b-divisor induced by a rational function by setting 
\[
b\text{-div}(f) = \left(\div_{X_{\Pi'}}(f)\right)_{\Pi' \in R_{\sm}(\Pi)},.
\]
\end{Def}
\begin{rem}\ 
\begin{enumerate}
\item By definition, we have that $H^0\left(\mathfrak{X}_{U}, \D\right)$ is an
  intersection of finite-dimensional vector spaces
\[
H^0(\mathfrak{X}_{U}, \D) = \bigcap_{\Pi' \in R_{\sm}(\Pi)}
H^0\left(X_{\Pi'}, D_{\Pi'}\right).
\]
 \item We have a well defined map 
\[
H^0\left(\mathfrak{X}_{U}, \D\right) \times H^0\left(\mathfrak{X}_{U}, \pmb{E}\right) \longrightarrow
H^0\left(\mathfrak{X}_{U}, \D + \pmb{E}\right) 
\]
for any toroidal b-divisors $\D$ and $\pmb{E}$.
\end{enumerate}
\end{rem}

\begin{Def}\label{def:volume-b-div}
 Let $\D$ be a toroidal b-divisor. The volume of $\D$ is defined by 
 \[
   \vol(\D) \coloneqq \limsup_{\ell \to \infty}
   \frac{h^0\left(\mathfrak X_{U}, \ell \D\right)}{\ell^n/n!},
 \]
 where $h^0\left(\mathfrak X_{U}, \ell \D\right)$ denotes the dimension of the space $H^0\left(\mathfrak X_{U}, \ell \D\right)$.
 \end{Def}
 We now associate a convex Okounkov body to $\D$. 
 \begin{Def}
 Let $\D$ be a toroidal $b$-divisor on $X$. We define the b-divisorial
 algebra $\mathcal{R}_{\mathfrak{X}_{U}}(\D)$ associated to $\D$ by  
 \[
 \mathcal{R}_{\mathfrak{X}_{U}}(\D) \coloneqq \bigoplus_{k\ge
   0}H^0\left(\mathfrak X_{U}, k \D\right) t^k.
 \]
 This is a graded sub-$k$-algebra of $F[t]$.  
 \end{Def}
 One of the fundamental problems of algebraic geometry is the question
 about finite generation of divisorial algebras. It is clear that, in
 general, the b-divisorial algebra associated to a toroidal b-divisor
 is not finitely generated. However, the next proposition shows that
 it satisfies the weaker condition of being of almost integral type,
 which nevertheless, following \cite{KK}, allows us to associate a
 convex Okounkov body to it.

 Recall that a graded subalgebra $R \subseteq F[t]$ is of \emph{integral
 type} if it is a finitely generated $k$-algebra and is a finite
 module over the algebra generated by $R_{1}$, while it is of  
 \emph{almost integral type} if it is contained in a graded subalgebra
 of integral type $R \subseteq A \subseteq F[t]$ (see \cite[Section~2.3]{KK}).
 \begin{prop}
 Let $\D$ be a toroidal b-divisor. Then the b-divisorial algebra
 $\mathcal{R}_{\mathfrak{X}_{U}}(\D) \subseteq F[t]$ is of almost integral
 type. 
 \end{prop}
 \begin{proof}
 We clearly have 
 \[
\mathcal{R}_{\mathfrak{X}_{U}}(\D) \subseteq \mathcal{R}_{X_{\Pi'}}(D_{\Pi'}),
\]
for any $\Pi' \in R(|\Pi|)$, and the latter is an algebra of almost
integral type by \cite[Theorem~3.7]{KK}.
 \end{proof}
  We now briefly sketch the construction of the Okounkov body
  associated to $\mathcal{R}_{\mathfrak X_{U}}(\D)$. For more details
  we refer to \cite{KK} and \cite{LM}. The choice of a (generic,
  infinitesimal) flag on $X$ determines a valuation $\nu \colon
  F\setminus \{0\}\to 
  \Z^{n}$ that can be extended to a valuation $\nu _{t}\colon
  F[t]\setminus \{0\}\to
  \Z^{n}\times \Z$.

  The semigroup
  $S(\D) \subseteq \Z^n \times \Z$ is defined as the image of $
  \mathcal{R}_{\mathfrak{X}_{U}}(\D)$  
  by the valuation $v_t$. Then, for any integer $\ell \geq 0$, the equality 
  $h^0(X, \ell \D) = \#\left(S(\D)\cap \left(\Z^n \times
      \{\ell\}\right)\right)$ holds. Moreover,
  $S(\D)$ satisfies the conditions (2.3--2.5) of \cite{LM}. One then
  defines the cone $C(\D) = \overline{\operatorname{convhull}}(S(\D)
  \cup \{0\}) \subseteq \R^n \times \R$. This is a strictly convex
  cone.
 \begin{Def}\label{def:okounkov}
 Let $\D$ be a toroidal b-divisor. The Okounkov body $\Delta_{\D}
 \subseteq \R^n$ is defined to be the \emph{slice} of $C(\D)$ at
 height $1$, i.e. 
 \[
 \Delta_{\D} \coloneqq  C(\D) \cap \left(\R^n \times \{1\}\right).
 \]
It is a convex body (see \cite[Theorem~2.30]{KK}).
 \end{Def}
 Okounkov bodies have been useful to study geometric properties of
 divisors in terms of convex geometry. In particular, in the study of
 volumes of divisors. We will see that this extends to toroidal b-divisors. 
  The following is a classical result. 
\begin{lemma} Let $\D$ be a nef toroidal Cartier b-divisor. Then 
\begin{align}\label{eqn:volume}
\vol(\D) = n!\vol(\Delta_{\D})=  \D^n.
\end{align}
\end{lemma} 
\begin{proof}
Let $\Pi' \in R_{\sm}(\Pi)$ be such that $\D$ is determined on $X_{\Pi'}$. Then the result
follows from the well known case of nef Cartier divisors
\[
\vol(\D) = \vol(D_{\Pi'}) = D_{\Pi'}^n =  \D^n
\]
combined with 
\[
\vol(D_{\Pi'}) = n! vol(\Delta_{\D}).
\]
\end{proof}
In the next section we extend Equation \ref{eqn:volume} to nef and big Weil toroidal
b-divisors.
\subsection{A Hilbert--Samuel formula}
\label{sec:hilb-samu-form}
We start with two monotonicity lemmas for nef toroidal b-divisors. These
play a key role. 
\begin{lemma}\label{lem:monotonicity}
Let $\D= \left(D_{\Pi'}\right)_{\Pi' \in R_{\sm}(\Pi)}$ be a nef
toroidal b-divisor. Let $\Pi'' \geq \Pi' $ be subdivisions in
$R_{\sm}(\Pi)$. Then
 \[
 D_{\Pi''} \leq \pi ^{\ast} D_{\Pi'},
\]
where $\pi \colon X_{\Pi ''}\to X_{\Pi '}$ is the corresponding proper
birational morphism. 
 In particular, we get the following inclusion of spaces
 \[
 H^0\left(X_{\Pi''}, D_{\Pi''}\right) \subseteq H^0\left(X_{\Pi'}, D_{\Pi'}\right). 
 \]
\end{lemma} 
\begin{proof}
  Note that this is not a direct consequence of Remark
  \ref{rem:negativity} because the fact that $\D$ is nef does not
  imply that $D_{\pi ''}$ is nef so a small argument is needed.
  Suppose first that $\D$ is a Cartier b-divisor and let
  $\widetilde{\Pi}\in R_{\sm}(\Pi)$ be a determination of $\D$. In
  particular, we have that $D_{\widetilde{\Pi}}$ is nef. Let $\Pi
  ''\ge \Pi '$ be subdivisions in $R_{\sm}(\Pi )$. Let $\Pi'''$ be a
  common refinement of $\Pi''$ and $\widetilde{\Pi}$ and consider the
  following commutative diagram.
\begin{center}
    \begin{tikzpicture}
      \matrix[dmatrix] (m)
      {
         & X_{\Pi'''}  & \\
        X_{\widetilde{\Pi}} &  & X_{\Pi''}\\
        &  X_{\Pi'} &\\
      };
      \draw[->] (m-1-2) to node[left]{$\gamma$}(m-2-1);
      \draw[->] (m-1-2) to node[right]{$\alpha$}(m-2-3);
      \draw[->] (m-1-2) to node[right]{$\beta$} (m-3-2);
      \draw[->] (m-2-3) to node[right]{$\pi $} (m-3-2);
      
     \end{tikzpicture}
   \end{center}
   Since the pullback of a nef divisor is again nef, $D_{\Pi'''} =
   \gamma^*D_{\widetilde{\Pi}}$ is nef. 
   
   By Remark \ref{rem:negativity},  $D_{\Pi'''} \leq
   \beta^*\beta_*D_{\Pi'''}$ and we conclude that  
   \[
   D_{\Pi''} = \alpha_*D_{\Pi'''} \leq \alpha_*\beta^* \beta_* D_{\Pi'''} = \alpha_*\beta^*D_{\Pi'} = \pi ^*D_{\Pi'}
   \]

   In the general case, choose a sequence  $\left\{\D_i\right\}_{i \in
     \N}$ of nef Cartier b-divisors converging to $\D$. Then, by what
   was shown above, for each $i \in \N$ 
   we have that 
   \[
     D_{i,\Pi ''} \leq \pi ^{\ast}D_{i,\Pi '}.
   \]
Hence, taking limits at both sides we deduce 
\[
  D_{\Pi ''} = \lim_{i \in \N} (D_{i,\Pi ''}) \le
  \lim_{i \in \N} \pi ^{\ast}(D_{i,\Pi'})=
  \pi ^{\ast} D_{\Pi '} 
\]
as we wanted to show. 
  \end{proof}
  The following is a monotone approximation lemma for nef toroidal
  b-divisors. Recall that we are assuming that there is an effective
  ample divisor $A$ whose support is $B$. Since $A$ is effective and
  the support of $A$ is the whole $B$ we deduce that $\phi_A$ is
  strictly negative in each ray of $\Pi \setminus \{0\}$. Since it is
  linear on each cone of $\Pi $, then   
  the function $\phi_A|_{\mathbb{S}^{\widehat{|\Pi|}}}$ is
  strictly negative. 
  \begin{lemma}\label{lem:mon-conv} Let $\D$ be a nef toroidal
    b-divisor. Then there is a sequence of nef Cartier b-divisors
    $(\D_i )_{i \in \N}$ such that the following two
    properties are satisfied.
   \begin{enumerate}
   \item The sequence $(\D_i)_{i\in \N}$ converges to $\D$.
   \item if $i > j$, then $\D_j \geq \D_i $. 
   \end{enumerate}
  \end{lemma}
  \begin{proof}
  Let $\left\{\D_j'\right\}_{j \in J}$ be a sequence of nef Cartier
  toroidal b-divisors converging to $\D$. By Theorem \ref{thm:1} the convergence 
  \[
  \phi_{\D_i'}|_{\mathbb{S}^{\widehat{|\Pi|}}} \xrightarrow[i \to
  \infty]{} \phi_{\D}|_{\mathbb{S}^{\widehat{|\Pi|}}} 
  \]
  is uniform. 
  Let 
  \[
  \alpha \coloneqq \inf_{x \in \mathbb{S}^{\widehat{|\Pi|}}}-\phi_A(x)
  >0
  \quad , \quad
  \beta \coloneqq \sup_{x \in \mathbb{S}^{\widehat{|\Pi|}}}-\phi_A(x)
  \geq \alpha > 0,
  \]
  and for each $i \in \N$, let 
  \[
    \delta_i \coloneqq \sup_{x \in \mathbb{S}^{\widehat{|\Pi|}}}
    \left|\phi_{\D_i'}(x) -\phi_{\D}(x)\right|.
  \]
  We know that
  \[
  \delta_i \xrightarrow[i \to \infty]{} 0.
  \]
  Now, choose a subsequence $\left\{i_k\right\}_{k \in \N}$ such that 
  \[
  \delta_{i_k} \leq \frac{1}{2k(k+1)}
  \]
  and let 
  \[
  \D_k \coloneqq \D_{i_k}' + \frac{1}{\alpha k}\pmb{A}
  \]
  for $k \in \N$, where $\pmb{A}$ denotes the Cartier b-divisor
  induced by $A$. Then 
  \begin{align*}
  \left(\phi_{\D_k} -
    \phi_{\D_{k+1}}\right)\Big|_{\mathbb{S}^{\widehat{|\Pi|}}}
    &= \left( \phi_{\D_{i_k}'}- \phi_{\D_{i_{k+1}}'}+
      \left(\frac{1}{\alpha k} -
      \frac{1}{\alpha(k+1)}\right)\phi_{\pmb{A}}
      \right)\Big|_{\mathbb{S}^{\widehat{|\Pi|}}} \\
    &\leq \frac{1}{k(k+1)} - \left(\frac{1}
      {\alpha k(k+1)}\right) \alpha \\ & = 0.
\end{align*}  
 Hence, $\D_k - \D_{k+1} \geq 0$ and thus $\D_k \geq \D_{k+1}$.

Moreover, we have 
\begin{align*}
\left|\phi_{\D_k}-\phi_{\D}\right| &= \left|\phi_{\D_{i_k}'} + \frac{1}{\alpha k}\phi_{\pmb{A}} - \phi_{\D}\right| \\
&\leq \left| \phi_{\D_{i_k}'} - \phi_{\D}\right| + \frac{1}{\alpha k}\left|\phi_{\pmb{A}}\right| \\
&\leq  \frac{1}{2k(k+1)} + \frac{1}{\alpha k}\beta \xrightarrow[k\to \infty]{} 0.
\end{align*}
Hence, we get that 
\[ \D_k \xrightarrow[k\to \infty]{} \D.
   \]
   This concludes the proof.
  \end{proof}

 Now, recall that a divisor on an algebraic variety is said to be big if it
 has strictly positive volume. We define big toroidal b-divisors analogously.  
 \begin{Def}\label{def:big}
 A toroidal b-divisor $\D$ is said to be \emph{big} if it has positive
 volume, i.e. if $\vol(\D) > 0$.
 \end{Def}
 \begin{rem}
 If a toroidal b-divisor $\D = (D_{\Pi'})_{\Pi' \in R_{\sm}(\Pi)}$
 is big and nef, then by the monotonicity property of Lemma
 \ref{lem:monotonicity}, it follows that $D_{\Pi'}$ is big for all
 $\Pi' \in R_{\sm}(\Pi)$. Moreover, if $\{\D_i\}_{i \in \N}$ is a non-increasing sequence of nef Cartier b-divisors converging to $\D$, then, for all $i \in \N$, $\D_i$ has to be big as well.
 \end{rem}
  We  have the following lemma.
  \begin{lemma}\label{lem:cont}
    Let $\D$ be a nef and big b-divisor and let
    $\left\{\D_i\right\}_{i \in \N}$ be a non-increasing sequence of big and
    nef Cartier b-divisors converging to $\D$ (which exists by Lemma
    \ref{lem:mon-conv}). Then
    \begin{displaymath}
      \vol(\D)= \lim_{ i\to \infty}\vol(\D_i).
    \end{displaymath}
  \end{lemma}
  \begin{proof}
   By Theorem \ref{thm:1} and by the monotonocity assumption, we have
   that the convergence 
   \[
   \phi_{\D_i} \xrightarrow[i\to \infty]{} \phi_{\D}
   \]
   is non-decreasing and uniform on $\mathbb{S}^{\widehat{|\Pi|}}$.
   Let
   \begin{displaymath}
     \alpha \coloneqq \inf_{x \in \mathbb{S}^{\widehat{|\Pi|}}}-\phi_A(x)
     >0
   \end{displaymath}
   as in the proof of Lemma \ref{lem:mon-conv}.
   For each $i$ let 
   \[
   a_i \coloneqq  \frac{2\sup_{x \in
       \mathbb{S}^{\widehat{|\Pi|}}}\left(\phi_{\D} -
       \phi_{\D_i}\right)}{\alpha }.
   \]
   We have that 
   \begin{align}\label{eq:ai}
   a_i \xrightarrow[i \to \infty]{} 0
   \end{align} 
   and for each $i$, the sequence of inequalities
   \begin{align*}
   \phi_{\D_i} - a_i\phi_A &\geq \phi_{\D_i}+ a_i \alpha \\
   &= \phi_{\D_i} +  \frac{2\sup_{x \in
     \mathbb{S}^{\widehat{|\Pi|}}}\left( \phi_{\D} -
     \phi_{\D_i}\right)}{\alpha } \alpha  \\
   &= \phi_{\D_i} + 2\sup_{x \in \mathbb{S}^{\widehat{|\Pi|}}}\left(
     \phi_{\D} - \phi_{\D_i}\right) \geq \phi_{\D}
\end{align*}
is satisfied. Hence, we get that 
\begin{align}\label{eq:vol-lim}
\vol(\D_i) \geq \vol(\D) \geq \vol(\D_i -a_iA)
\end{align}
for each $i$.
Now, set $\omega \coloneqq \D_0 + A$. Then $\omega$ is a big and nef Cartier b-divisor. Moreover, by the monotonicity of the sequence, we have that $\omega \geq \D_i$ for all $i$. Thus, by \cite[Corollary~3.4]{BFJ:diffvol}, we obtain 
\begin{align}\label{eq:vol-lim2}
\vol(\D_i - a_iA) \geq \D_i^n - na_i (\D_i^{n-1}A)-Ca_i^2,
\end{align}
where $C$ is a constant depending only on $\omega$. Since the $\D_i$'s are nef Cartier b-divisors, we know that $\vol(\D_i) = \D_i^n$. Therefore, taking limits in \eqref{eq:vol-lim} and \eqref{eq:vol-lim2}, we get that 
\[
\liminf_i\vol(\D_i) \geq \vol(\D) \geq  \limsup_i \left(\vol(\D_i) - na_i (\D_i^{n-1}A)-Ca_i^2 \right).
\]
Now, since $\left(\D_i^{n-1}A\right)$ is bounded, using \eqref{eq:ai}, we have that $\lim_i \left(na_i (\D_i^{n-1}A)-Ca_i^2\right) = 0$. We conclude that 
\[
\liminf_i\vol(\D_i) \geq \vol(\D) \geq \limsup_i\vol(\D_i),
\]
as we wanted to show. 
\end{proof}
As a consequence, we get the following Hilbert--Samuel type formula. 
\begin{theorem}\label{the:hilbertsamuel}
Let $\D$ be a big and nef toroidal b-divisor. Then 
\[
\vol(\D) = n!\vol(\Delta_{\D}) = \D^n.
\]
\end{theorem}
\begin{proof}
Write $\D = \lim_{i \to \infty} \D_i$ as a non-increasing limit of Cartier
big and nef toroidal b-divisors. Then, by Lemma \ref{lem:cont}, it
follows that 
\begin{align}\label{eqn:1} 
  \vol(\D) =  \lim_{i \to\infty}\vol\left(\D_i\right) =
  \lim_{i \to \infty}\D_i^n = \D^n.
\end{align}
On the other hand, consider the semigroup $S(\D)$ used in the construction of the Okounkov body $\Delta_{\D}$ (see discussion preceeding Definition \ref{def:okounkov}) and set 
\[
S(\D)_{\ell} = S(\D) \cap \left(\Z^n \times \{\ell\}\right)
\]
for any integer number $\ell > 0$. By \cite[Proposition~2.1]{LM}, we have that 
\[
\lim_{\ell \to \infty} \frac{\# S(\D)_{\ell}}{\ell^n/n!} = \vol_{\R^n}\left(\Delta_{\D}\right).
\]
This implies that
\[
\vol(\D) = \limsup_{\ell}\frac{h^0(X, \ell \D)}{\ell/n!} = \limsup_{\ell}\frac{\#S(\D)_{\ell}}{\ell/n!}  = \lim_{\ell\to \infty}\frac{\#S(\D)_{\ell}}{\ell/n!} = \vol_{\R^n}\left(\Delta_{\D}\right),
\]
as we wanted to show.
\end{proof}

\begin{rem}
  As we will show in a forthcoming paper with D. Holmes and R. de
  Jong Theorem \ref{the:hilbertsamuel} is not true if we drop the condition of being
  toroidal. 
\end{rem}
We obtain the following two corollaries. 
\begin{cor}\label{cor:cont}
  The function $\vol$ is continuous on the space of big and nef toroidal b-divisors. 
\end{cor}
\begin{proof}
  By Theorem \ref{the:hilbertsamuel} the
  volume function agrees with the degree function on the space of big
  and nef toroidal b-divisors. By Theorem \ref{thm:intersectionbdivisors},
  it is continuous.
\end{proof}
The following is a Brunn--Minkowski type inequality.
\begin{cor}\label{cor:brunnminkowski}
Let $\D$ and $\F$ be two nef and big toroidal b-divisors. Then the following Brunn--Minkowski type inequality holds true. 
\[
(\D^n)^{1/n} + (\F^n)^{1/n} \geq ((\D + \F)^n)^{1/n}.
\]
\end{cor}
\begin{proof}
Consider the associated Okounkov bodies $\Delta_{\D}$ and $\Delta_{\F}$, respectively. Then, using Theorem \ref{the:hilbertsamuel}, the inequality follows from a standard result in convex geometry about volumes of convex
sets (see e.g. \cite{BM}).
\end{proof}

\printbibliography

\end{document}